\documentclass[12pt]{amsart}

\usepackage{amssymb,latexsym}

\usepackage{graphicx,epsfig}
\usepackage[dvips]{color}

\def\cal{\mathcal}

\def\Bbb{\mathbb}

\def\Re{\text{\rm Re\,}}
\def\rank{\text{\rm rank\,}}

\def\sgn{\text{\rm sgn\,}}
\def\ord{\text{\rm ord\,}}

\def\pad{\phi^a}

\def\ra{\rangle}
\def\laa{\langle}

\def \supp {\text{\rm supp\,}}

\def\nn{\nonumber}
\def \sgn {\text{\rm sgn\,}}
\def\trans{{\,}^t}

\def\hl{{h_{\rm \,lin}}}
\def\A{{\cal A}}
\def\D{{\cal D}}

\def\N{{\cal N}}
\def\S{{\cal S}}
\def\T{{\cal T}}
\def\bC{{\Bbb C}}
\def\CC{{\Bbb C}}

\def\NN{{\Bbb N}}
\def\bN{{\Bbb N}}
\def\bR{{\Bbb R}}
\def\RR{{\Bbb R}}
\def\bZ{{\Bbb Z}}
\def\ZZ{{\Bbb Z}}

\def\QQ{{\Bbb Q}}
\def\vp{{\varphi}}
\def\al{{\alpha}}
\def\be{{\beta}}
\def\ga{{\gamma}}
\def\la{{\lambda}}
\def\om{{\omega}}

\def\ze{{\zeta}}
\def\x{(x_1,x_2)}
\def\y{(y_1,y_2)}
\def\pa{{\partial}}

\def\Hess{{\rm Hess}}
\def\ve{{\varepsilon}}
\def\si{{\sigma}}

\def\de{{\delta}}
\def\dee{{\delta^e}}

\def\Om{{\Omega}}
\def\ka{{\kappa}}

\def\th{{\theta}}

\def\pr{\text{\rm pr\,}}

\def\bpm{\begin{pmatrix}}
\def\epm{\end{pmatrix}}

\def\noi{\noindent}
\def\bee{\begin{enumerate}}
\def\ee{\end{enumerate}}

\def\pf{\proof}

\def\qed{\smallskip\hfill Q.E.D.\medskip}

\newtheorem{thm}{Theorem}[section]
\newtheorem{prop}[thm]{Proposition}
\newtheorem{proposition}[thm]{Proposition}
\newtheorem{cor}[thm]{Corollary}
\newtheorem{lemma}[thm]{Lemma}
\newtheorem{remark}[thm]{Remark}
\newtheorem{remarks}[thm]{Remarks}

\newtheorem{assumption}[thm]{Assumption}
\newtheorem{example}[thm]{Example}

\begin{document}

\title[A sharp restriction theorem]{
$L^p$-$L^2$  Fourier restriction  for  hypersurfaces in  $\bR^3:$ Part I}

\author[I. A. Ikromov]{Isroil A. Ikromov}
\address{Department of Mathematics, Samarkand State University,
University Boulevard 15, 140104, Samarkand, Uzbekistan}
\email{{\tt ikromov1@rambler.ru}}

\author[D. M\"uller]{Detlef M\"uller}
\address{Mathematisches Seminar, C.A.-Universit\"at Kiel,
Ludewig-Meyn-Stra\ss{}e 4, D-24098 Kiel, Germany}
\email{{\tt mueller@math.uni-kiel.de}}
\urladdr{{http://analysis.math.uni-kiel.de/mueller/}}

\thanks{We gratefully acknowledge the support for this work by the Deutsche Forschungsgemeinschaft.}

\begin{abstract}
This is the first of two  articles,  in which we prove a sharp  $L^p$-$L^2$ Fourier
restriction theorem for  a large class of smooth, finite type hypersurfaces in
$\RR^3,$ which includes in particular all  real-analytic hypersurfaces.
\end{abstract}

\maketitle

\maketitle


\tableofcontents

\thispagestyle{empty}

\setcounter{equation}{0}
\section{Introduction}\label{intro}

Let $S$ be a smooth, finite type  hypersurface in $\RR^3$ with Riemannian 
surface measure $d\si,$ and consider the compactly supported measure 
$d\mu:=\rho d\si$ on $S,$ where $0\le\rho\in C_0^\infty(S).$
The goal of this article is  to determine the sharp range of exponents $p$ for which a  Fourier restriction estimate 
\begin{equation}\label{rest1}
\Big(\int_S |\widehat f|^2\,d\mu\Big)^{1/2}\le C_p\|f\|_{L^p(\RR^3)},\qquad  f\in\S(\RR^3),
\end{equation}
holds true. To this end, we may localize to a sufficiently small neighborhoods of a given point $x^0$ on $S.$  Observe also that if estimate \eqref{rest1} holds for the hypersurface $S$, then it is valid also for every affine-linear image of $S,$ possibly with a different constant if the Jacobian of this map is not one.   By applying a suitable Euclidean motion of   $\RR^3$ we may then assume that $x^0=(0,0,0),$ and  that $S$ is the graph 
$$
 S=\{(x_1,x_2, \phi\x): \x\in \Om \},
$$
of a smooth function $\phi$ defined on a  sufficiently small neighborhood $\Om$ of the origin, such that 
$\phi(0,0)=0,\, \nabla \phi(0,0)=0.$

In our preceding article \cite{IM-uniform},  this problem had been solved, in terms of  Newton diagrams associated to $\phi,$  under the assumption that there exists a linear coordinate system which is adapted to the function $\phi,$ in the sense of Varchenko. More precisely, if denote by $h(\phi)$ the  height of $\phi,$ in the sense of Varchenko, then 
we had proved the following result:
\begin{thm}\label{adarestrict}
Assume that, after applying a suitable linear change of coordinates, the coordinates $\x$ are adapted to $\phi.$ 
We then define the critical exponent $p_c$ by
\begin{equation}\label{pcritical}
p'_c:=2h(\phi) +2,
\end{equation}
where $p'$ denotes the exponent conjugate to $p,$ i.e., $1/p+1/p'=1.$

Then there exists a neighborhood $U\subset S $ of the point $x^0$ such
that for every  non-negative density $\rho\in C_0^\infty(U)$ the  Fourier restriction estimate \eqref{rest1}
holds true for every $p$ such that
\begin{equation}\label{rest2}
1\le p\le p_c.
\end{equation}

Moreover, if $\rho(x^0)\ne 0,$ then the condition \eqref{rest2} on $p$ is also necessary for the validity of  \eqref{rest1}.
\end{thm}

Earlier results for particular classes of hypersurfaces in $\RR^3$ are for instance in the work by E. Ferreyra and M. Urciuolo \cite{ferreyra-urciuolo1}, \cite{ferreyra-urciuolo2} and \cite{ferreyra-urciuolo3}, who studied particular  classes of quasi-homogeneous hypersuraces, for which they were able to  prove $L^p$-$L^q$- restriction estimates when $p<4/3,$ which in some cases are  sharp, except possibly for the endpoint, and    $L^p$-$L^2$  restriction estimates for  general analytic hypersurfaces  in   A. ~Magyar's article  \cite{magyar}.  For particular classes of hypersurfaces given as graphs of functions in adapted coordinates, his  results were sharp, with the exception of the  endpoint.

\medskip
In view of Theorem \ref{adarestrict}, we shall  from now on  always make the following 
\begin{assumption}\label{assumption}
There is no linear coordinate system which is adapted to $\phi.$
\end{assumption}

\subsection{Basic notions, and the case of analytic hypersurfaces}

In order to formulate our main result, we need more notation. 
 We shall build on the results and technics developed in
\cite{IM-ada} and \cite{IKM-max}, which will be our main references,
also for references to earlier and related work. Let us first recall
some basic notions from \cite{IM-ada}, which essentially go back to Arnol'd (cf. \cite{arnold}, \cite{agv}) and his school, most notably Varchenko \cite{Va}. 

If $\phi$ is given as before, consider the associated Taylor series
$$\phi(x_1,x_2)\sim\sum_{\al_1,\al_2=0}^\infty c_{\al_1,\al_2} x_1^{\al_1} x_2^{\al_2}$$
of $\phi$ centered at  the origin.
The set
$$\T(\phi):=\{(\al_1,\al_2)\in\bN^2: c_{\al_1,\al_2}=\frac 1{\al_1!\al_2!}\partial_{ 1}^{\al_1}\partial_{ 2}^{\al_2}­ \phi(0,0)\ne 0\}
$$
will be called the {\it Taylor support } of $\phi$ at $(0,0).$  We shall always assume that
$$\T(\phi)\ne \emptyset,$$
i.e., that the function $\phi$ is of finite type at the origin. The
{\it Newton polyhedron} $\N(\phi)$ of $\phi$ at the origin is
defined to be the convex hull of the union of all the quadrants
$(\al_1,\al_2)+\bR^2_+$ in $\bR^2,$ with $(\al_1,\al_2)\in\T(\phi).$  The associated
{\it Newton diagram}  $\N_d(\phi)$ in the sense of Varchenko
\cite{Va}  is the union of all compact faces  of the Newton
polyhedron; here, by a {\it face,} we shall  mean an edge or a
vertex.

We shall use coordinates $(t_1,t_2)$ for points in the plane
containing the Newton polyhedron, in order to distinguish this plane
from the $(x_1,x_2)$ - plane.

The {\it Newton distance}, or shorter {\it distance} $d=d(\phi)$
between the Newton polyhedron and the origin in the sense of
Varchenko is given by the coordinate $d$ of the point $(d,d)$ at
which the bi-sectrix   $t_1=t_2$ intersects the boundary of the
Newton polyhedron.

The {\it principal face} $\pi(\phi)$  of the Newton polyhedron of
$\phi$  is the face of minimal dimension  containing the point
$(d,d)$. Deviating from the notation in \cite{Va}, we shall call the
series
$$\phi_\pr(x_1,x_2):=\sum_{(\al_1,\al_2)\in \pi(\phi)}c_{\al_1,\al_2} x_1^{\al_1} x_2^{\al_2}
$$
the {\it principal part} of $\phi.$ In case that $\pi(\phi)$ is
compact,  $\phi_\pr$ is a mixed homogeneous polynomial; otherwise,
we shall  consider $\phi_\pr$ as a formal power series.

Note that the distance between the Newton polyhedron and the origin
depends on the chosen local coordinate system in which $\phi$ is
expressed.  By a  {\it local  coordinate system at the origin} we
shall mean a smooth   coordinate system defined near the origin
which preserves $0.$ The {\it height } of the  smooth function
$\phi$ is defined by
$$h(\phi):=\sup\{d_y\},$$
 where the
supremum  is taken over all local  coordinate systems $y=(y_1,y_2)$ at the origin, and where $d_y$
is the distance between the Newton polyhedron and the origin in the
coordinates  $y$.

A given   coordinate system $x$ is said to be
 {\it adapted} to $\phi$ if $h(\phi)=d_x.$

In \cite{IM-ada} we proved that one  can always find an adapted
local  coordinate system in two dimensions, thus  generalizing  the
fundamental work by Varchenko  \cite{Va} who worked in the   setting
of  real-analytic functions $\phi$ (see also \cite{PSS}).

Recall also that if the principal face of the Newton polyhedron $\N(\phi)$ is a compact edge, then it lies on a unique ``principal line'' 
$$
 L:=\{(t_1,t_2)\in \RR^2:\ka_1t_1+\ka_2 t_2=1\}, 
$$
with $\ka_1,\ka_2>0.$ By permuting the  coordinates $x_1$ and $x_2,$ if necessary, we shall always assume that $\ka_1\le\ka_2.$ The  weight $\ka=(\ka_1,\ka_2)$ will be called the {\it principal weight} associated to $\phi.$  It induces dilations $\de_r\x:=(r^{\ka_1}x_1,r^{\ka_2} x_2),\ r>0,$ on $\RR^2,$ so that the principal part $\phi_\pr$ of $\phi$ is $\ka$- homogeneous of degree one with respect to these dilations, i.e.,  $\phi_\pr(\de_r\x)=r\phi_\pr\x$ for every $r>0,$ and 
\begin{equation}\label{dnew}
d=\frac 1{\ka_1+\ka_2}=\frac1{|\ka|}.
\end{equation}

More generally,  if $\ka=(\ka_1,\ka_2)$ is any weight with $0<\ka_1\le \ka_2$ such that  the line  $L_\ka:=\{(t_1,t_2)\in\bR^2:\ka_1 t_1+\ka_2t_2=1\}$ is a supporting  line to the Newton polyhedron $\N(\phi) $ of $\phi,$ then  the {\it $\ka$-principal part} of $\phi$
$$
\phi_\ka(x_1,x_2):=\sum_{(\al_1,\al_2)\in L_\ka} c_{\al_1,\al_2} x_1^{\al_1} x_2^{\al_2}
$$
 is a non-trivial polynomial  which is $\ka$-homogeneous of degree $1$ with respect to the dilations associated to this weight as before. By definition, we then have
$$
\phi(x_1,x_2)=\phi_\ka(x_1,x_2) +\ \mbox{terms of higher $\ka$-degree}
$$

Adaptedness of a given coordinate system can be verified by means of the following criterion (see  \cite{IM-ada}): Denote  by

$$m(\phi_\pr):=\ord_{S^1} \phi_\pr$$
the maximal order of vanishing of $\phi_\pr$ along the unit circle $S^1$ centered at the
origin. The {\it homogeneous distance} of a $\ka$-homogeneous polynomial $P$ (such as $P=\phi_\pr$) is given by $
d_h(P):= 1/{(\ka_1+\ka_2)}=1/|\ka|.$ Notice that $(d_h(P),d_h(P))$ is just the point of intersection of the line given by $\ka_1t_1+\ka_2t_2=1$ with the bi-sectrix $t_1=t_2.$
The height of $P$ can the be computed by means of the formula
\begin{equation}\label{heightp}
h(P)=\max\{m(P), d_h(P)\}.
\end{equation}

 According to  \cite{IM-ada}, Corollary 4.3 and  Corollary 2.3, {\it the coordinates $x$ are adapted to $\phi$ if and only if one of the following conditions is satisfied:
\medskip

\bee
\item[(a)]  The principal face  $\pi(\phi)$ of the Newton polyhedron  is a compact edge, and $m(\phi_\pr)\le d(\phi).$
\item[(b)] $\pi(\phi)$  is a vertex.
\item[(c)] $\pi(\phi)$ is an unbounded edge.
\ee}
\medskip

We like to mention that in case (a) we have $h(\phi)=h(\phi_\pr)=d_h(\phi_\pr).$  Notice also that (a) applies whenever $\pi(\phi)$ is a compact edge  and $\ka_2/\ka_1\notin\NN;$   in this case we even have  $m(\phi_\pr)< d(\phi)$ (cf. \cite{IM-ada}, Corollary 2.3).

\medskip
 
In the case where  the coordinates $\x$ are not adapted to $\phi,$ we see that the principal face $\pi(\phi)$ is a  compact edge lying on a unique  line 
 $$
 L=\{(t_1,t_2)\in \RR^2:\ka_1t_1+\ka_2 t_2=1\}, 
$$
and  that $m:=\ka_2/\ka_1\in\NN.$
 Now, if $\ka_2/\ka_1=1,$ then a linear change of coordinates of the form $y_1=x_1, y_2=x_2-b_1x_1$ will transform $\phi$ into a function $\tilde\phi$ for which, by our assumption, the coordinates $(y_1,y_2)$ are still not adapted (cf. \cite{IM-ada}). Replacing $\phi$ by $\tilde\phi,$ it is also immediate that estimate \eqref{rest1} will hold for the graph of $\phi$ if and only if it holds for the graph of $\tilde\phi.$ Replacing $\phi$ by $\tilde\phi,$ we may and shall therefore always assume that  our original coordinate system $(x_1,x_2)$ is chosen so that 
\begin{equation}\label{m1}
m=\ka_2/\ka_1\in\NN\qquad\mbox{and} \ m\ge 2.
\end{equation}
Such a linear coordinate system will be called {\it linearly adapted} to $\phi$ (see Section \ref{normalforms} for a more comprehensive discussion of this notion).

Then, by Theorem 5.1 in \cite{IM-ada}, there exists a smooth real-valued function $\psi$  (which we may choose as the so-called principal root jet of $\phi$) of the form
\begin{equation}\label{prjet}
\psi(x_1)=cx_1^{m}+O(x_1^{m+1})
\end{equation}
 with $c\ne 0$ defined on a neighborhood of the origin  such that an adapted  coordinate system $(y_1,y_2)$ for $\phi$ is given locally near the origin by means of the (in general non-linear) shear
\begin{equation}\label{adaptco}
y_1:= x_1, \ y_2:= x_2-\psi(x_1).
\end{equation}

  In these coordinates, $\phi$ is given by
\begin{equation}\label{phia1}
 \phi^a(y):=\phi(y_1,y_2+\psi(y_1)).
\end{equation}

We remark that such an adapted coordinate system can be constructed by means of an algorithm which goes back  Varchenko \cite {Va} in the case of real-analytic  $\phi$ (see  \cite{IM-ada}). 
\medskip


\begin{figure}
\centering
\vskip-5cm
\scalebox{0.35}{\input{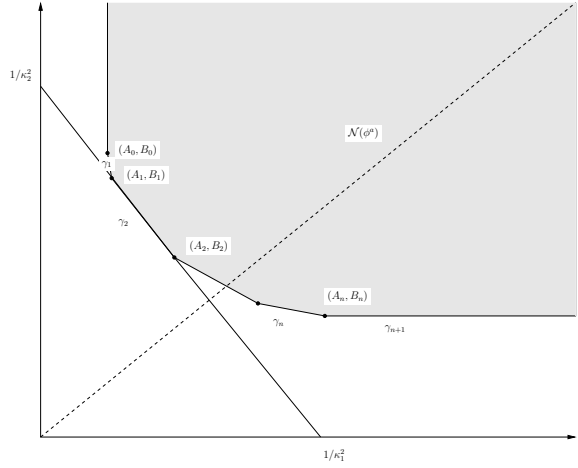}}
\caption{Edges and weights}{\label{fig2}}
\end{figure}

Let us then denote the vertices of the Newton polyhedron $\N(\pad)$ by 
 $(A_l,B_l), \  l=0,\dots,n,$ where  we assume that they are ordered so that $A_{l-1}< A_{l},\ l=1,\dots,n,$ with associated compact edges given by the intervals
  $\ga_l:=[(A_{l-1},B_{l-1}), (A_l,B_l)], l=1,\dots,n.$ The unbounded horizontal edge with left endpoint $(A_n,B_n)$ will be denoted by  $\ga_{n+1}.$ To each of these edges $\ga_l,$ we associate the weight
$\ka^l=(\ka^l_1,\ka^l_2),$ so that $\ga_l$ is  contained in  the line
$$L_l:=\{(t_1,t_2)\in \bR^2:\ka^l_1t_1+\ka^l_2 t_2=1\}.$$ 
For $l=n+1,$ we have $\ka_1^{n+1}:=0,\ka_2^{n+1}=1/B_{n}.$ We denote by 
$$a_l:=\frac {\ka^l_2}{\ka^l_1}, \quad l=1,\dots, n$$
the  reciprocal of the slope of the line $L_l.$ For $l=n+1,$ we formally set $a_{n+1}:=\infty$.

If $l\le n,$  the $\ka^l$-principal part $\phi_{\ka^l }$ of $\phi$ corresponding to the supporting line $L_l$  is of the form
\begin{equation}\label{normalhom}
\phi_{\ka^l }(x)=c_l\, x_1^{A_{l-1}} x_2^{B_l}\prod_\al \Big(x_2-c^\al_l x_1^{a_{l}}\Big)^{N_\al}
\end{equation}
(cf. \cite{IKM-max}).
In view of this identity, we shall say that the edge $\ga_l:=[(A_{l-1},B_{l-1}) ,(A_l,B_l)]$ is
associated to the cluster of roots $[l].$ 

Consider the line parallel to the bi-sectrix
$$
\Delta^{(m)}:=\{(t,t+m+1):t\in\RR\}.
$$
For any edge $\ga_l\subset L_l:=\{(t_1,t_2)\in \bR^2:\ka^l_1t_1+\ka^l_2 t_2=1\}$  define $h_l$ by
$$
\Delta^{(m)}\cap L_l=\{(h_l-m, h_l+1)\},
$$ i.e.,
\begin{equation}\label{hl}
h_l=\frac {1+m\ka^l_1-\ka^l_2}{\ka^l_1+\ka^l_2},
\end{equation}
and define the {\it restriction height}, or short, {\it $r$-height,} of $\phi$ 
\color{black} by 
$$
h^r(\phi):= \max(d, \max\limits_{\{l=1,\dots, n+1:a_l>m\}} h_l).
$$

\begin{remarks}\label{r1}
 \begin{itemize}
\item[(a)] For $L$ in place of $L_l$ and $\ka$ in place of $\ka^l,$  one has $m=\ka_2/\ka_1$ and $d=1/(\ka_1+\ka_2),$ so that one gets $d$ in place of $h_l$ in \eqref{hl}
\item[(b)]  Since $m< a_l,$ we have $h_l<1/(\ka^l_1+\ka^l_2),$ hence $h^r(\phi)<h(\phi).$ On the other hand, since the line $\Delta^{(m)}$ lies above the bi-sectrix, it is obvious that $h^r(\phi)+1\ge h(\phi),$ so that 
\begin{equation}\label{hrh}
h(\phi)-1\le h^r(\phi)<h(\phi).
\end{equation}
\end{itemize}
\end{remarks}

It is easy to see by Remark \ref{r1} (a)  that the  $r$-height admits the following {\it geometric interpretation:}
\medskip

By following Varchenko's algorithm (cf. Subsection 8.2 of \cite{IKM-max}), one realizes that  the Newton polyhedron of $\pad$ intersects the line $L$ of the Newton polyhedron of $\phi$ in a compact face,  either in a single vertex, or a compact edge. I.e., the intersection contains at least one and at most two vertices of $\pad,$ and we choose  $(A_{l_0-1},B_{l_0-1})$ as the one with smallest second coordinate.  Then $l_0$ is the smallest index $l$ such that $\ga_l$ has a slope smaller than the slope of $L,$  i.e., 
 $
a_{l_0-1}\le m<a_{l_0}
$
We may thus consider  the ``augmented'' Newton polyhedron $\N^r(\pad)$ of $\pad,$ which is the convex hull of the union of $\N(\pad)$ with  the half-line $L^+\subset L$ with right endpoint $(A_{l_0-1},B_{l_0-1}).$ Then $h^r(\phi)+1$ is the second coordinate of the  point at which the line $\Delta^{(m)}$ intersects the boundary of $\N^r(\pad).$

\begin{figure}
\centering
\scalebox{0.35}{\input{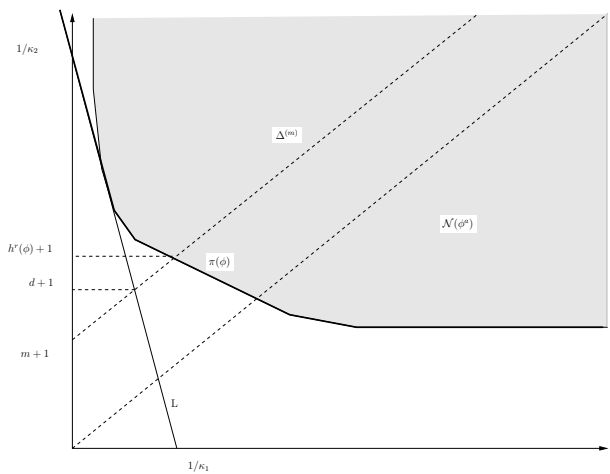}}
\caption{r-height}{\label{fig1}}
\end{figure}

 \begin{thm}\label{nonadarestrictanal}
Let $\phi\ne 0$ be real analytic, and assume that there is no linear  coordinate system  adapted to $\phi.$  Then there exists a neighborhood $U\subset S $ of  $x^0=0$ such
that for every  non-negative density $\rho\in C_0^\infty(U),$ the Fourier restriction estimate \eqref{rest1} holds true for every $p\ge 1$ such that 
$p'\ge p'_c:=2h^r(\phi)+2.$
\end{thm}

\begin{remarks}\label{r2}
 \begin{itemize}
\item[(a)] An application of Greenleaf's result would  imply, at best,  that the condition  $p'\ge  2h(\phi)+2$ is sufficient for \eqref{rest1} to hold, which is a strictly stronger condition than  $p'\ge  p_c'.$  
\item[(b)] A. Seeger recently informed us  that in a preprint, which  regretfully had remained unpublished,   Schulz \cite{schulz-un} had already observed this  kind of phenomenon for particular examples  of surfaces of revolution.
\item[(c)] It can be shown that the number $m$ is well-defined, i.e., it does not depend on the chosen linearly adapted  coordinate system $x$ (cf. Proposition \ref{linada}).
\color{black}
 \end{itemize}
\end{remarks}

\color{black}

\begin{example}\label{ex1}
{\rm $$\phi\x:=(x_2-x_1^{m})^n, \qquad n,m\ge 2.$$
The coordinates $\x$ are not adapted. Adapted coordinates are 
$y_1:=x_1, y_2:=x_2-x_1^m,$ in which $\phi$ is given by
 $$\phi^a(y_1,y_2)=y_2^n.$$
Here 
 \begin{eqnarray*}
&&\ka_1=\frac 1{mn}, \quad \ka_2=\frac 1n,\\
&&d:=d(\phi)= \frac 1{\ka_1+\ka_2}=\frac {nm}{m+1},
\end{eqnarray*}
and 
$$ p'_c=\left\{  \begin{array}{cc}
2d+2, & \mbox{   if } n\le m+1 ,\hfill\\ 
2n, & \mbox{   if } n>m+1\ .\hfill 
\end{array}\right.
$$
On the other hand, $h:=h(\phi)=n,$ so that $2h+2=2n+2>p'_c.$ }
\end{example}

\smallskip

\subsection{Finite type hypersurfaces, condition (R), and an invariant  description of the notion of $r$-height. }

An analogous theorem holds true even for smooth, finite type functions $\phi,$ under an  additional condition which, however, is always satisfied when $\phi$ is real-analytic. To state this more general result, and in order to prepare a more invariant description of the notion of $r$-height, we need to introduce  more notation. Again, we shall assume that the coordinates $\x$ are linearly adapted to $\phi.$
\medskip

{\bf Definitions.} 
Denote by $\RR_{\pm}:=\{x_1\in\RR: \pm x_1>0\}$ and by $H^\pm:=\RR_{\pm}\times \RR$ the corresponding right,  respectively left half-plane.  

We say that a function $f=f(x_1)$ defined in $U\cap\RR_+$  (respectively $U\cap \RR_-$), where $U$ is an open neighborhood of the origin, is  {\it fractionally smooth}, if there exist  a smooth function $g$ on $U$ and a positive integer $q$ such that $f(x_1)=g(|x_1|^{1/q})$ for $x_1\in U\cap \RR_+$ (respectively $x_1\in U\cap\RR_-$). Moreover, we shall  say that a fractionally smooth function $f$  is {\it flat}, if $f(x_1)=O(|x_1|^N) $ for every $N\in\NN.$ Two smooth functions $f$ and $g$ defined on a neighborhood of the origin will be called {\it equivalent},  and we shall write $f\sim g,$ if $f-g$ is flat.  Finally, a {\it fractional shear} in $H^\pm$ will be a change of coordinates of the form
$$
y_1:=x_1, \ y_2:=x_2-f(x_1),
$$
where $f$ is  real-valued and fractionally smooth, but not flat.  If we express the smooth function $\phi$ on, say, the half-plane $H^+,$ as a function of $y=\y,$ the resulting function 
$$\phi^f(y):=\phi(y_1,y_2+f(y_1))$$
 will in general  no longer be smooth at the origin, but ``fractionally smooth''.

For such functions, there are straight-forward generalizations of the notions of Newton-polyhedron, etc.. Namely, following
 \cite{IKM-max}, and assuming without loss of generality that we are in $H^+$ where $x_1>0,$ 
let  $\phi$ be a  function of the variables $x_1^{1/q}$ and $x_2$ near the origin, i.e., there exists a smooth function $\phi^{[q]}$ near the origin such that $\phi(x)=\phi^{[q]}(x_1^{1/q},x_2)$ (more generally, we could  assume that $\phi$ is a smooth function of the variables $x_1^{1/q}$ and $x_2^{1/p},$ where $p$ and $q$ are positive integers, but we won't need this generality here). Such functions $\phi$ will also be called {\it fractionally smooth.}  If the Taylor series of $\phi^{[q]}$ is given by 
$$
\phi^{[q]}\x\sim \sum_{\al_1,\al_2=0}^\infty c_{\al_1,\al_2}x_1^{\al_1}x_2^{\al_2},
$$
then $\phi$ has the formal Puiseux series expansion
$$
\phi\x\sim \sum_{\al_1,\al_2=0}^\infty c_{\al_1,\al_2}x_1^{ {\al_1}/q}x_2^{\al_2}.
$$
We therefore define the {\it Taylor-Puiseux support}, or shorter,  {\it Taylor-support} of $\phi$ by
$$
\T(\phi):=\{(\tfrac {\al_1}q,\al_2)\in \NN_q^2: c_{\al_1,\al_2}\ne 0\},
$$
where $\NN_q^2:=(\tfrac 1q \NN)\times\NN.$
The {\it Newton-Puiseux  polyhedron} (shorter: {\it Newton polyhedron})  $\N(\phi)$ of $\phi$ at the origin is then defined to be the convex hull of the union of all the quadrants $(\al_1/q,\al_2)+\bR^2_+$ in $\bR^2,$ with $(\al_1/q,\al_2)\in\T(\phi),$ and  other notions, such as the notion of principal face, Newton distance or homogenous distance, are defined in analogy with our previous definitions  for smooth functions $\phi.$ 

Now, if $f(x_1)$ has the formal Puiseux series expansion (say for $x_1>0$)
$$
f(x_1)\sim \sum_{j\ge 0} c_{j} x_1^{m_j},
$$
with non-zero coefficients $c_j$  and exponents $m_j$ which are growing with $j$ and are all multiples of $1/q,$ we isolate the leading exponent $m_0$ and choose the weight $\ka^f$ so that $\ka^f_2/\ka^f_1=m_0$ and such  that the line 
$$L^f:=\{(t_1,t_2)\in \bR^2:\ka^f_1t_1+\ka^f_2 t_2=1\}$$
is a supporting line to $\N(\phi^f).$ We can then define the augmented Newton polyhedron $\N^r(\phi^f)$ in the same way as we defined $\N^r(\pad),$ replacing the exponent $m$ by $m_0$ and the line $L$ by $L^f,$  and define, in analogy with $h^r(\phi),$ the $r$-height  $h^f(\phi)$ associated to $f$ by requiring that $h^f(\phi)+1$ is the second coordinate of the  point at which the line $\Delta^{(m_0)}$ intersects the boundary of $\N^r(\phi^f).$ Again, it is easy to see that 
\begin{equation}\label{hf}
h^f(\phi)= \max(d^f, \max\limits_{\{l:a_l>m_0\}} h^f_l),
\end{equation}
where $(d^f,d^f)$ is the point of intersection of the line $L^f$ with the bi-sectrix,  and where $h^f_l$ is associated to the edge $\ga_l$ of $\N(\phi^f)$ by the analogue of formula \eqref{hl}, i.e., 
\begin{equation}\label{hl2}
h^f_l=\frac {1+m_0\ka^l_1-\ka^l_2}{\ka^l_1+\ka^l_2},
\end{equation}
if $\ga_l$ is again contained in the line $L_l$ defined by the weight $\ka^l.$

\smallskip
Finally, let us say that a fractionally smooth function $f(x_1)$ {\it agrees with the principal root jet $\psi(x_1)$ up to terms of higher order,} if the following holds: if $\psi$ is not   a polynomial,  then $f\sim \psi,$ and if $\psi$ is  polynomial of degree $D,$ then the leading exponent in the formal Puiseux expansion of $f-\psi$ is strictly bigger than $D.$ 
\medskip

We can now formulate the condition that we need when $\phi$ is non-analytic.
\medskip

{\bf Condition (R).}  For every  fractionally smooth, real  function  $f(x_1)$  which  agrees with the principal root jet $\psi(x_1)$ up to terms of higher order, the following holds true: 

If $B\in\NN$ is maximal such that  $\N(\phi^f)\subset \{(t_1,t_2): t_2\ge B\},$ then $\phi$ factors as 
$\phi\x=(x_2-\tilde f (x_1))^B\tilde\phi\x,$  where $\tilde f\sim  f$ and where $\tilde\phi$ is fractionally smooth.

\medskip
Clearly, Condition (R) is satisfied if $\phi$ is real-analytic. 

\begin{thm}\label{nonadarestrict}
Let $\phi$ be smooth and  of finite type, and assume that the coordinates $\x$ are linearly adapted to $\phi,$ but not adapted, and   that  Condition (R) is satisfied.

Then there exists a neighborhood $U\subset S $ of  $x^0=0$ such
that for every  non-negative density $\rho\in C_0^\infty(U),$ the Fourier restriction estimate \eqref{rest1} holds true for every $p\ge 1$ such that 
$p'\ge p'_c:=2h^r(\phi)+2.$
\end{thm}

This theorem is sharp in the following sense:
\begin{thm}\label{nec}
Let $\phi$ be smooth of finite type, and assume that the Fourier restriction estimate \eqref{rest1}  holds true in a neighborhood of $x^0.$ Then, if $\rho(x^0)\ne 0,$ necessarily $p'\ge p_c'.$
\end{thm}

Finally, we can also give a more invariant description of the notion of $r$-height, which  conceptually resembles more closely   Varchenko's definition of the notion of height, only that we restrict the admissible changes of coordinates to the class of fractional shears in the half-planes $H^+$ and $H^-.$  Assume again that the coordinates $\x$ are linearly adapted to $\phi,$ and let
\begin{equation}\label{clasheight}
\tilde h^r(\phi):=\sup_f h^f(\phi),
\end{equation}
where the supremum is taken over all non-flat fractionally smooth, real  functions  $f(x_1)$ of $x_1>0$ (corresponding to a fractional shear in $H^+$) or of $x_1<0$ (corresponding to a fractional shear in $H^-$). Then obviously
\begin{equation}\label{hrs}
h^r(\phi)\le \tilde h^r(\phi),
\end{equation}
but in fact there is equality:

\begin{proposition}\label{rheight}
 Assume that the coordinates $\x$ are linearly adapted to $\phi,$ where $\phi$ is smooth and of finite type and satisfies $\phi(0,0)=0,\, \nabla \phi(0,0)=0.$
  \begin{itemize}
\item[(a)] If the coordinates $\x$ are not adapted to $\phi,$ then for every non-flat fractionally smooth, real  function  $f(x_1)$ and the corresponding fractional shear in $H^+$ respectively  $H^-,$ we have $h^f(\phi)\le h^r(\phi).$  Consequently, $h^r(\phi)= \tilde h^r(\phi).$ 

\item[(b)] If the coordinates $\x$ are adapted to $\phi,$ then $ \tilde h^r(\phi)=d(\phi)=h(\phi).$ 
 \end{itemize}
In particular, the critical exponent for the restriction estimate \eqref{rest1} is in all cases  given by $p_c':=2 \tilde h^r(\phi)+2.$ 
\end{proposition}

\subsection{ Organization of the article} \label{organization}
 Before we turn to the proof of Theorem \ref{nonadarestrict}, we shall first clarify the notion of linearly adapted coordinates in Section \ref{prelim}.

Moreover, as in the preceding papers \cite{IKM-max}, \cite{IM-uniform}, assuming that the 
coordinates $x$ are linearly adapted, it  will be natural to distinguish the cases where  $d(\phi)<2$  and where $d(\phi)\ge 2,$ since, in contrast to the first case,  in the latter case in many situations a reduction to estimates for one-dimensional oscillatory integrals will be possible, which in return can be performed by means of van der Corput's lemma (\cite{stein-book}), respectively  the van der  Corput type Lemma \ref{corput}. The latter result will be stated in Section \ref{prelim} too.

Our discussion of the case where  $d(\phi)<2$ will rely on certain normal forms to which $\phi$ can be transformed by means of a linear change of coordinates. These will be derived in Section  \ref{normalforms}.

Next, in Section \ref{outside}, as a first step in the proof of Theorem \ref{nonadarestrict} we shall show that one may  reduce the restriction estimate to the piece of surface which lies above  a small ,  ``curved-conic ''  neighborhood of the principal root
 jet $\psi.$  This step works in all cases, no matter what the value of $d(\phi)$ is.
 
 Sections \ref{nearAD} to  \ref{complex2} will be devoted to the proof of Theorem \ref{nonadarestrict} in the case where $d(\phi)<2.$  Some of the main tools will consist of various kinds of dyadic domain decompositions   in combination with Littlewood-Paley theory and re-scaling arguments,  and  additional dyadic decompositions in frequency space. It turns out that the particular case where $m=2$ in \eqref{m1}, \eqref{prjet} requires a more refined analysis than the case $m\ge 3.$  Indeed, in this case, it turns out that further dyadic decompositions with respect to the distance to a certain ``Airy cone'' are needed. This particular case will be discussed in Section \ref{propm2_A}, with the exception of the endpoint $p=p_c.$   Indeed, the discussion of this endpoint in the cases left when $m=2$ (compare Proposition \ref{m2_A}) will require  rather intricate complex interpolation arguments, which will be presented in 
 Sections  \ref{complex1} and  \ref{complex2}. A further, useful tool will be Lemma \ref{doublesum} on oscillatory double sums, whose proof will be given in the Appendix in Section \ref{doublesumproof}.
  
 We should like to mention that a beautiful, real interpolation method has been  devised  by Bak and Seeger recently in \cite{bs}, which  in many cases allows to replace the more classical complex interpolation methods in the proof  of Stein-Tomas-type Fourier restriction estimates by substantially  shorter proofs.     
  In \cite{IM-rest2}, we shall be able to make use  of this new method in a few situations, but it  does not seem to apply to the situations arising in Proposition \ref{m2_A}.

 Sections \ref{prep} - \ref{nearrootjet} will deal with  the case where $d(\phi)\ge 2.$ It is natural to decompose the surface $S$ according to the ``root structure''  of the function $\phi,$ which in return is reflected by properties of the Newton diagram associated to $\pad$ (cf.  \cite{phong-stein}, \cite{IKM-max} and \cite{IM-uniform}.   More precisely, we shall decompose the domain $\Om$ into certain domains $D_l,$ which are homogeneous in adapted coordinates, and intermediate ``transition'' domains $E_l,$  and consider the corresponding decomposition of the surface $S.$  The particular domain $D_l$ which contains the principal root jet  $x_2=\psi(x_1)$ will be called $D_\pr.$ It is this  domain whose discussion  will require the most refined arguments. All this is described  in Section \ref{prep}.
 Next, in Section \ref{transition}, we estimate the contribution of the transition domains $E_l$ to the restriction problem. It turns out that this works whenever $d(\phi)\ge 2.$  Similarly, in Section \ref{homogeneous} we can also treat the contributions by the domains $D_l$ different from $D_\pr$ whenever $d(\phi)\ge 2.$ 
 
What remains is the domain $D_\pr.$ The contribution by this domain is studied in Section \ref{nearrootjet}, by means of a certain domain decomposition algorithm, which, roughly speaking, reflects the ``fine splitting'' of roots of $\pa_2\pad.$ 
In this discussion, various cases arise, and  there is one case in which we may fibre the corresponding piece of surface into a family of curves with non-vanishing torsion, so that we can apply Drury's restriction theorem for curves \cite{drury}.  However, it turns out that this requires that $d(\phi)\ge 5.$ 
 
 \smallskip
  What remains open at this stage is the proof  of Proposition \ref{nearjet} in the case where $2\le \hl(\phi)<5.$ The discussion of this case requires substantially more refined techniques  and interpolation arguments, and will be the content of  \cite{IM-rest2}.
  
  \smallskip
Finally, in Section  \ref{necessary}, we shall  employ a Knapp-type argument in order to show that  the condition $p'\ge p'_c$  is necessary in Theorem \ref{nonadarestrict}, and conclude the article with a proof of Proposition \ref{rheight}.
 
 \medskip
 {\bf Conventions:}  In this article, we shall use the ``variable constant'' notation, i.e.,  many constants appearing in the paper, often  denoted by  $C,$  will typically have different values  at different lines.  Moreover, we shall use  symbols such as  $\sim, \lesssim$ or $\ll$ in order to avoid  writing down constants. By $A\sim B$ we mean that there are constants $0<C_1\le C_2$ such that $C_1 A\le B\le C_2 A,$ and  these constants will not depend on the relevant parameters arising in the context in which the quantities $A$ and $B$  appear. Similarly, by $A\lesssim B$ we mean that there  is a  (possibly large)  constant  $C_1>0$ such that  $A\le C_1B,$  and by $A\ll B$ we mean that there  is a sufficiently  small  constant  $c_1>0$ such that  $A\le c_1B,$  and  again these constants do not depend on the relevant parameters.

By $\chi_0 $ and $\chi_1$ we shall  always denote  smooth cut-off functions with compact support on $\RR^n,$ where $\chi_0 $ will   be supported in a neighborhood of the origin, whereas $\chi_1=\chi_1(x)$  will be support away from the origin in each of its coordinates $x_j,$ i.e., $|x_j|\sim 1$ for every $j=1,\dots,n.$ These cut-off functions may also vary from line to line, and may  in some instances, where several of  such functions of different variables appear within the same formula, even   designate different functions.

Also, if we speak of the {\it slope}  of a line such as a supporting line to a Newton polyhedron, then we shall actually  mean the modulus of the slope.

\setcounter{equation}{0}
\section{Preliminaries: Linear height, and   van der Corput type estimates}\label{prelim}

In analogy with Varchenko's notion of height , let us introduce  the notion of {\it linear height} of $\phi,$ which measures  the upper limit of all Newton distances of $\phi$ in linear coordinate systems:
$$
\hl(\phi):=\sup\{d(\phi\circ T):T\in GL(2,\RR)\}.
$$
Note that 
$$
d(\phi)\le \hl(\phi)\le h(\phi).
$$
We also say that a linear coordinate system $y=\y$ is {\it linearly adapted} to $\phi,$ if $d_y=\hl(\phi).$ Clearly, if there is a linear coordinate system which is adapted to $\phi,$ it is in particular linearly adapted to $\phi.$ The  following proposition gives a  characterization of linearly adapted coordinates under the complementary Assumption \ref{assumption}.

\begin{proposition}\label{linada}
If $\phi$ satisfies Assumption \ref{assumption}, and if $\phi=\phi(x),$  then the following are equivalent:

(a) The coordinates $x$ are linearly adapted to $\phi.$

(b)  If the principal face $\pi(\phi)$ is contained in the   line 
 $$
 L=\{(t_1,t_2)\in \RR^2:\ka_1t_1+\ka_2 t_2=1\}, 
$$
then either $\ka_2/\ka_1\ge 2$ or $\ka_1/\ka_2\ge 2.$

Moreover, in  all linearly adapted coordinates $x$ for which $\ka_2/\ka_1>1,$  the principal face of the Newton polyhedron is the same, so that in particular the number $m:=\ka_2/\ka_1$ does not depend on the choice of the linearly adapted coordinate system.

\end{proposition} 
This result shows in particular that linearly adapted coordinates always exist under  Assumption \ref{assumption}, since either the original coordinates for $\phi$ are already linearly adapted, or we arrive at such coordinates after applying the first step in Varchenko's algorithm (when $\ka_2/\ka_1=1$ in the original coordinates).

\proof
In order to prove that (a) Êimplies (b), assume that $d_x:=d(\phi)=\hl(\phi).$ By interchanging the coordinates $x_1$ and $x_2,$ if necessary, we may assume that $\ka_2/\ka_1\ge 1,$  where we recall that $\ka_2/\ka_1\in\NN.$ Now, if we had $\ka_2/\ka_1= 1,$ then, by Varchenko's algorithm, there would exist a linear change of coordinates of the form $y_1=x_1, y_2=x_2-cx_1$ so that $d_y>d_x=d,$ which would  contradict the maximality of $d_x.$ Thus, necessarily $\ka_2/\ka_1\ge 2.$
\smallskip

Conversely, assume without loss of generality that $\ka_2/\ka_1\ge 2.$ Consider any matrix $T=\left(
\begin{array}{cc}
  a&  b    \\
 c & d    \\
\end{array}
\right)\in GL(2,\RR),$
and the corresponding linear coordinates $y$ given by 
$$x_1=ay_1+by_2, \quad x_2=cy_1+dy_2.
$$
To prove (a), we have to show that $d_y\le d_x$ for all such matrices $T.$ 

\smallskip
{\bf 1. Case.} $a\ne 0.$ Then we may factorize $T=T_1T_2,$ where
$$T_1:=\left(
\begin{array}{cc}
  a&  0    \\
 c & \frac{ad-bc}a    \\
\end{array}
\right),\quad
T_2:=\left(
\begin{array}{cc}
  1&  \frac ba    \\
 0 & 1    \\
\end{array}
\right).
$$
We first consider $T_2.$ Since $\phi_\pr(T_2y)=\phi_\ka(y_1+\frac ba y_2, y_2),$ where $y_2$ is $\ka$-homogenous of degree $\ka_2>\ka_1,$ where $\ka_1$ is the $\ka$-degree of $y_1,$ we see that the $\ka$-principal part of $\phi\circ T_2$ is given by $(\phi\circ T_2)_\ka=\phi_\ka,$ so that $\phi\circ T_2$ and $\phi$ have the same principal face, and in particular the same Newton distance. This shows that we may assume without loss of generality that $b=0.$ Then necessarily $d\ne 0. $ But then our change of coordinates is of the type $x_1=ay_1, \quad x_2=cy_1+dy_2$ considered in Lemma 3.2 of \cite{IM-ada}, so that this lemma implies that  $d_y\le d_x.$ Indeed, one finds more precisely that $d_y<d_x,$ if $c\ne 0,$ and $d_y=d_x$ otherwise.

\smallskip
{\bf 2. Case.} $a= 0, d=0.$
Since separate scalings of the coordinates have no effect on the Newton polyhedra, $T$ then essentially interchanges the roles of $x_1$ and $x_2,$ i.e., the Newton polyhedron is reflected at the bi-sectrix under this coordinate change. This shows that here $d_y=d_x.$

\smallskip
{\bf 3. Case.} $a= 0, d\ne 0.$  Then we may factorize $T=\left(
\begin{array}{cc}
  0&  b    \\
 c & d    \\
\end{array}
\right)=T_1T_2,$ where
$$
T_1:=\left(
\begin{array}{cc}
  0&   1    \\
 1 & 0   \\
\end{array}
\right),\quad
T_2:=\left(
\begin{array}{cc}
  c&  d    \\
 0 & b    \\
\end{array}
\right).
$$
We have seen in the previous  cases  that both $T_1$  and $T_2$ do not change the Newton distance, and thus here $d_y=d_x.$ 
This concludes the proof of the first part of Proposition \ref{linada}.
\medskip

Assume finally that $x$ and $y$ are two linearly adapted coordinate systems  for $\phi,$ for which the corresponding principal weights $\ka$ and $\ka '$ satisfy $\ka_2/\ka_1>1$  and $\ka'_2/\ka'_1>1,$  respectively. Choose $T\inÊGL(2,\RR)$ such that $x=Ty.$

 Inspecting  the three cases from the previous argument, we see that in Case 1 the mapping $T_2 $ does not change the principal face, and that necessarily  $c=0,$ since otherwise we had $d_y<d_x.$ But then also $T_1$ does not change the principal face.
 Case 2 cannot arise here, since we assume that both $\ka_2/\ka_1>1$  and $\ka'_2/\ka'_1>1,$ and similarly Case 3 cannot apply. This proves also the second statement in the proposition.  

\qed

We shall often make use of van der Corput type estimates. This includes the classical van der Corput Lemma \cite{vdC}  (see also \cite{stein-book}) as well as variants of it, going back to   J.\,E. Bj\"ork  (see \cite{domar}) and  G.\,I.~Arhipov \cite{arhipov}.

\begin{lemma}\label{corput}
Let $M\ge 2\ (M\in\NN)$, and let $f$ be a real valued function of class $C^M$ defined on an interval $I\subset \RR$.  Assume that  either
 \begin{itemize}
\item[(i)] $|f^{(M)}(s)|\ge 1$ on $I,$ or that 
\item[(ii)]  $f$
is of polynomial type  $M\ge 2,$ i.e., there are positive constants $c_1,c_2>0$ such that
$$
c_1\le\sum^M_{j=1}|f^{(j)}(s)|\le c_2\quad\mbox{for every}\  s\in I,
$$
and $I$ is compact.
 \end{itemize}

Then  the following hold true: For every $\la\in\RR, $ 
 \begin{itemize}
\item[(a)]  $$
\Big|\int_{I}e^{i\la f(s)} g(s)\, ds\Big| \le C(\| g\|_{L^\infty(I)} +\|g'\|_{L^1(I)})\,(1+|\la|)^{-1/M},
$$
where the constant $C$ depends only on   $M$ in case (i), and on $M,c_1, c_2$ and $I$ in case (ii).
\item[(b)] If $G\in L^1(I)$ is a non-negative function which is majorised by a function $H\in L^1(I)$ such that $\hat H\in L^1(\RR),$ then
$$
\int_{I}G(\la f(s)) \, ds\le C |\la|^{-1/M},
$$
 \end{itemize}
 where the constant $C$ depends only on   $M$ and $\|H\|_1+\|\hat H\|_1$  in case (i), and on $M,c_1, c_2,I$ and $\|H\|_1+\|\hat H\|_1$ in case (ii).
\end{lemma}

\begin{proof} 
For (a), we refer to \cite{vdC}, \cite{stein-book}, \cite{domar}) and \cite{arhipov}. Moreover, it  is well-known (see \cite{vdC}) that (b) is an immediate consequence of (a).  Indeed,  by means of the Fourier inversion formula  and Fubini's theorem  we may estimate 
$$
\int_{I}G(\la f(s)) \, ds\le\frac 1{2\pi} \Big| \int\hat H(\xi)  \int_I e^{i\xi \la f(s)}\, ds d\xi\Big| \le C |\la|^{-1/M} 
\int_\RR|\hat H(\xi)||\xi|^{-1/M}\, d\xi.
$$
\end{proof}

We remark that the conditions on the function $G$ in (b) are satisfied in particular if $G=|\vp|,$ where $\vp$ is of Schwartz class.

 \setcounter{equation}{0}
\section{Normal forms of $\phi$ under linear coordinate changes when  $\hl<2$}\label{normalforms}

\bigskip

In this section we shall  provide normal forms of the functions $\phi$ under linear coordinate changes when $\hl<2.$ This extends Siersma's work on analytic functions \cite{siersma} to the smooth, finite type case. The designation of the type of singularity that we list below corresponds to Arnol'd's classification of singularities in the case of analytic functions (cf. \cite{agv} and  \cite{duistermaat}), i.e., in the analytic case, non-linear analytic changes of coordinates would  allow to further  reduce $\phi$ to Arnol'd's normal forms.

\begin{proposition}\label{normalform1}
Assume that $\hl(\phi)< 2,$ where $\phi$ satisfies Assumption \ref{assumption}.

Then, after applying  a suitable linear change of coordinates, $\phi$ can be written in the following form on a sufficiently small neighborhood of the origin: 

\begin{equation}\label{A}
\phi(x_1,x_2)=b(x_1,x_2)(x_2-\psi(x_1))^2 +b_0(x_1),
\end{equation}
where $b,b_0$ and $\psi$ are  smooth functions, and where  $\psi(x_1)=cx_1^{m}+O(x_1^{m+1})$, with $c\ne0$ and $m\ge 2.$ Moreover, we can distinguish two cases: 

\medskip
{\bf Case a.}   $b(0,0)\ne 0.$ Then either

\smallskip
(i) $b_0$ is flat, \hfill (singularity of type $A_\infty$) 

\smallskip
or 
\smallskip

(ii) $b_0(x_1)=x_1^n \beta(x_1),$ where $\beta(0)\ne 0$ and $n\ge 2m+1.$ \hfill (singularity of  type $A_{n-1})$ In these cases we say that $\phi$ is  of type $A$.

\medskip
{\bf Case b.}   $b(0,0)= 0.$ 
Then we may assume that 
\begin{equation}\label{D}
b\x=x_1 b_1\x+x_2^2 b_2(x_2),
\end{equation}

where $b_1$ and $b_2$ are smooth functions, with $b_1(0,0)\ne 0.$

\smallskip
Moreover, either 

\smallskip
(i) $b_0$ is flat, \hfill (singularity of type $D_\infty$) 

or 
\smallskip

(ii) $b_0(x_1)=x_1^n \beta(x_1),$ where $\beta(0)\ne 0$ and $n\ge 2m+2.$ \hfill (singularity of  type $D_{n+1})$ 
  In these cases we say that $\phi$ is   of type $D$.

\end{proposition}

\begin{remarks}\label{distances1}
 \begin{itemize}
\item[(a)] It is easy to see that the Newton distance  $d=d(\phi)$ for these normal forms is given as follows:
$$
d=\begin{cases}   \frac {2m}{m+1}, &\quad  \mbox{  if $\phi$ is of type A,   }      \\
 \frac {2m+1}{m+1},& \quad  \mbox{  if $\phi$ is of type D},  
\end{cases}
$$
and by Proposition \ref{linada} that $\hl(\phi)=d,$ i.e., that the coordinates $x$ are linearly adapted.

\item[(b)] Similarly, the coordinates $y_1:=x_1, \, y_2:=x_2-\psi(x_1)$ are adapted to $\phi,$ and we can choose $\psi$ as the principal root jet. 

\item[(c)] When $\phi$ has a singularity of type $A_\infty$ or $D_\infty$  and satisfies Condition (R), then necessarily $b_0\equiv 0.$
 \end{itemize}
\end{remarks}

\proof

If $D^2\phi(0,0)$ had full rank $2,$ then the coordinates $x$ would already be adapted to $\phi,$ which would contradict  our assumptions. Therefore $\rank D^2\phi(0,0)\le 1.$ Let us denote by $P_n$ the  homogeneous part of degree $n$ of the Taylor polynomial of $\phi,$ i.e., $P_n(x_1,x_2)=\sum_{j+k=n} c_{jk} x_1^jx_2^k.$  
\medskip

{\bf 1. Case: $\rank D^2\phi(0,0)=1.$}

\smallskip
In this case, by passing to a suitable linear coordinate system, we may assume that $P_2\x=ax_2^2,$ where $a\ne 0.$ Consider the equation
\begin{equation}\nonumber
\pa_2\phi\x=0.
\end{equation}
By the implicit function theorem, it has locally a unique smooth solution $x_2=\psi(x_1),$ i.e.,  
$\pa_2\phi(x_1,\psi(x_1))=0.$ A Taylor series expansion of  the function $\phi(x_1, x_2)$ with respect to the variable $x_2$  around $\psi(x_1)$ then shows that 
\begin{equation}\label{2.9}
\phi(x_1,x_2)=b(x_1,x_2)(x_2-\psi(x_1))^2 +b_0(x_1),
\end{equation}
where $b$ and $b_0$ are  smooth functions and  $b(0,0)= \frac 12\pa_2^2\phi(0,0)=a\ne 0,$  whereas $b_0(x_1)=O(x_1^2),$ since $\phi(0,0)=0,\, \nabla \phi(0,0)=0$ (this is a special instance  of what would follow from a classical division theorem, see, e.g., \cite{hormander-book}).

Now, either $b_0$ is flat, which leads to type $A_\infty,$ or otherwise we may write $b_0(x_1)=x_1^n \beta(x_1),$ where $\beta(0)\ne 0$ and $n\ge 2,$ which leads to type $A_{n-1}.$ 

Observe also that the function $\psi$ cannot be flat, for otherwise the Newton polyhedron of $\phi$ would be the set $(0,2)+\RR_+^2,$ in case that $b_0$ is flat, or its principal edge would be the compact line segment with vertices $(0,2)$ and $(n,0).$  In the latter case, the principal part of $\phi$ is given by $\phi_\pr\x= ax_2^2+g(0)x_1^n,$ so that the maximal multiplicity $m(\phi_\pr)$ of any real root of $\phi_\pr$ along the unit circle is at most $1,$ whereas the Newton distance is given by $d=1/(\frac 12+\frac 1n)\ge 1.$ Therefore, in both cases, the coordinates $x$ would already be adapted to $\phi,$ according to Corollary 4.3 in \cite{IM-ada}. Notice also that the same argument shows that the coordinates $y$ introduced in $\eqref{adaptco}$ are adapted to $\phi,$ so that in particular indeed $h=2$ (in case that $b_0$ is flat) respectively $h=1/(\frac 12+\frac 1n)<2$ (if $b_0(x_1)=x_1^n \beta(x_1)$). 

In particular, since $\psi(0)=0,$ we can write $\psi(x_1)=c x_1^m+O(x_1^{m+1})$ for some $m\in\NN,$  where $c\ne 0.$ Note that indeed $m\ge 2,$ since $P_2\x=ax_2^2.$

Finally, when $b_0(x_1)=x_1^n \beta(x_1),$ a similar reasoning as before shows that the coordinates $x$ are already adapted if $2m\ge n,$ so that under Assumption \ref{assumption} we must have $n\ge 2m+1.$ 

\medskip

{\bf 2. Case: $D^2\phi(0,0)=0.$}

\smallskip

Then $P_2=0,$  and $P_3\ne 0,$  for otherwise we had $\hl\ge d\ge 1/(1/4+1/4)=2,$ which would contradict our assumption that $\hl<2.$  Notice also that $P_3\ne 0$ is homogeneous of  odd degree $3$, so that necessarily $m(P_3)\ge 1.$ 

Assume first that $m(P_3)=1.$ Then, passing to a suitable linear coordinate system, we may assume that $P_3\x=x_1(x_2-\alpha x_1)(x_2-\beta x_1),$ where either $\alpha\ne \beta$ are both real, or $\alpha=\overline{\beta}â$ are non-real. Then one checks easily that the Newton diagram of $P_3$ is a compact edge intersecting the bi-sectrix in its interior and  contained in the line given by $\frac 13 t_1+\frac 13 t_2=1.$ Consequently, it agrees with the  principal face $\pi(\phi),$  so that $P_3=\phi_\pr.$ We thus find that the Newton distance $d$ in this linear coordinate system  satisfies  $d=3/2>m(\phi_\pr),$ so that these  coordinates  would  already be adapted, contradicting our assumptions.

Assume next that $m(P_3)=3.$ Then, in a suitable linear coordinate system, $P_3\x$ $=x_2^3.$ These coordinates are then adapted to $P_3,$ so that $h(P_3)=d(P_3)=3>2.$ However, as has been shown in \cite{IKM-max}, p. 217, under Assumption \ref{assumption} this implies that the Taylor support of $\phi$ is contained in the region where $ \frac 16 t_1+\frac 13 t_2\ge 1.$ This in return implies  that $\hl\ge d\ge 1/(\frac 16+\frac 13)=2,$ in contrast to what we assumed.

\smallskip
We have thus seen that necessarily $m(P_3)=2.$ Then, after applying a suitable linear change of coordinates, we may assume that $P_3\x=x_1x_2^2,$  i.e., 
$$
\phi\x=x_1x_2^2+O(|x|^4).
$$
Consider here  the equation
\begin{equation}\label{2.10}
\pa_1\pa_2\phi\x=0.
\end{equation}
By the implicit function theorem, it has locally a unique smooth solution $x_2=\psi(x_1),$ i.e.,  
$\pa_1\pa_2\phi(x_1,\psi(x_1))=0.$ By means of a  Taylor series expansion of  the function $\pa_1\phi(x_1, x_2)$ with respect to the variable $x_2$  around $\psi(x_1)$ and subsequent integration in $x_1$ one then finds that
\begin{equation}\nonumber
\phi(x_1,x_2)=b(x_1,x_2)(x_2-\psi(x_1))^2 +b_2(x_1) x_2 +b_0(x_1),
\end{equation}
where $b, b_0$ and $b_2$ are  smooth functions.   Again, we have that $\psi(x_1)=c x_1^m+O(x_1^{m+1}),$ with $m\ge 2.$ Then \eqref{2.10} implies that $b_2'=0,$ and since $\pa_2\phi(0,0)=0,$ we see that $b_2=0,$ hence
\begin{equation}\label{2.11}
\phi(x_1,x_2)=b(x_1,x_2)(x_2-\psi(x_1))^2  +b_0(x_1),
\end{equation}
Moreover, since $\pa_2^2\phi(0,0)=0, \pa_1\pa_2^2\phi(0,0)\ne 0, \pa^3_2\phi(0,0)=0,$ we have that 
$$
b(0,0)=0, \quad \pa_1b(0,0)\ne 0\quad \mbox{ and } \pa_2 b(0,0)=0.
$$
By Taylor's formula, this implies that 
$$
b\x=x_1 b_1\x+x_2^2 b_2(x_2),
$$
where $b_1$ and $b_2$ are smooth functions, with $b_1(0,0)\ne 0.$

In a similar way as in Case 1, one can see that the coordinates from \eqref{adaptco} are adapted to $\phi.$  Moreover, if $b_0$ is flat, which leads to case $D_\infty$, then  $h=2,$  and if $b_0(x_1)=x_1^n \beta(x_1),$  which leads to case $D_{n+1},$ then $h=\frac {2n}{n+1}<2.$ 
Finally, one also checks easily that the coordinates $x$ in \eqref{adaptco} are already adapted to $\phi,$ if $2m+1\ge n,$ so that under our assumption we must have $ n\ge 2m+2.$ 

This concludes the proof of Proposition \ref{normalform1}.
\qed

\begin{cor}\label{s2.1}
Assume that $\phi$ satisfies Assumption \ref{assumption}. By passing to a suitable linear coordinate system, let us also assume  that the coordinates $x$ are linearly adapted to $\phi.$ Then, if  $d=d(\phi)<2,$  the critical exponent in  Theorem \ref{nonadarestrict} is given by $p'_c=2d+2.$
\end{cor}

\pf Proposition \ref{normalform1} shows that the principal face $\pi(\phi)$ of the Newton polyhedron of $\phi$ is a compact edge whose  ``upper'' vertex $v$ is one the following points $(0,2)$ or  $(1,2),$ which both  lie below the line $H:=\{(t_1,t_2):t_2= 3\}$ within the positive quadrant. On the other hand, $m+1\ge 3.$ It is then clear from the geometry of the lines $H$, the line $L$ which contains $\pi(\phi)$  and the line $\Delta^{(m)},$ that $\Delta^{(m)}$ will intersect $L$ above or in the vertex $v.$ Since, by Varchenko's algorithm, the point $v$ will also be a vertex of the Newton polyhedron of $\pad,$ this easily implies that $h^r(\phi)=d$ (compare Figure 2). This proves the claim.
\qed

\setcounter{equation}{0}
\section{Reduction to restriction estimates near the principal root jet }\label{outside}

We now turn to the proof of Theorem \ref{nonadarestrict} (which includes Theorem \ref{nonadarestrictanal}). As a first step,   we shall reduce considerations to a small  neighborhood of the principal root jet $\psi.$ Recall that our coordinates $x$ are assumed to satisfy \eqref{prjet} and \eqref{m1}.

Following \cite{IM-uniform}, by decomposing $\RR^2$  into its four quadrants, we shall in the sequel always assume   that the surface carried  measure $d\mu=\rho d\si$ is  supported in the positive quadrant where $x_1>0,x_2>0,$ i.e., that it is of the form
$$
\laa \mu, f\ra=\int_{(\RR_+)^2} f(x,\phi(x))\, \eta(x)\, dx, \qquad f\in C_0(\RR^3),
$$
where $\eta(x):=\rho(x,\phi(x))\sqrt{1+|\nabla \phi(x)|^2}$ is smooth and has its support in the neighborhood $\Omega$ of the origin, which we may assume to be sufficiently small. The contributions by the other quadrants can be treated in a very similar way.

 If $\chi$ is an integrable function defined on $\Omega,$ we put
$$
\mu^\chi:=(\chi\otimes 1)\mu,\  \mbox{ i.e.,}\   \laa \mu^\chi, f\ra=\int_{(\RR_+)^2} f(x,\phi(x))\, \eta(x)\chi(x)\, dx.
$$
Recall from \eqref{prjet} that $\psi(x_1)=c x_1^m+O(x_1^{m+1}).$ 
We choose a non-negative bump function $\chi_0\in C^\infty_0(\RR)$ supported in $[-1,1],$ and put 
$$
\rho_1\x:=\chi_0\Big(\frac{ x_2-cx_1^{m}}{\ve x_1^{m}}\Big),
$$
where $\ve>0$ is a small parameter to be determined later. Notice that $\rho_1$ is supported in the $\ka$-homogeneous subdomain of $\Om\cap\RR^2$ where
\begin{equation}\label{3.1}
|x_2-cx_1^m|\le \ve x_1^m,
\end{equation}
which contains the curve $x_2=\psi(x_1)$ when $\Om$ is sufficiently small.

\begin{proposition}\label{awayjet}
For every $\ve>0,$ when the support of $\mu$ is sufficiently small then  
$$
\Big(\int_S |\widehat f|^2\,d\mu^{1-\rho_1}\Big)^{1/2}\le C_{p,\ve} \|f\|_{L^p(\RR^3)},\qquad  f\in\S(\RR^3)
$$
whenever $p'\ge 2d+2.$ In particular, this estimate is valid for $p'\ge p'_c.$
\end{proposition}

The proof of this result will, by and large, follow the proof of Corollary 1.6 in \cite{IM-uniform}.
By $\{\de_r\}_{ r>0}$ we shall again denote the dilations associated to the principal weight $\ka.$ 
Fixing a suitable smooth cut-off function $\chi\ge 0$ on $\RR^2$ supported in an annulus $\A\subset\RR^2$ such that the functions $\chi_k:=\chi\circ \de_{2^k}$ form a partition of unity, we then decompose the measure $\mu^{1-\rho_1}$ dyadically as 
\begin{eqnarray}\label{3.2}
\mu^{1-\rho_1}=\sum_{k\ge k_0}\mu_k,
\end{eqnarray}
where $\mu_k:= \mu^{\chi_k(1-\rho_1)}.$  Let us extend the dilations $\de_r$ to $\RR^3$ by putting
$$
\de^e_r(x_1,x_2,x_3):=(r^{\ka_1} x_1, r^{\ka_2} x_2, r x_3).
$$
We re-scale the measure $\mu_k$ by defining $\mu_{0,(k)}:=2^{-k}\mu_k\circ \dee_{2^{-k}},$ i.e.,
\begin{equation}\label{3.3}
\laa \mu_{0,(k)},f\ra=2^{|\ka|k}\laa \mu_k, f\circ \dee_{2^{k}}\ra=\int_{(\RR_+)^2} f(x,\phi^k (x))\, \eta(\de_{2^{-k}}x)\chi(x)(1-\rho_1\x)\, dx,
\end{equation}
with 
\begin{equation}\label{3.4}
\phi^k(x):=2^k\phi(\de_{2^{-k}}x)=\phi_\ka (x)+ \mbox{error terms of order } O(2^{-\delta k}),
\end{equation}
 where $\delta>0.$ Recall here that the principal part $\phi_\pr$ of $\phi$ agrees with $\phi_\ka.$ This shows that the measures $\mu_{0,(k)}$ are supported on the smooth hypersurfaces  $S^k$ defined as the graph of $\phi^k$, their total variations are uniformly bounded,
i.e.,
$\sup_k\|\mu_{0,(k)}\|_1<\infty,$
and that they are approaching the surface carried  measure $ \mu_{0,(\infty)}$   on  $S$  defined  by
$$\laa \mu_{0,(\infty)},f\ra:=\int_{(\RR_+)^2} f(x,\phi_\ka (x))\, \eta(0)\chi(x)(1-\rho_1\x)\, dx
$$
as $k\to\infty. $
The proof of  Corollary 1.6 in \cite{IM-uniform}, which is based on a classical result by A. Greenleaf \cite{greenleaf} which relates uniform estimates for the Fourier transform of a surface carried measure to $L^p$-$L^2$- Fourier restriction estimates for this measure, as well as on Littlewood-Paley theory, then shows that it is sufficient to verify the following estimate in order to prove Proposition \ref{awayjet}:

\begin{lemma}\label{decay1}
If $k_0\in \NN$ is sufficiently large, then there exists a constant $C>0$ such that 
$$
|\widehat{ \mu_{0,(k)}}(\xi)|\le C (1+|\xi|)^{-1/d} \quad \mbox{for every }\xi\in\RR^3, k\ge k_0.
$$
\end{lemma}

We turn to the proof of  Lemma \ref{decay1}. Assume first that $\hl=\hl(\phi)\ge 2.$ Then $h(\phi)>2$ by Assumption \ref{assumption}. Thus, in this case, the proof of Lemma 2.3 in \cite{IM-uniform} shows that indeed the estimate in Lemma \ref{decay1} holds true. 

\smallskip
We may therefore assume that $\hl< 2,$ so that  $\phi$ can be assumed to be  given by one of the normal forms appearing in Proposition \ref{normalform1}. Moreover, then $\hl=d$ is the Newton distance.  Let us re-write
$$
\widehat{ \mu_{0,(k)}}(\xi)=\int_{(\RR_+)^2}e^{-i(\xi_1x_1+\xi_2x_2+\xi_3\phi^k\x)}\eta(\de_{2^{-k}}x)\chi(x)(1-\rho_1\x)\, dx,
$$ 
and observe that, by a partition of unity argument, it will suffice to prove the following:
\smallskip

Given any point $v\in\A$ such that 
\begin{equation}\label{3.5}
v_2-cv_1^m\ne 0,
\end{equation}
  there is neighborhood $V$ of $v$ such that  for every bump function $\chi_v\in C^\infty(\RR^2)$ supported in $V$ we have 
\begin{equation}\label{3.6}
|J^{\chi_v}(\xi)|\le C (1+|\xi|)^{-1/d} \quad \mbox{for every }\xi\in\RR^3, k\ge k_0,
\end{equation}
where 
$$
J^{\chi_v}(\xi):=\int_{(\RR_+)^2}e^{-i(\xi_1x_1+\xi_2x_2+\xi_3\phi^k\x)}\eta(\de_{2^{-k}}x)\chi_v(x)\, dx.
$$
\smallskip
To prove this, we shall distinguish the cases a and b from Proposition \ref{normalform1}.
\medskip

{\bf Case a ($\phi$ of type A).}  In this case, we see that $\ka=(\frac 1{2m}, \frac 12)$ and 
$$
\phi_\ka\x=\phi_\pr\x=b(0,0)(x_2-cx_1^m)^2,
$$
so that   $\frac 1d=\frac 12+\frac 1{2m}.$ After applying a suitable linear change of coordinates (and possibly complex conjugation to $J^{\chi_v}(\xi)$), we may assume that $b(0,0)=1.$ Then, the Hessian of $\phi_\ka$ is given by 
$$
\Hess(\phi_\ka)(x_1,x_2):=-4m(m-1)cx_1^{m-2}(x_2-cx_1^m).
$$

Therefore, by \eqref{3.5}, if $m=2,$ or $v_1\neq0,$ then $\Hess(\phi_\ka)(v)\neq 0$. In this case, in view of \eqref{3.4} we can apply the method of stationary phase for phase functions depending on small parameters and easily obtain 
$$
|J^{\chi_v}(\xi)|\le C (1+|\xi|)^{-1} \quad \mbox{for every }\xi\in\RR^3, k\ge k_0,
$$
provided $V$ is sufficiently small and $k_0$ sufficiently large.  Since $d\ge 1,$ this yields \eqref{3.6}.

\smallskip
We are left with the case where  $m>2$ and $v_1=0$. Since $v=(v_1,v_2)\in\A,$ this  implies that $v_2\neq0.$

Putting $\tilde\phi^k(y_1,y_2):=\phi^k(y_1,v_2+y_2),$  we may re-write $J^{\chi_v}(\xi)$ as 
$$
J^{\chi_v}(\xi)=e^{-iv_2\xi_2} \int_{(\RR_+)^2}e^{-i(\xi_1y_1+\xi_2y_2+\xi_3\tilde\phi^k(y_1,y_2))}\eta(\de_{2^{-k}}(y_1,v_2+y_2)\tilde\chi_0(y)\, dy,
$$
where $\tilde\chi_0$ is now supported in a sufficiently small neighborhood of the origin.
But,
\begin{eqnarray*}
\tilde\phi^k(y_1,y_2)&=&(v_2+y_2-cy_1^m)^2+O(2^{-\delta k})\\
&=&v_2^2+2v_2y_2 +\Big(y_2^2-2cv_2y_1^m +c^2y_1^{2m}-2cy_2y_1^m\Big)+O(2^{-\delta k}).
\end{eqnarray*}
The main term here is $(y_2^2-2cv_2y_1^m),$ which shows that the phase has a singularity of type $A_{m-1}.$

By means of a linear change of variables in $\xi$-space, which replaces $\xi_2+2v_2\xi_3$ by $\xi_2,$    we may thus reduce to assuming that the complete phase in the oscillatory integral $J^{\chi_v}(\xi)$ 
is given by 
$$
\xi_1y_1+\xi_2y_2+\xi_3\Big(y_2^2-2cv_2y_1^m +c^2y_1^{2m}-2cy_2y_1^m+O(2^{-\delta k})\Big).
$$
We claim that 
$$
|J^{\chi_v}(\xi)|\le C (1+|\xi|)^{-(\frac 12+\frac 1m)} \quad \mbox{for every }\xi\in\RR^3, k\ge k_0,
$$
which is even stronger  than \eqref{3.6}.

Indeed, if 
$$
 |\xi_3|\ll \max\{|\xi_1|,|\xi_2| \},
$$
then this  follows easily by integration by parts, so let us assume that 
$$
|\xi_3|\ge M \max\{|\xi_1|,|\xi_2|\}
$$
for some constant $M>0.$ Then $|\xi_3|\sim |\xi|.$ Consequently, by  applying first the method of stationary phase to the integration in $y_2,$ and then van der Corput's estimate to the $y_1$ integration, we obtain the estimate above. Observe here that these types of estimates are stable under small, smooth perturbations.

\medskip
{\bf Case b ($\phi$ of type D).}  In this case, we see that $\ka=(\frac 1{2m+1}, \frac {m}{2m+1})$ and 
$$
\phi_\ka\x=\phi_\pr\x=g(0,0)x_1(x_2-cx_1^m)^2,
$$
so that  $\frac 1d=\frac {m+1}{2m+1}.$ Again, we may assume without loss of generality that $g(0,0)=1,$  so that 
$$
\phi_\ka\x=x_1x_2^2-2cx_1^{m+1}x_2+c^2x_1^{2m+1}.
$$
Straight-forward computations show that
\begin{eqnarray*}
\pa_1^2\phi_\ka(x)&=&-2cm(m+1)x_1^{m-1}x_2+c^22m(2m+1)x_1^{2m-1},\\
\pa_1\pa_2\phi_\ka(x)&=&2x_2-2c(m+1)x_1^m,\quad \pa_2^2\phi_\ka(x)=2x_1,
\end{eqnarray*}
hence
$$
\Hess(\phi_\ka)(v):=-4(x_2-cx_1^m)\Big(x_2+c(m^2-m-1)x_1^m\Big).
$$
In view of \eqref{3.5}, we see that $\Hess(\phi_\ka)(v)\ne 0,$ if $v_2+c(m^2-m-1)v_1^m\ne 0,$ so that we can again estimate $J^{\chi_v}(\xi)$ by means of the method of stationary phase.

\smallskip
Let us therefore  assume that $\Hess(\phi_\ka)(v)=0,$ i.e., 
\begin{equation}\label{3.7}
v_2=-c(m^2-m-1)v_1^m.
\end{equation}
Observe that then $v_1\ne0, v_2\ne 0.$ Denote by 
$$
P_j(y):=\sum\limits_{|\al|=j}\frac1{\al!}\pa^\al\phi_\ka(v)y^\al
$$
the homogeneous Taylor polynomial of $\phi_\ka$ of degree $j,$   centered at $v.$ Then clearly
$$
P_2(y)=v_1\Big(y_2+(v_2-c(m+1)v_1^m)y_1/v_1\Big)^2=v_1\Big(y_2-cm^2v_1^{m-1}y_1\Big)^2.
$$
Moreover, by \eqref{3.7}
\begin{eqnarray*}
P_3(y)&=&-y_1\Big(\frac 13c^2m^2(m^3-m^2+2m+1)v_1^{2m-2}y_1^2 -cm(m+1)v_1^{m-1}y_1y_2+y_2^2\Big)\\
&=&-y_1Q(y).
\end{eqnarray*}
Passing to the  linear coordinates $z_1:=y_1, \ z_2:= y_2-cm^2v_1^{m-1}y_1,$  one  finds that 
$$P_2=v_1z_2^2,\quad P_3=-z_1\tilde Q(z),
$$
where again $\tilde Q=z_2^2+2\beta_1z_1z_2+\beta_2z_1^2$ is again a quadratic form. Moreover, straight-forward computations show that 
$$\beta_2=\frac{c^2}3m^2(m-1)(m^2-1)v_1^{2m-2}\ne 0.
$$ 
Applying Taylor's  formula, we thus find that, in the coordinates $z,$ 
$$
\tilde\phi(z):= \phi_\ka(v_1+y_1,v_2+y_2)= c_0+c_1z_1+c_2z_2+(v_1z_2^2-\beta_2z_1^3) -(z_1z_2^2+2\beta_1z_1^2z_2)+O(|z|^4). 
$$
Let us put $\phi^v(z):=\phi(z)-(c_0+c_1z_1+c_2z_2),$ so that $\phi^v(0,0)=0,\, \nabla \phi^v(0,0)=0.$
Then one finds that the principal part of $\phi^v$ is given by
$$
\phi^v_\pr(z)=v_1z_2^2-\beta_2z_1^3, \quad \mbox{ where }  \beta_2\ne 0.
$$
We can now argue in a very similar way as in the previous case. Indeed, by passing for the variables $x$ in the integral defining $J^{\chi_v}(\xi),$ and then applying first the method of stationary phase to the integration in $z_2,$ and subsequently van der Corput's estimate to the $z_1$ integration (in the case where $|\xi_3|\ge M \max\{|\xi_1|,|\xi_2|\}$), 
 we obtain the estimate  
$$
|J^{\chi_v}(\xi)|\le C (1+|\xi|)^{-(\frac 12+\frac 13)} \quad \mbox{for every }\xi\in\RR^3, k\ge k_0.
$$
Again, this is a stronger estimate  than \eqref{3.6}, since here 
$$\frac 1d=\frac 12+\frac1{4m+2}\le \frac 12+\frac 13.
$$ 
The proof of Proposition \ref{awayjet} is thus complete.

\bigskip

We are thus left with proving Fourier restriction estimates for the measure $\mu^{\rho_1}$ which is supported in the small neighborhood \eqref{3.1} of the principal root jet. Our main goal will thus be to prove the following
\begin{proposition}\label{nearjet}
Assume that $\phi$ satisfies the assumptions of Theorem \ref{nonadarestrict}. If $\ve>0$ is sufficiently small, then we have 
$$
\Big(\int_S |\widehat f|^2\,d\mu^{\rho_1}\Big)^{1/2}\le C_{p,\ve} \|f\|_{L^p(\RR^3)},\qquad  f\in\S(\RR^3)
$$
whenever $p'\ge p'_c .$ 
\end{proposition}
In combination with Proposition \ref{awayjet} this will conclude the proof of Theorem \ref{nonadarestrict}. Notice that by interpolation with the trivial $L^1$-$L^2$- restriction estimate, it will suffice to prove this for $p=p_c.$

\medskip
We shall distinguish between the cases where $\hl<2,$ and  where $\hl\ge 2,$ since their treatments will require somewhat different approaches. Moreover, when   $\hl\ge 5,$ some arguments simplify substantially  compared to the case where    $2<\hl <5,$ since we can then apply  restriction estimates for curves with non-vanishing torsion  originating from seminal work by S.W.  Drury,   so that we shall also distinguish between those subcases.

\setcounter{equation}{0}
\section{The case  when $\hl(\phi)<2$ }\label{nearAD}

In this case, we may assume that $\phi$ is given by one of the normal forms in Proposition \ref{normalform1}. Recall from Corollary \ref{s2.1} that then $p'_c=2d+2.$ Recall also that, because we are assuming  Condition (R) to hold,  the term $b_0$ in \eqref{A} respectively \eqref{D} vanishes identically if $\phi$ is of type $A_\infty$ or $D_\infty$ (cf. Remark \ref{distances1} (c)).
\medskip

In a first step, we shall follow the arguments  from the preceding section and   decompose the measure $\mu^{\rho_1}$ dyadically by means of the dilations associated to the principal weight $\ka.$ Applying  subsequent re-scalings, we may  then reduce ourselves by means of Littlewood-Paley theory to proving the following uniform restriction estimates \eqref{njet}:
\smallskip

For $k\in\NN$ denote by $\nu_k$ the measure given by 
\begin{equation}\label{nuk}
\laa \nu_k,f\ra=2^{|\ka|k}\laa \mu_k, f\circ \dee_{2^{k}}\ra=\int_{(\RR_+)^2} f(x,\phi^k (x))\, \eta(\de_{2^{-k}}x)\chi(x)\rho_1\x\, dx,
\end{equation}
where $\phi^k$ is again given by \eqref{3.4}. Observe  that 
\begin{equation}\label{sim1}
x_1\sim 1\sim x_2
\end{equation}
in the support of the integrand. Recall also from \eqref{phia1} that
$$
\phi\x=\pad(x_1,x_2-\psi(x_1)),
$$
where  according to \eqref{prjet} we may write 
$$
\psi(x_1)=x_1^m \om(x_1),\quad (m\ge 2),
$$
with a smooth  function $\om$   satisfying $\om(0)\ne 0.$

Then, if $\ve>0$  and $\de$ are chosen sufficiently small,  there are constants $C_{\ve}>0$ and $k_0\in \NN$ such that for every $k\ge k_0$
\begin{equation}\label{njet}
\Big(\int |\widehat f|^2\,d\nu_k\Big)^{1/2}\le C_{\ve} \|f\|_{L^{p_c}(\RR^3)},\qquad  f\in\S(\RR^3).
\end{equation}

\medskip
 
 In order to prove this estimate, observe that  $\phi^k$ can be written in the form
\begin{equation}\label{5.1}
\phi(x,\de):=\tilde b( x_1, x_2,\de_1, \de_2)\Big(x_2-x_1^m\om(\de_1x_1)\Big)^2+\de_3 x_1^n \beta(\delta_1x_1),
\end{equation}
where 
$$
\de=(\de_1,\,\de_2\, ,\de_3)=(2^{-\ka_1k},2^{-\ka_2k}, 2^{-(n\ka_1-1)k}) 
$$ 
are small
parameters  which tend to $0$ as $k$ tends to infinity, and where $\tilde b$ is a  smooth function in all variables given  by
\begin{eqnarray}\label{btilde}
\quad \tilde b( x_1, x_2,\de_1, \de_2)
:=\begin{cases}    b(\de_1 x_1,\de_2 x_2),& \mbox{for $\phi$ of type $A$},Ê  \\
x_1b_1(\de_1x_1,\de_2x_2)+ \de_1^{2m-1}x_2^2b_2(\de_2 x_2),& \mbox{for $\phi$ is  type $D$}.
\end{cases}
\end{eqnarray}
Note that  $\de_3:=0$ when $\phi$ is of type $A_\infty$ or $D_\infty.$
Recall also that here $x_1\sim1\sim x_2,$ and notice  that 
$$
\om(0)\ne 0, \mbox{  and  } \tilde b(x_1,x_2,0,0)\sim 1.
$$

It is thus easily  seen by means of a partition of unity argument that it will suffice to prove the following proposition in order to verify \eqref{njet}.
\begin{proposition}\label{s5.1}
Let $\phi(x,\de)$ be as in \eqref{5.1}. Then, for every point $v=(v_1,v_2)$ such that $v_1\sim 1$ and $v_2=v_1^m\om(0),$ there exists a neighborhood $V$ of $v$ in $(\RR_+)^2$ such that for every cut-off function $\eta\in \D(V),$ the measure 
$\nu_\de$ given by
$$
\laa \nu_\de,f\ra:=\int f(x,\phi (x,\de))\, \eta\x\, dx
$$
satisfies a restriction estimate 
\begin{equation}\label{5.2}
\Big(\int |\widehat f|^2\,d\nu_\de\Big)^{1/2}\le C_{\eta} \|f\|_{L^{p_c}(\RR^3)},\qquad  f\in\S(\RR^3),
\end{equation}
provided $\de$ is sufficiently small, with a constant $C_{\eta}$ which depends only on  some $C^k$-norm of $\eta.$

\end{proposition}

In oder to prove this proposition, we shall perform yet another dyadic decomposition, this time with respect to the $x_3$-variable. A straight-forward modification of  the proof of Corollary 1.6 in \cite{IM-uniform} then allows to reduce the proof again  by means of Littlewood-Paley theory to uniform restriction estimates for the following family of measures:
\begin{equation}\label{nudj}
\laa \nu_{\de,j},f\ra:=\int f(x,\phi (x,\de))\, \chi_1(2^{2j}\phi(x,\de))\eta\x\, dx.
\end{equation}

Here, $\chi_1\in\D(\RR)$ is a fixed, non-negative smooth bump-function supported in \hfill\newline $(-2,-1/2)\cup(1/2,\, 2)$ such that $\chi_1\equiv 1$ in a neighborhood of the points $-1$ and $1.$ Notice that $\nu_{\de,j}$ is supported where $|\phi(x,\de)|\sim 2^{-2j}.$ I.e., in place of \eqref{5.2}, it will be sufficient to prove an analogous uniform estimate

\begin{equation}\label{5.3}
\Big(\int |\widehat f|^2\,d\nu_{\de,j}\Big)^{1/2}\le C_{\eta} \|f\|_{L^{p_c}(\RR^3)},\qquad  f\in\S(\RR^3),
\end{equation}
for all $j\in\NN$ sufficiently big, say $j\ge j_0,$ where the constant $C_{\eta}$ does neither depend on $\de,$ nor on  $j.$

\smallskip
In order to verify \eqref{5.3}, we shall distinguish three cases, depending on the size of $2^{2j}\de_3.$

\subsection[caseAD1]{The situation where $2^{2j}\de_3 \gg1$}

Observe first that if $j$ is sufficiently large, then by \eqref{5.1} and since $x_1\sim 1,$  $\nu_{\de,j}=0$ unless  $\tilde b(v,\de_1,\de_2)$ and $\be(0)$ have  opposite signs. So, let us for instance assume that $\tilde b( x_1, x_2,\de_1, \de_2)>0$ and $\beta(\delta_1x_1)<0$ on the support of $\eta.$ Then $\tilde\beta:=-\beta>0,$ and we may re-write 
$$
2^{2j}\phi(x,\de)=2^{2j}\tilde b( x_1, x_2,\de_1, \de_2)\Big(x_2-x_1^m\om(\de_1x_1)\Big)^2-2^{2j}\de_3 x_1^n \tilde\beta(\delta_1x_1).
$$
We introduce new coordinates $y$ by putting $y_1:=x_1$ and $y_2:=2^{2j}\phi(x,\de).$ Solving for $x_2,$ one easily finds that 
\begin{equation}\label{5.4}
x_2=\tilde b_1\Big(y_1,\sqrt{2^{-2j} y_2+\de_3y_1^n\tilde\beta(\de_1y_1)},\de_1,\delta_2\Big)\, \sqrt{2^{-2j} y_2+\de_3y_1^n\tilde\beta(\de_1y_1)}+y_1^m\om(\de_1y_1)\, ,
\end{equation}
where $\tilde b_1$ has similar properties like $\tilde b.$ Moreover, by the support properties of the amplitude $\chi(2^{2j}\phi(x,\de))\eta\x,$ we see that also for the new coordinates we have 
$
y_1\sim 1\sim y_2,
$
and that we can re-write 
$$
\laa \nu_{\de,j},f\ra=\frac{2^{-2j}}{\sqrt{\de_3}}\int f\Big(y_1,\,\phi(y,\de,j), 2^{-2j}y_2\Big)\, a(y,\de,j)\, \chi_1(y_1)\chi_1(y_2)\, dy,
$$
with a cut-off function $\chi$ as before, and where $a(y,\de,j)$ is smooth in $y$ and $\de,$ with  $C^k$-norms uniformly bounded in $\de$ and $j,$ and where 
\begin{equation}\label{5.5}
\phi(x,\de,j):=\tilde b_1\Big(x_1,\sqrt{2^{-2j} x_2+\de_3x_1^n\tilde\beta(\de_1x_1)},\de_1,\delta_2\Big)\, \sqrt{2^{-2j} x_2+\de_3x_1^n\tilde\beta(\de_1x_1)}+x_1^m\om(\de_1x_1)\, .
\end{equation}

We have re-named the variable $y$ to become $x$ here, since if we  define the  measure $\tilde\nu_{\de,j}$  by
\begin{equation}\label{5.6}
 \laa\tilde\nu_{\de,j},f\ra:=\int f\Big(x_1,\,\phi(x,\de,j), x_2\Big)\, a(x,\de,j)\, \chi_1(x_1)\chi_1(x_2)\, dx,
\end{equation}
then the restriction estimate \eqref {5.3} for the measure $\nu_{\de,j}$ is equivalent to the following restriction estimate for the measure $\tilde\nu_{\de,j}:$

\begin{equation}\label{5.7}
\int |\widehat f|^2\,d\tilde\nu_{\de,j}\le C_{\eta}\,\sqrt{\de_3}\,2^{2j(1-\frac2{{p_c}'})}\, \|f\|^2_{L^{p_c}(\RR^3)},\qquad  f\in\S(\RR^3)
\end{equation}
for all $j\in\NN$ sufficiently big, say $j\ge j_0,$ where the constant $C_{\eta}$ does neither depend on $\de,$ nor on  $j.$

\smallskip

Formula \eqref{5.6} shows that the Fourier transform of the
measure $ \tilde\nu_{\de,j}$ can be expressed as an oscillatory integral
\begin{equation}\label{5.8}
\widehat{\tilde\nu_{\de,j}}(\xi)=\int e^{-i
\Phi(x,\de,j,\xi)}a(x,\de,j)\, \chi_1(x_1)\chi_1(x_2)\, dx,
\end{equation}
where the complete phase function $\Phi$ is given by 
\begin{equation}\label{5.9}
\Phi(x,\de,j,\xi):=\xi_2\phi(x,\de,j)+ \xi_3x_2+\xi_1x_1.
\end{equation}

Finally, we shall perform a Littlewood- Paley decomposition of the   measure $\tilde\nu_{\de,j}$ in each coordinate. To this end, we fix again  a suitable smooth cut-off function $\chi_1\ge 0$ on $\RR$ supported in $(-2,-1/2)\cup(1/2,2)$ such that the functions $\chi_k(t):=\chi_1(2^{1-k}t), k\in\NN\setminus\{0\},$ in combination with a suitable smooth function $\chi_0$ supported  in $(-1,1),$ form a partition of unity, i.e.,
\begin{equation}\label{pu1}
\sum_{k=0}^\infty\chi_k(t)=1\quad \mbox{ for all } t\in\RR.
\end{equation}
For every multi-index $k=(k_1,k_2,k_3)\in\NN^3,$ we put
\begin{equation}\label{pu2}
\chi_k(\xi):=\chi_{k_1}(\xi_1)\chi_{k_2}(\xi_2)\chi_{k_3}(\xi_3),
\end{equation}
and finally define the smooth functions $\nu_{k,j}$ by 
$$
\widehat{\nu_{k,j}}(\xi):=\chi_k(\xi)\widehat{\tilde\nu_{\de,j}}(\xi).
$$
In order to defray the notation, we have suppressed here the dependency of this smooth function on the small parametersÊ $\de.$ We then find that 
\begin{equation}\label{5.12}
\tilde\nu_{\de,j}=\sum_{k\in\NN^3} \nu_{k,j},
\end{equation}
in the sense of distributions. To simplify the subsequent discussion, we shall concentrate on those measures $\nu_{k,j}$ for which none of its components  $k_i$'s are zero, since the remaining cases where for instance $k_i$ is zero can be dealt with in the same way as the corresponding cases where $k_i\ge 1$ is small.

Now, if  $1\le \la_i=2^{k_i-1}, i=1,2,3,$ are dyadic numbers, we shall  accordingly write $\nu^{\la}_{j}$ in place of $\nu_{k,j},$ i.e., 
\begin{equation}\label{nulaj}
\widehat{\nu^\la_{j}}(\xi)=\chi_1\Big(\frac{\xi_1}{\la_1}\Big)\chi_1\Big(\frac{\xi_2}{\la_2}\Big)\chi_1\Big(\frac{\xi_3}{\la_3}\Big)\widehat{\tilde\nu_{\de,j}}(\xi).
\end{equation}
Note that 
\begin{equation}\label{5.13}
|\xi_i| \sim\la_i, \quad \mbox{  on  } \supp\widehat{\nu^{\la}_{j}}.
\end{equation}

Moreover, by \eqref{5.6},
\begin{eqnarray}\nonumber
\nu^{\la}_{j}(x)=\la_1\la_2\la_3&\int \check\chi_1\Big({\la_1}(x_1-y_1)\Big) \, \check\chi_1\Big({\la_2}(x_2-\phi(y,\de,j))\Big)\label{5.14}\\
&\check\chi_1\Big({\la_3}(x_2-y_2)\Big)
\,a(y,\de,j)\, \chi(y_1)\chi(y_2)\, dy,
\end{eqnarray}
where $\check f$ denotes the inverse Fourier transform of $f.$

\medskip
We begin by  estimating the Fourier transform of $ \nu^{\la}_{j}.$  To this end, we first integrate in $x_1$ in \eqref{5.6}, and then in $x_2,$ assuming  that \eqref{5.13} holds true. 
We shall concentrate on those $ \nu^{\la}_{j}$  for which
\begin{equation}\label{5.15}
\la_1\sim\la_2\sim \sqrt{\de_3}2^{2j}\la_3.
\end{equation}
\smallskip
In all other cases, the phase has  no critical point on  the support of the amplitude, and we obtain much faster Fourier decay estimates by repeated integrations by parts, so that the corresponding terms  can be considered as error terms. 
Observe also that 
$$
\frac{\pa^2}{\pa x_2^2}\Phi(x,\de,j,\xi)\sim \la_2 \de_3^{-3/2}2^{-4j} 
$$
on the support of the amplitude. We therefore distinguish two subcases.

\smallskip
{\bf 1. Case: $1\le \la_1\lesssim \de_3^{3/2}2^{4j}.$}
 In this case  we cannot gain from the integration in  $x_2$ but, by applying van der Corput's lemma 
(or the method of stationary phase) in $x_1$ we obtain 
\begin{equation}\label{5.16}
\|\widehat{\nu^{\la}_{j}}\|_\infty\lesssim
\frac{1}{\la_1^{1/2}}.
\end{equation}

\smallskip
{\bf 2. Case: $\la_1\gg \de_3^{3/2}2^{4j}.$}  Then, by first applying the method of stationary phase to the integration in $x_1,$ and subsequently applying the classical van der Corput lemma (or Lemma \ref{corput}, with $M=2$) to the integration in $x_2,$ we obtain
\begin{equation}\label{5.17}
\|\widehat{\nu^{\la}_{j}}\|_\infty\lesssim \frac{1}{\la_1^{1/2}}\frac 1{(\la_2 \de_3^{-3/2}2^{-4j})^{1/2}}
\lesssim \frac{\de_3^{3/4}2^{2j}}{\la_1}.
\end{equation}

\medskip
Next, from \eqref{5.14}, we trivially obtain the following estimate for the $L^\infty$-norm of $ \nu^{\la}_{j}:$ 
\begin{equation}\label{5.18}
\|\nu^{\la}_{j}\|_\infty\lesssim \la_2\sim \la_1,
\end{equation}
in Case 1 as well as in Case 2. All these estimates are uniform in $\de,$ for $\de$ sufficiently small.

\medskip

For each of the measures $\nu^{\la}_{j},$ we can now obtain suitable restriction estimates by applying the usual approach. Let us denote by $ T_{\de,j}$ the convolution operator 
$$
 T_{\de,j}: \vp \mapsto \vp\ast \widehat{\tilde\nu_{\de,j}},
$$
and similarly by $ T_j^\la$ the convolution operator 
$$
 T_j^\la: \vp \mapsto \vp\ast \widehat{\nu^{\la}_{j}}.
$$
Formally, by \eqref{5.12}, $ T_{\de,j}$ decomposes as
\begin{equation}\label{5.19}
T_{\de,j}=\sum_{k\in\NN^3} T_j^{2^k},
\end{equation}
if $2^k$ represents the vector $2^k:=(2^{k_1},2^{k_2},2^{k_3})$ (with a suitably modified definition of  $T_j^{2^k}$ when one of the components $k_i$ is zero).
If we denote by $\|T\|_{p\to q}$ the norm of  $T$ as an operator from $L^p$ to $L^q,$ then clearly $\| T_j^\la\|_{1\to\infty}=\|\widehat{\nu^{\la}_{j}}\|_\infty$ and $\| T_j^\la\|_{2\to 2}
=\|\nu^{\la}_{j}\|_\infty.$

The estimates \eqref{5.16} - \eqref{5.18} thus yield the following bounds:
\begin{eqnarray*}
\| T_j^\la\|_{1\to\infty}\lesssim 
\begin{cases}       \la_1^{-1/2}, & \mbox{  if  } \ 1\le \la_1\lesssim \de_3^{3/2}2^{4j},\\
   \dfrac{\de_3^{3/4}2^{2j}}{\la_1},&  \mbox{  if  } \ \la_1\gg \de_3^{3/2}2^{4j},
\end{cases}
   \end{eqnarray*}
   and 
$
\| T_j^\la\|_{2\to 2}\lesssim \la_1.
$
Interpolating these estimates,  and defining the critical interpolation parameter $\th=\th_c$ by  $1/p'_c=(1-\theta)/\infty+\theta/2=\theta/2,$
i.e.,
$$
\th:=\frac 1{p_c'},
$$ 
we find that 
\begin{eqnarray}\label{5.20}
\| T_j^\la\|_{p_c\to p_c'}\lesssim 
\begin{cases}       \la_1^{\frac{3\theta-1}2}, & \mbox{  if  } \ 1\le \la_1\lesssim \de_3^{3/2}2^{4j},\\
  \de_3^{\frac 34(1-\theta)} 2^{2(1-\theta) j}\la_1^{2\theta-1},&  \mbox{  if  } \ \la_1\gg \de_3^{3/2}2^{4j},
\end{cases}
   \end{eqnarray}
where according to Remark \ref{distances1}
\begin{equation}\label{5.21}
\theta=\begin{cases}   \frac {m+1}{3m+1}, & \mbox{ if $\phi$ is of type $A$},    \\
 \frac {m+1}{3m+2},  &     \mbox{ if $\phi$ is of type $D$}.
\end{cases}
\end{equation}
Observe that in particular 
\begin{equation}\label{thcond}
\frac 13<\th\le \frac 3 7,
\end{equation}
and  $\th=3/7$ if and only if $m=2$ and $\phi$ is of type $A.$ The latter case will turn out to be the most difficult one.




\medskip
Observe next that the  main contributions to the series \eqref{5.19} come from those dyadic $\la=2^k$ for which $\la_1\sim\la_2\sim \sqrt{\de_3}2^{2j}\la_3.$ Under these relations, for $\la_1$ given, $\la_2$ and $\la_3$ may only vary in a finite set whose cardinality is bounded by a fixed number.
This shows that, up to an easily bounded error term, 
$$
\| T_{\de,j}\|_{p_c\to p_c'}\lesssim \sum_{\la_1=2}^{ \de_3^{3/2}2^{4j}}\la_1^{\frac{3\theta-1}2 } +
\sum_{\la_1>\de_3^{3/2}2^{4j}} \de_3^{\frac 34(1-\theta)} 2^{2(1-\theta) j} \la_1^{(2\theta-1)}.
$$
Here, and in the sequel, summation over $\la_1,\la_2$ etc. means that we sum over dyadic numbers $\la_1,\la_2$ etc. only.
Now, by \eqref{thcond}, $2\theta-1<0 $ and $  0\le 3\theta-1\le 1,$
which  yields 
$$
\| T_{\de,j}\|_{p_c\to p_c'} \lesssim \de_3^{\frac 34(3\theta-1)}2^{(3\theta-1)2j}.
$$

Applying the usual $T^*T$-argument, we thus need to prove that 
$$
\de_3^{\frac 34(3\theta-1)}2^{(3\theta-1)2j}\le C \sqrt{\de_3}\,2^{2j(1-\frac2{p_c'})}
$$ in order to  verify that  the restriction estimate \eqref{5.7} holds true for $p=p_c=2d+2.$ However, since $2/p_c'=\theta,$ the previous estimate  is equivalent to 
$$
2^{2j(4\theta-2)}\le C \de_3^{\frac {5-9\theta}4}.
$$
But, since $2^{2j}\de_3 \gg1$ and $2\theta-1<0,$ we see that 
$2^{2j(4\theta-2)}\le C\, \de_3^{2-4\theta},$ and therefore we only have to verify that $2-4\theta\ge (5-9\theta)/4,$ i.e., 
$7\theta\le 3,$ which is true according to \eqref{thcond}.
\smallskip

This is obvious  by \eqref{5.21}, and we thus have  verified the  restriction estimate \eqref{5.3}  in this subcase.

\bigskip

{\bf There remains the case  $2^{2j}\de_3 \le C,$ }where $C$ is a fixed, possibly large constant. 

\medskip

Observe  that 
 the change of variables $(x_1,x_2)\mapsto(x_1, x_2+x_1^m\om(\de_1x_1))$ and subsequent scaling in $x_2$ by the factor $2^{-j}$ allows to re-write  the measure $\nu_{\de,j}$ given by \eqref{nudj} as 
$$
 \laa\nu_{\de,j},f\ra=2^{-j}\int f\Big(x_1,2^{-j}x_2+x_1^m\om(\de_1x_1),2^{-2j}\pad(x,\de,j)\Big)\,a(x,\de,j)\,dx,
$$
where here 
\begin{equation}\label{padj}
\pad(x,\de,j):=\tilde b( x_1, 2^{-j}x_2+x_1^m \om(\de_1
x_1),\de_1,\de_2)x_2^2
+2^{2j}\de_3 x_1^n \beta(\de_1 x_1),
\end{equation}
and
\begin{equation}\label{adef}
 a(x,\de,j):= \chi_1(\pad(x,\de,j))\,\eta(x_1,
 2^{-j}x_2+x_1^m\om(\de_1x_1)).
\end{equation}
Let us here introduce the re-scaled measure $\tilde\nu_{\de,j}$  by
\begin{equation}\label{5.23}
 \laa\tilde\nu_{\de,j},f\ra:=\int f\Big(x_1,2^{-j}x_2+x_1^m\om(\de_1x_1),\pad(x,\de,j)\Big)\,a(x,\de,j)\,dx.
\end{equation}
Then, it is easy to see by means of a scaling in the variable $x_3$ by the factor $2^{-2j}$ that the restriction estimate \eqref {5.3} for the measure $\nu_{\de,j}$ is equivalent to the following restriction estimate for the measure $\tilde\nu_{\de,j}:$

\begin{equation}\label{rest4}
\int_S |\widehat f|^2\,d\tilde\nu_{\de,j}\le C_{\eta}\,2^{(1-\frac 4{p'_c})j}\, \|f\|^2_{L^{p_c}(\RR^3)}=C_{\eta}\,2^{(1-2\th)j}\, \|f\|^2_{L^{p_c}(\RR^3)},\qquad  f\in\S(\RR^3),
\end{equation}
for all $j\in\NN$ sufficiently big, say $j\ge j_0,$ where the constant $C_{\eta}$ does neither depend on $\de,$ nor on  $j.$

\medskip

In order to prove \eqref{rest4}, we again distinguish two subcases.

\subsection[caseAD2]{The situation  where $2^{2j}\de_3 \ll1$}

Notice that here the phase $\pad(x,\de,j)$ is a small perturbation of $\tilde b(v,0,0)x_2^2,$ where $\tilde b(v,0,0)\sim 1.$ This shows that also in the new coordinates, $x_1\sim 1\sim x_2$ on the support of the amplitude $a,$ which in return implies 
\begin{equation}\label{5.25}
\pa_2\pad(x,\de,j)\sim 1.
\end{equation}

We can thus write 
\begin{equation}\label{5.26}
\widehat{\tilde\nu_{\de,j}}(\xi)=\int e^{-i
\Phi(x,\de,j,\xi)}a(x,\de,j)\chi_1(x_1)\chi_1(x_2)\,dx,
\end{equation}
where the complete phase function $\Phi$ is now given by 
\begin{equation}\label{5.27}
\Phi(x,\de,j,\xi):=\xi_3\pad(x,\de,j)+2^{-j}\xi_2x_2+\xi_2x_1^m\om(\de_1
x_1)+\xi_1x_1,
\end{equation}
with $\pad$ given by \eqref{padj}, and where $\chi$ has similar properties as before

As in the previous subcase, we  perform a Littlewood- Paley decomposition \eqref{5.12} of the  the measure $\tilde\nu_{\de,j}$ in each coordinate and define the measure $\nu_j^\la$ by \eqref{nulaj}. 
Then here we have 
\begin{eqnarray}\nonumber
\nu^{\la}_{j}(x)=\la_1\la_2\la_3&\int \check\chi_1\Big({\la_1}(x_1-y_1)\Big) \, \check\chi_1\Big({\la_2}(x_2-2^{-j}y_2-y_1^m\om(\de_1y_1))\Big)\label{5.28}\\
&\check\chi_1\Big({\la_3}(x_3-\pad(y,\de,j))\Big)
\,a(y,\de,j)\,\chi_1(y_1)\chi_1(y_2)\,dy,
\end{eqnarray}
where $\check f$ denotes the inverse Fourier transform of $f.$

\medskip
We begin by  estimating the Fourier transform of $ \nu^{\la}_{j}.$  To this end, we first integrate in $x_2$ in \eqref{5.26}, and then in $x_1.$ We may assume  that \eqref{5.13} holds true. Since  then the phase function  $\Phi$ has no critical point in $x_2$ unless $\la_3\sim 2^{-j}\la_2,$ and similarly in $x_1,$ unless $\la_2\sim \la_1,$ we shall concentrate on those $ \nu^{\la}_{j}$  for which
\begin{equation}\label{5.29}
\la_1\sim \la_2 \quad \mbox{and} \ 2^{-j}\la_2\sim \la_3.
\end{equation}
\smallskip
In all other cases, we obtain much faster Fourier decay estimates by repeated integrations by parts, so that the corresponding terms  can be considered as error terms.

\smallskip
{\bf 1. Case: $1\le \la_1\le 2^{j}.$}  In this case the phase function has essentially 
no oscillation in the $x_2$ variable. But, by applying van der Corput's lemma 
(or the method of stationary phase) in $x_1$ we obtain in combination with  \eqref{5.29} Êthe estimate
\begin{equation}\label{5.30}
\|\widehat{\nu^{\la}_{j}}\|_\infty\lesssim
\frac{1}{\la_1^{1/2}}.
\end{equation}

\smallskip
{\bf 2. Case: $\la_1>2^{j}.$}  Observe that in this case, our assumptions imply that 
$\de_32^{2j}\la_3\ll\la_3\ll \la_2,$ if  $j\ge j_0\gg1$. Moreover, depending on the signs  of the 
$\xi_i,$ we may have no critical point, or exactly one  non-degenerate critical point, with respect to each of the variables $x_2$ and $x_1.$  So,  integrating by parts, respectively  applying the method of stationary phase in the presence of a critical point,   first in $x_2$ and then in $x_1,$  we obtain 
\begin{equation}\label{5.31}
\|\widehat{\nu^{\la}_{j}}\|_\infty\lesssim
\frac{2^{j/2}}{\la_1}.
\end{equation}

\medskip
Next, we estimate the $L^\infty$-norm of $ \nu^{\la}_{j}.$ To this end, notice that \eqref{5.25} shows that we may change coordinates in \eqref{5.28} by putting $(z_1,z_2):=(y_1, \phi(y_1,y_2,\de,j)).$ Since the Jacobian of this coordinate change is of order $1,$ we thus obtain  that
\begin{eqnarray*}
|\nu^{\la}_{j}(x)|\lesssim\la_1\la_2\la_3\iint \Big|\check\chi_1\Big({\la_1}(x_1-z_1)\Big) \, 
\check\chi_1\Big({\la_3}(x_3-z_2)\Big)
\,\tilde a(z,\de,j)\Big|\,dz_1 \, dz_2,
\end{eqnarray*}
hence 
\begin{equation}\label{5.32}
\|\nu^{\la}_{j}\|_\infty\lesssim \la_2\sim \la_1,
\end{equation}
in Case 1 as well as in Case 2. 
\medskip

For the operators $ T_{\de,j}$ and $ T_j^\la$ which  appear in this subcase, the estimates \eqref{5.30} - \eqref{5.32} thus yield the following bounds:
\begin{eqnarray*}
\| T_j^\la\|_{1\to\infty}\lesssim 
\begin{cases}       \la_1^{-1/2}, & \mbox{  if  } \ 1\le \la_1\le 2^{j},\\
   2^{j/2} \la_1^{-1},&  \mbox{  if  } \ \la_1>2^{j},
\end{cases}
   \end{eqnarray*}
   and 
$
\| T_j^\la\|_{2\to 2}\lesssim \la_1.
$
Interpolating these estimates,  we find that 
\begin{eqnarray}\label{5.33}
\| T_j^\la\|_{p_c\to p_c'}\lesssim 
\begin{cases}       \la_1^{\frac{3\theta-1}2}, & \mbox{  if  } \ 1\le \la_1\le 2^{j},\\
   2^{\frac {1-\theta} 2 j}\la_1^{2\theta-1},&  \mbox{  if  } \ \la_1>2^{j},
\end{cases}
   \end{eqnarray}
where $\theta$ is again given by \eqref{5.21}.

\medskip
Now, in view of \eqref{5.29}, the main contributions to the series \eqref{5.19} comes here  from those dyadic $\la=2^k$ for which $\la_1\sim \la_2 \quad \mbox{and} \ 2^{-j}\la_2\sim \la_3.$ 
Thus, up to an easily bounded error term, 
\begin{equation}\label{5.34}
\| T_{\de,j}\|_{p_c\to p_c'}\lesssim \sum_{\la_1=2}^{2^j} \la_1^{\frac{3\theta-1}2 } +
\sum_{\la_1=2^{j+1}}^\infty 2^{\frac {1-\theta}2 j}\la_1^{2\theta-1}\lesssim 2^{\frac{3\theta-1}2 j}\le 2^{(1-2\th) j} ,
\end{equation}
since, by \eqref{thcond} we have $2\theta-1<0$  and $(3\theta-1)/2\le 1-2\th.$

This verifies the  restriction estimate \eqref{rest4} and thus  
concludes the proof of Proposition \ref{s5.1} also in this subcase.

\smallskip
\subsection[caseAD3]{The situation where $2^{2j}\de_3 \sim1$} Notice that here we can no longer conclude that $x_2\sim 1$ on the support of the amplitude $a(x,\de,j),$ only that $|x_2|\lesssim 1,$ whereas still $x_1\sim 1.$
Observe also that here the cases $A_\infty$ and $D_\infty$ are excluded, since in these cases $\de_3=0.$ 

 Putting $\si:=2^{2j}\de_3,$ and
 $
b^\sharp(x,\de,j):=\tilde b( x_1, 2^{-j}x_2+x_1^m \om(\de_1x_1),\de_1,\de_2),
 $
 we may re-write the complete phase in \eqref{5.27} as
\begin{eqnarray}\nonumber
\Phi(x,\de,j,\xi)&=&\xi_1x_1+\xi_2x_1^m\om(\de_1x_1)+\xi_3\si x_1^n \beta(\de_1 x_1)  \\
&&+2^{-j}\xi_2x_2+\xi_3b^\sharp(x,\de,j)\,x_2^2,\label{5.35}
\end{eqnarray}
where $\si\sim 1$  and $|b^\sharp(x,\de,j)|\sim 1,$
and \eqref{5.28} as
\begin{eqnarray}\nonumber
\nu^{\la}_{j}(x)=\la_1\la_2\la_3&\int \check\chi_1\Big({\la_1}(x_1-y_1)\Big) \, \check\chi_1\Big({\la_2}(x_2-2^{-j}y_2-y_1^m\om(\de_1y_1))\Big)\\
&\check\chi_1\Big({\la_3}(x_3-b^\sharp(y,\de,j)\,y_2^2
-\si y_1^n \beta(\de_1 y_1))\Big)\label{nujla}
\,a(y,\de,j)\,dy.
\end{eqnarray}
Here, we have suppressed the dependence on the parameter $\si$ in order to defray the notation. Observe also that we then may drop the parameter $\de_3$ from  the definition of $\de,$ i.e., we may assume that $\de=(\de_1,\de_2),$ since only $\si$ depends on $\de_3.$  Recall from that \eqref{adef} that $a(y,\de,j)$ is supported where $y_1\sim 1$ and $|y_2|\lesssim 1.$ 

Since $\Big|\int \check\chi_1\Big({\la_3}(c-t^2)\Big)\,dt\Big|\le C \la_3^{-1/2},$ with a constant $C$ which is independent of $c,$ making use of the localizations given for the integration in $y_2$ from the third factor, respectively second,  factor, and then for the integration in $y_1$ by the first factor in the integrand, it is easy to see that
\begin{equation}\label{5.36}
\|\nu^{\la}_{j}\|_\infty\lesssim \min\{\la_2\la_3^{1/2}, 2^j\la_3, \} =\la_3^{1/2}\min\{\la_2,  2^j\la_3^{1/2}\}\,.
\end{equation}


Let us again assume \eqref{5.13}. We shall first integrate in $x_1$ in order to estimate $\widehat {\nu^{\la}_{j}}(\xi).$ If one of the quantities $\la_1,\la_2$  and $\la_3$ is much bigger than the two other ones, we see that we have no critical point on the support of the amplitude, so that the corresponding terms can again be viewed as error terms. Let us therefore assume that all three are of comparable size, or two of them are of comparable size, and the third one is much smaller. We shall begin with the latter situation, and distinguish various possibilities.

\smallskip
{\bf 1. Case: $\la_1\sim \la_3$ and $\la_2\ll \la_1.$}  In this case, we apply the method of stationary phase to the integration in $x_1,$ and subsequently van der Corput's estimate to the $x_2$-integration and obtain
$\|\widehat{\nu^{\la}_{j}}\|_\infty\lesssim \la_1^{-1/2} \la_3^{-1/2} \sim \la_1^{-1}.$

{\bf 1.1. The subcase where $\la_2\le 2^j\la_1^{1/2}.$ } Then, by \eqref{5.36}, $\|\nu^{\la}_{j}\|_\infty\lesssim \la_2\la_1^{1/2},$  and we obtain  in a similar way as before by interpolation that
$$
\| T_j^\la\|_{p_c\to p_c'}\lesssim \la_1^{\frac {3\theta-2}2}\la_2^\theta.
$$
Here, $ \frac {3\theta-2}2<0,$ because of \eqref{thcond}. Note next that if $2^j\la_1^{1/2}\le \la_1,$ i.e., if $\la_1\ge 2^{2j},$ then by our assumptions $\la_2\le 2^j\la_1^{1/2},$ and if $\la_1< 2^{2j},$ then $\la_2\le \la_1.$ We thus find that the contribution  $T_{\de,j}^I$ of the operators $T^{\la}_j$ with $\la$ satisfying the assumptions of  this subcase  to $T_{\de,j}$ can be estimated by
\begin{eqnarray*}
\| T_{\de,j}^I\|_{p_c\to p_c'}&&\lesssim \sum_{\la_1=2}^{2^{2j}} \sum_{\la_2=2}^{\la_1} \la_1^{\frac{3\theta-2}2 }\la_2^{\theta }
 + \sum_{\la_1=2^{2j+1}}^\infty  \sum_{\la_2=2}^{2^j\la_1^{1/2}}   \la_1^{\frac{3\theta-2}2 }\la_2^{\theta }\\
 &&\lesssim \sum_{\la_1=2}^{2^{2j}} \la_1^{\frac{5\theta-2}2 }+ \sum_{\la_1=2^{2j+1}}^\infty 2^{\theta j} \la_1^{2\theta-1}.
\end{eqnarray*}
But, we have seen that $2\theta-1<0,$ so that
$$
\| T_{\de,j}^I\|_{p_c\to p_c'}\lesssim \max\{j, 2^{(5\theta-2)j}\}.
$$
Now, if $5\theta-2>0, $ then, again because of \eqref{thcond}, $\| T_{\de,j}^I\|_{p_c\to p_c'}\lesssim 2^{(5\theta-2)j} \le 2^{(1-2\th)j}. $ And, if $5\theta-2\le0, $ then  $\| T_{\de,j}^I\|_{p_c\to p_c'}\lesssim j\lesssim 2^{(1-2\th)j},$ i.e.,
\begin{equation}\label{5.37}
\| T_{\de,j}^I\|_{p_c\to p_c'}\lesssim  2^{(1-2\th)j}.
\end{equation}

{\bf 1.2. The subcase where  $\la_2> 2^j\la_1^{1/2}.$ }Then, by \eqref{5.36}, $\|\nu^{\la}_{j}\|_\infty\lesssim 2^j\la_1,$  and we obtain  in a similar way as before by interpolation that
$$
\| T_j^\la\|_{p_c\to p_c'}\lesssim 2^{\theta j} \la_1^{2\theta-1}.
$$
Observing that we have $2^j\la_1^{1/2}<\la_2\le \la_1,$ and then also  $\la_1>2^{2j},$  we  see  that the contribution  $T_{\de,j}^{II}$ of the operators $T^{\la}_j$ with $\la$ satisfying the assumptions in this subcase  to $T_{\de,j}$ can be estimated by
\begin{eqnarray*}
\| T_{\de,j}^{II}\|_{p_c\to p_c'}&\lesssim&  2^{\theta j}  \sum_{\la_1=2^{2j}}^\infty \sum_{2^j\la_1^{1/2}<\la_2\le \la_1}  \la_1^{2\theta-1}\lesssim
2^{\theta j} \sum_{\la_1=2^{2j}}^\infty (\log_2 {\la_1}-2j)\,\la_1^{2\theta-1}\\
&\lesssim&2^{\theta j} \sum_{k=2j}^\infty (k-2j)\,2^{(2\theta-1)k} \lesssim 2^{(5\theta-2)j}\sum_{k=0}^\infty k\,2^{(2\theta-1)k}\lesssim 2^{(5\theta-2)j},
\end{eqnarray*}
so that also
\begin{equation}\label{5.38}
\| T_{\de,j}^{II}\|_{p_c\to p_c'}\lesssim 2^{(1-2\th)j}.
\end{equation}

\smallskip
{\bf 2. Case: $\la_2\sim \la_3$ and $\la_1\ll \la_2.$}  Here, we can estimate $\widehat{\nu^{\la}_{j}}$ in the same way as in the previous
case and obtain $\|\widehat{\nu^{\la}_{j}}\|_\infty\lesssim \la_2^{-1/2} \la_3^{-1/2}\sim \la_2^{-1}.$ Moreover, by \eqref{5.36}, 
we have $\|\nu^{\la}_{j}\|_\infty\lesssim \la_2\min\{\la_2^{1/2}, 2^j\}.$ Both these estimates are independent of $\la_1.$ We therefore  consider the sum over all $\nu^\la_j$ such that $\la_1\ll \la_2,$  by putting
 $
\si_j^{\la_2,\la_3}:=\sum_{\la_1\ll \la_2}\nu^\la_j.
$
This means that 
$$
\widehat{\si_j^{\la_2,\la_3}}(\xi)=\chi_0\Big(\frac {\xi_1}{\la_2}\Big)\chi_1\Big(\frac {\xi_2}{\la_2}\Big)\chi_1\Big(\frac {\xi_3}{\la_3}\Big)\, \widehat{\tilde\nu_{\de,j}}(\xi),
$$
where now $\chi_0$ is smooth and compactly supported in an interval $[-\ve,\ve],$ where $\ve>0$ is sufficiently small. In particular, $\si_j^{\la_2,\la_3}(x)$ is given again by the expression \eqref{nujla}, only with the first factor $\check\chi_1\Big({\la_1}(x_1-y_1)\Big)$ in the integrand replaced by $\check\chi_0\Big({\la_2}(x_1-y_1)\Big)$ and $\la_1$ replaced by $\la_2.$  Thus we obtain the same type of estimates 
\begin{equation}\label{5.36b}
\|\widehat{\si_j^{\la_2,\la_3}}\|_\infty\lesssim \la_2^{-1}, \quad \|\si_j^{\la_2,\la_3}\|_\infty\lesssim \la_2\min\{\la_2^{1/2}, 2^j\}.
\end{equation}
Denote by $T_j^{\la_2,\la_3}$ the operator of convolution with $\widehat{\si_j^{\la_2,\la_3}}.$

\smallskip

{\bf 2.1. The subcase where $\la_2\le 2^{2j}.$ } Then  we have $\|\si_j^{\la_2,\la_3}\|_\infty\lesssim \la_2^{3/2},$ and interpolating this  with the first estimates in  \eqref{5.36b}, we obtain
$$
\| T_j^{\la_2,\la_3}\|_{p_c\to p_c'}\lesssim \la_2^{\th-1}\la_2^{\frac 32 \theta}=\la_2^{\frac {5\th -2}2}.
$$
We thus find that the contribution  $T_{\de,j}^{III}$ of the operators $T^{\la}_j$ with $\la$ satisfying the assumptions of this subcase  to $T_{\de,j}$ can be estimated by
\begin{eqnarray*}
\| T_{\de,j}^{III}\|_{p_c\to p_c'}\lesssim \sum_{\la_2=2}^{2^{2j}} \la_2^{\frac {5\th -2}2}.
\end{eqnarray*}
Arguing in a similar way as in  Sub-case 1.1, this implies that 
\begin{equation}\label{5.39}
\| T_{\de,j}^{III}\|_{p_c\to p_c'}\lesssim 2^{(1-2\th)j}.
\end{equation}

 \smallskip
 {\bf 2.2. The subcase where $\la_2> 2^{2j}.$ }Then, by \eqref{5.36b}, $\|\si_j^{\la_2,\la_3}\|_\infty\lesssim 2^j\la_2,$  and we obtain  in a similar way as before by interpolation that
$$
\| T_j^{\la_2,\la_3}\|_{p_c\to p_c'}\lesssim \la_2^{\theta-1} 2^{\theta j} \la_2^\theta=2^{\theta j}\la_2^{2\theta-1},
$$
where, according to \eqref{thcond}, $2\th -1<0.$
We thus find that the contribution  $T_{\de,j}^{IV}$ of the operators $T^{\la}_j$ with $\la$ satisfying the assumptions of this subcase  to $T_{\de,j}$ can be estimated by
\begin{eqnarray*}
\| T_{\de,j}^{IV}\|_{p_c\to p_c'}\lesssim 2^{\theta j}\sum_{\la_2=2^{2j}}^\infty   2^{\theta j}\la_2^{2\theta-1}
 \lesssim 2^{(5\theta-2)j}.
\end{eqnarray*}
As before, this  implies that
\begin{equation}\label{5.40}
\| T_{\de,j}^{IV}\|_{p_c\to p_c'}\lesssim 2^{(1-2\th)j}.
\end{equation}

\smallskip
{\bf 3. Case: $\la_1\sim \la_2$ and $\la_3\ll \la_1.$}  Notice that the phase $\Phi$ has no critical point with respect to $x_2$ when $2^{-j}\la_2\gg \la_3,$ so that we shall concentrate on the case   where $\la_2\lesssim 2^j\la_3.$
Then we can estimate $\widehat{\nu^{\la}_{j}}$ in the same way as in the previous cases and obtain 
\begin{equation}\label{hatnu3}
\|\widehat{\nu^{\la}_{j}}\|_\infty\lesssim \la_1^{-1/2} \la_3^{-1/2}.
\end{equation}

\smallskip

{\bf 3.1. The subcase where $\la_3^{1/2}\gtrsim \la_12^{-j}.$ } Then $(2^{-j}\la_1)^2\lesssim \la_3\ll \la_1\,$ and hence  we may assume that $\la_1\le 2^{2j},$ and  from \eqref{5.36}  and the previous estimate for $\widehat{\nu^{\la}_{j}}$ we obtain by interpolation
$$
\| T_j^\la\|_{p_c\to p_c'}\lesssim \la_1^{\frac {3\theta-1}2}\la_3^{\frac {2\theta-1}2}.
$$
We thus find that the contribution $T_{\de,j}^{V}$ of the operators $T^{\la}_j$ with $\la$ satisfying the assumptions of this subcase  to  the operator $T_{\de,j}$ can be estimated by
\begin{eqnarray*}
\| T_{\de,j}^{V}\|_{p_c\to p_c'}\lesssim \sum_{\la_1=2}^{2^{2j}} \,\sum\limits_{(2^{-j}\la_1)^2\lesssim \la_3\ll \la_1} \la_1^{\frac {3\theta-1}2}\la_3^{\frac {2\theta-1}2}\lesssim 2^{(1-2\th)j}\sum_{\la_1=2}^{2^{2j}}\la_1^{\frac{7\th-3}2}
\end{eqnarray*}
(recall that $2\th-1<0,$ according to  \eqref{thcond}). If $\th<3/7,$ this implies the desired estimate
\begin{equation}\label{5.48I}
\| T_{\de,j}^{V}\|_{p_c\to p_c'}\lesssim 2^{(1-2\th) j}\qquad (\mbox{ if } \th<\frac 37)
\end{equation}

However, if $\th=3/7,$ i.e., if $\phi$ is of finite type $A$ and $m=2,$ we only get the estimate
\begin{equation}\label{compint1}
\| T_{\de,j}^{V}\|_{p_c\to p_c'}\lesssim j2^{(1-2\th) j}.
\end{equation}
 In order to improve on this estimate, we shall have to apply a complex interpolation argument. There will be a few more cases which require such an interpolation argument, and we shall collect all of them in Section \ref{complex2}.   
 We also remark that 
 $$\sum_{2\le \la_1\lesssim 2^{j}} \,\sum\limits_{(2^{-j}\la_1)^2\lesssim \la_3\ll \la_1} \la_1^{\frac {3\theta-1}2}\la_3^{\frac {2\theta-1}2}\lesssim 2^{\frac {3\theta-1}2 j}\lesssim 2^{(1-2\th) j},
 $$ 
 so that we only need to control the terms with $\la_1\gg 2^j.$

\smallskip
{\bf 3.2. The subcase where $\la_3^{1/2}\ll  \la_12^{-j}.$ } Then we have 
$$
\la_3\ll \min\{\la_1,\, (2^{-j}\la_1)^2\},
$$ 
which implies that necessarily $\la_1\gg 2^j,$
and interpolation yields in this case that 
 $$
\|T_{\de,j}^\la\|_{p_c\to p_c'}\lesssim  2^{\th j}\la_1^{-\frac{(1-\theta)}2}\la_3^\frac{3\theta-1}2 .
$$

First, assume that $\la_1>2^{2j}$. Then 
$\la_1=\min\{\la_1,\, (2^{-j}\la_1)^2\},$ so that we shall use that  $\la_3\ll \la_1.$  Denoting by  $T_{\de,j}^{VI,1}$ the sum of the operators $T^{\la}_j$ with $\la$ satisfying the assumptions of this subcase  and  $\la_1>2^{2j},$ and recalling that $3\th-1>0$ and $2\th-1<0,$ we see that 

\begin{equation}\label{5.50I}
\| T_{\de,j}^{VI,1}\|_{p_c\to p_c'}\lesssim 2^{\th j}\sum_{\la_1=2^{2j}}^\infty \sum_{\la_3=2}^{\la_1}\la_1^{-\frac{(1-\theta)}2}\la_3^\frac{3\theta-1}2\lesssim 2^{(5\theta-2)j}\le 2^{(1-2\th)j}.
\end{equation}

There remains the case where $2^j\ll \la_1\le 2^{2j}.$ Then  $\la_3\ll (2^{-j}\la_1)^2$. Denoting by  $T_{\de,j}^{VI,2}$ the sum of the operators $T^{\la}_j$ with $\la$ satisfying the assumptions of this subcase  and  $2^j\ll \la_1\le 2^{2j},$ and recalling that $3\th-1>0$ and $2\th-1<0,$ we see that 

\begin{equation}\label{5.51I}
\| T_{\de,j}^{VI,2}\|_{p_c\to p_c'}\lesssim 2^{\th j}\sum_{\la_1=2^j}^{2^{2j}}\sum_{\la_3=2}^{(2^{-j}\la_1)^2}\la_1^{-\frac{(1-\theta)}2}\la_3^\frac{3\theta-1}2\lesssim 2^{(1-2\th)j}\sum_{\la_1=2}^{2^{2j}}\la_1^{\frac{(7\theta-3)}2}
\lesssim 2^{(1-2\th)j},
\end{equation}
provided $\th<3/7.$  If $\th =3/7,$ we pick up an additional factor $j$ as in \eqref{compint1}:

\begin{equation}\label{compint2}
\| T_{\de,j}^{VI,2}\|_{p_c\to p_c'}\lesssim j 2^{\frac17 j}=j2^{(1-2\th) j}.
\end{equation}
 In order to improve on this estimate, we shall have to apply  again a complex interpolation argument (cf. Section \ref{complex2}).

\medskip

What is left is 

\medskip

{\bf 4. Case: $\la_1\sim \la_2\sim \la_3.$}  We can here  first apply the method of stationary phase to the integration in $x_2.$ This produces a phase function in $x_1,$ which is of the form $\phi_1(x_1)=\xi_1x_1+\xi_2 (\om(0)x_1^m+\mbox{error})+\xi_3 (\si\beta(0)x_1^n+\mbox{error}),$ with small  error terms of order $O(|\de|+2^{-j}).$  We assume again that \eqref{5.13} holds true. Then $\phi_1$ has a singularity of Airy type, which implies that the oscillatory integral with phase $\phi_1$ that we have arrived at decays of order $O(|\la|^{-1/3}).$ Indeed,  we have  $n\ge 2m+1$ and $m\ge 2,$ and since  $x_1\sim 1,$ it is  easy to see by studying the linear system of equations $y_j=\phi_1^{(j)}(x_1), \ j=1,2,3,$  that there exist constants $0<c_1\le c_2$ which do not depend on $\xi$ and $x_1\sim 1$ such that 
$$
c_1|\xi|\le\sum^3_{j=1}|\phi_1^{(j)}(x_1)|\le c_2|\xi|.
$$
Thus, our claim follows from Lemma \ref{corput}. We thus find that 
$$
\|\widehat{\nu^{\la}_{j}}\|_\infty\lesssim \la_1^{-5/6}.
$$

\medskip
{\bf 4.1. The subcase where  $\la_1> 2^{2j}.$ }Then, by \eqref{5.36} $\|\nu^{\la}_{j}\|_\infty\lesssim 2^j \la_1,$ and we obtain 
$$
\| T_j^\la\|_{p_c\to p_c'}\lesssim 2^{\theta j}\la_1^{\frac {11\theta-5}6}.
$$
The estimates in \eqref{thcond} show that $11\theta-5<0,$ which implies that the contribution  $T_{\de,j}^{VII}$ of the operators $T^{\la}_j$ with $\la$ satisfying the assumptions of this subcase  to $T_{\de,j}$ can again be estimated by
\begin{equation}\label{5.53I}
\| T_{\de,j}^{VII}\|_{p_c\to p_c'}\lesssim \sum_{\la_1=2^{2j}}^\infty 2^{\theta j}\la_1^{\frac {11\theta-5}6}\lesssim 
 2^{\frac {14\theta-5}3 j} \lesssim 2^{(1-2\th)j},
\end{equation}
provided that $\th\le 2/5.$ According to \eqref{5.21}, this is true, with the only exception of the case where  $\phi$ is of type 
$A$  and $m=2.$ 

Observe also that if $m=2,$ then $\theta=3/7$ and $p'_c=14/3,$ so that $\| T_j^\la\|_{p_c\to p_c'}\lesssim 2^{3j/7}\la_1^{-1/21},$ and 
$$
\sum_{\la_1>2^{6j}} 2^{\frac{3j}7}\la_1^{-\frac 1{21}}\lesssim 2^{\frac j 7}= 2^{(1-2\th)j}.
$$
 This leaves open  the sum over the  terms with $\la_1\le 2^{6j},$ in the case where $\phi$ is of type A and $m=2.$
\medskip

{\bf 4.2. The subcase where  $\la_1\le 2^{2j}.$ } 
 Then, by \eqref{5.36} $\|\nu^{\la}_{j}\|_\infty\lesssim \la_1^{3/2},$ and we obtain 
$$
\| T_j^\la\|_{p_c\to p_c'}\lesssim \la_1^{\frac {14\theta-5}6}.
$$
We thus find that the contribution  $T_{\de,j}^{VIII}$ of the operators $T^{\la}_j$ with $\la$ satisfying the assumptions of this subcase to $T_{\de,j}$ can be estimated by
\begin{eqnarray*}
\| T_{\de,j}^{VIII}\|_{p_c\to p_c'}\lesssim \sum_{\la_1=2}^{2^{2j}} \la_1^{\frac {14\theta-5}6}.
\end{eqnarray*}
If $14\theta-5\le 0,$ then we immediately obtain the desired estimate $\| T_{\de,j}^{VIII}\|_{p_c\to p_c'}\lesssim j\lesssim 2^{(1-2\th)j},$  so assume that $14\theta-5> 0.$  Then 
$
\| T_{\de,j}^{VIII}\|_{p_c\to p_c'}\lesssim 2^{\frac {14\theta-5}3 j},
$
and arguing as before (compare \eqref{5.53I}), we see that 
\begin{equation}\label{5.54I}
\| T_{\de,j}^{VIII}\|_{p_c\to p_c'}\lesssim 2^{(1-2\th)j},
\end{equation}
unless $\phi$ is of type A and $m=2.$  But, recall that the case $A_\infty$ was excluded here, so that $\phi$ is  of type $A_{n-1}$, with finite $n\ge 5$ (compare Proposition \ref{normalform1}).

\medskip
The estimates \eqref{5.37} - \eqref{5.48I}, \eqref{5.50I} -\eqref{5.51I}, \eqref{5.53I} and \eqref{5.54I}  show that estimate \eqref{rest4} holds true also in the situation of this subsection, which completes the proof of Proposition \ref{s5.1}, with the exception of the case where  $\phi$ is of type $A_{n-1},$ with finite $n\ge 5$ and $m=2,$  in which we still need to improve on the estimates \eqref{compint1} and \eqref{compint2} in the Sub-cases 3.1 and 3.2, and moreover need to find stronger estimates for the cases where $\la_1\sim \la_2\sim \la_3$ when $\la_1\le 2^{6j}.$  Observe also that in   Case 3, we have that $\la_1\sim \la_2,$ and thus we may assume that $\la_2=2^K\la_1,$ where $K$ is from a finite set of integers. This allows to assume that $\la_2=2^K\la_1,$ for a given, fixed integer $K,$ and for the sake of simplicity, we shall even assume that $K=0,$ so that $\la_1=\la_2$  (the other cases can be treated in exactly the same way). In a similar way, we may and shall assume that $\la_1=\la_2=\la_3$ in Case 4. Thus, in order to complete the proof of Proposition \ref{s5.1},  and hence that of Theorem \ref{nonadarestrict} when $\hl(\phi)< 2,$ what remains   to prove is  the following 

\begin{prop}\label{m2_A}
Assume that  $\phi$ is of type $A_{n-1}$, with $m=2$ and  finite $n\ge 5,$  so that $p'_c=14/3$ and $\th:=2/p_c'=3/7.$ Then the following hold true, provided $j,M\in\NN$ are sufficiently large and $\de$ is sufficiently small:
 \begin{itemize}
\item[(a)] Let  
$$
\nu_{\de,j}^{V}:=\sum_{\la_1=2^{M+j}}^{2^{2j}} \sum_{\la_3=(2^{-M-j}\la_1)^2}^{2^{-M} \la_1} \nu^{(\la_1,\la_1,\la_3)}_j,
$$
and denote by $T_{\de,j}^{V}$ the convolution operator $\vp\mapsto \vp* \widehat{\nu_{\de,j}^{V}}.$
Then
\begin{equation}\label{restm3}
\|T_{\de,j}^{V}\|_{\frac{14}{11}\to \frac{14}3}\le C\, 2^{\frac {j}7}.
\end{equation}

\item[(b)] Let  
$$
\nu_{\de,j}^{VI}:=\sum_{\la_1=2^{M+j}}^{2^{2j}} \sum_{\la_3=2}^{(2^{-M-j}\la_1)^2} \nu^{(\la_1,\la_1,\la_3)}_j,
$$
and denote by $T_{\de,j}^{VI}$ the convolution operator $\vp\mapsto \vp* \widehat{\nu_{\de,j}^{VI}}.$
Then
\begin{equation}\label{restm2}
\|T_{\de,j}^{VI}\|_{\frac{14}{11}\to \frac{14}3}\le C\, 2^{\frac {j}7}.
\end{equation}

\item[(c)]
Let  
$$
\nu_{\de,j}^{VII}:=\sum_{\la_1=2}^{2^{6j}}  \nu^{(\la_1,\la_1,\la_1)}_j,
$$
and denote by $T_{\de,j}^{VII}$ the convolution operator $\vp\mapsto \vp* \widehat{\nu_{\de,j}^{VII}}.$
Then
\begin{equation}\label{restm4}
\|T_{\de,j}^{VII}\|_{\frac{14}{11}\to \frac{14}3}\le C\, 2^{\frac {j}7}.
\end{equation}
\end{itemize}
Here, the constant $C$ does neither depend on $\de,$ nor on  $j.$ 
\end{prop}

\begin{remark}\label{rem5.3I}
If $\phi$ is of type $A_{n-1}$ and  $m=2,$ it will  often be convenient in the sequel  to augment our former vector $\de=(\de_1,\de_2)$ by the parameter 
$$\de_0:=2^{-j}\ll 1,
$$ 
i.e., we re-define $\de$  to become   $\de:=(\de_0,\de_1,\de_2).$ Observe that according to  \eqref{btilde} and   \eqref{5.35},
we may then re-write  in \eqref{nujla}  $b^\sharp(y,\de_1,\de_2, j)=b_0(y,\de):=b^a( \de_1y_1,\de_0 \de_2 y_2),$ where  $b^a(y_1,y_2):=b(y_1, y_2+y_1^m \om(y _1))$  expresses $b$ in adapted coordinates. 

Then, by  \eqref{nujla}, we may write 
\begin{eqnarray}\nonumber
\nu^{\la}_{j}(x)=:\nu_\de^\la(x)&=&\la_1\la_2\la_3\int \check\chi_1\Big({\la_1}(x_1-y_1)\Big) \, \check\chi_1\Big({\la_2}(x_2-\de_0y_2-y_1^m\om(\de_1y_1))\Big)\\
&&\check\chi_1\Big({\la_3}(x_3-b_0(y,\de)\,y_2^2
-\si y_1^n \beta(\de_1 y_1))\Big)\label{nujla2}
\,\eta (y,\de)\,dy,
\end{eqnarray}
where $\eta\in C^\infty(\RR^2\times \RR^3)$ is supported where $y_1\sim 1$ and $|y_2|\lesssim 1$ (and, say, $|\de|\le 1$), and where  $\chi_1$ is  a smooth cut-off function supported  near $1.$  Notice that the measure  $\nu_\de^\la$ indeed also depends on $\si\sim1,$ but we shall suppress this dependency in order to defray the notation.
\end{remark}

The proof of  the first two parts of Proposition \ref{m2_A} will be based on a complex interpolation argument, whereas the proof of part (c) will in addition  require  substantially more refined estimations, making use of the fact that $\widehat{\nu^{\la}_{j}}(\xi)$ is large on a small neighborhood of some  ``Airy cone''  only.  

We shall therefore  defer the discussion of our complex interpolation arguments which are needed in order to cover the endpoint cases  to the Sections \ref{complex1} and  \ref{complex2}, and first continue to outline the Airy type analysis which is needed in order to narrow down the large gap between the desired estimate and the actual estimate given by \eqref{5.53I} in the case where $\phi$ is of type A and $m=2.$  
\bigskip

\color{black}

\setcounter{equation}{0}
\section{On Proposition \ref{m2_A}(c): Airy type analysis}\label{propm2_A}

 In order to prove estimate \eqref{restm2} in Proposition \ref{m2_A}, we recall that  $\si\sim 1,$ and that we are assuming that  
$$
2\le \la_1=\la_2=\la_3\le 2^{6j}.
$$  
 In order to defray the notation, we shall in the sequel denote by $\la$ the common value of $\la_1=\la_2=\la_3,$ 
and put  
\begin{equation}\label{sdef}
 s_1:=\frac{\xi_1}{\xi_3}, \, s_2:=\frac{\xi_2}{\xi_3},\, s_3:=\frac{\xi_3}\la,
\end{equation}
so that $|s_1|\sim|s_2|\sim |s_3|\sim  1$ and 
$$
\xi=\la s_3(s_1,s_2,1).
$$
In view of the special role $s_3$ will play, we shall write 
$$
s:=(s_1,s_2,s_3), \quad s':=(s_1,s_2).
$$
Correspondingly, we shall   re-write
\begin{equation}\label{phase1}
\Phi(x,\de_1,\de_2,j,\xi)=\la s_3\tilde\Phi(x,\de,\si,s_1,s_2),
\end{equation}
where  
\begin{eqnarray}\nn
\tilde\Phi(x,\de,\si,s_1,s_2)&:=&s_1x_1+s_2x_1^2\om(\de_1x_1)+\si x_1^n \beta(\de_1 x_1) \\ \label{phit}
&&+\de_0s_2x_2+x_2^2b_0(x,\de).
\end{eqnarray}
Recall also that $\om(0)\ne 0, \, \beta(0)\ne 0,$ and 
$b_0(x,0)=b(0,0)\ne 0.$

According to \eqref{nujla2},  we then have 
$$
\widehat{\nu^\la_j}(\xi)=\widehat{\nu^\la_\de}(\xi)=\chi_1(s_1s_3)\chi_1(s_2s_3)\chi_1(s_3)\; \int e^{-i\la s_3\tilde\Phi(x,\de,\si,s_1,s_2)}\, \tilde a(x,\de)\, dx,
$$
where the amplitude $\tilde a(x,\de):=a(x,\de)\chi_1(x_1) \chi_0(x_2)$  (compare \eqref{nujla2}) is a smooth function of $x$ supported where $x_1\sim 1$ and $|x_2|\lesssim 1,$ whose derivatives are uniformly bounded with respect to the parameters $\de.$    
Moreover, if $T^\la_\de$ denotes the convolution operator
$$
T^\la_\de \vp:=\vp *\widehat{\nu_\de^\la}, 
$$ 
then we see that the estimate \eqref{restm2} can be re-written as 
\begin{equation}\label{restm4b}
\Big\|\sum\limits_{2\le \la\le \de_0^{-6}} T^\la_\de\Big\|_{\frac{14}{11}\to \frac{14}3}\le C\, \de_0^{-\frac 17}
\end{equation}

We shall need to understand the precise behavior of $\widehat{\nu^\la_\de}(\xi).$
To this end,  consider the integration with respect to $x_2$ in the corresponding integral. Notice that there always is a critical point  $x_2^c$ with respect to $x_2.$  Writing $x_2=\de_0s_2y_2,$ and applying the implicit function theorem to $y_2,$ we find  that 
\begin{equation}\label{ai0}
x_2^c=\de_0s_2Y_2(\de_1x_1,\de_2, \de_0s_2),
\end{equation}
where $Y_2$ is smooth and of size $|Y_2|\sim 1.$ Notice also that $Y_2(0,0,0)=-1/(2b(0,0))$ when $\de=0.$
Let us put
\begin{equation}\label{defPsi}
\Psi(x_1,\de,\si, s_1,s_2):=\tilde\Phi(x_1,x_2^c,\si,s_1,s_2)=\tilde\Phi(x_1,\de_0s_2Y_2(\de_1x_1,\de_2, \de_0s_2),\si,s_1,s_2).
\end{equation}

Applying the method of stationary phase with parameters to the $x_2$- integration (see, e.g., \cite{stein-book}) and ignoring the region away from the critical point $x_1^c,$ which leads to better estimates by means of integrations by parts,   we find that we may assume that

\begin{equation}\label{nudehat}
\widehat{\nu_\de^\la}(\xi)= \la^{-1/2}\chi_1(s_1s_3)\chi_1(s_2s_3)\chi_1(s_3) 
 \int e^{-i\la s_3\Psi(y_1,\de,\si, s')}
\, a_0(y_1,s',\de;\,\la)  \,\chi_1(y_1)\, dy_1\, ,
\end{equation}
where  $\chi_1$ is a  smooth cut-off function supported, say, in the interval $[1/2, 2].$ 

Moreover, $a_0(y_1,s',\de;\,\la) $ is  smooth  and {\it uniformly a classical symbol}  of order $0$ with respect to $\la.$  By this we mean that it is  a classical symbol of order zero for every  given parameter (here  these are $y_1,s_1,s_2$ and $\de$), and the constants in the symbol estimates are uniformly controlled for these parameters.  It will be important to observe that this implies that   $\frac {\pa}{\pa \la} a_0(y_1,s',\de;\,\la)$ is even a symbol of order $-2$ with respect to $\la,$ uniformly in $y_1, s',\de$  (the latter property will become relevant later!).

\medskip
We shall  need more precise information on the phase $\Psi.$   Indeed, in the subsequent lemmata,  we shall establish two different presentations of $\Psi,$  both of which will become relevant.

\begin{lemma}\label{Psi1}
For $|x_1|\lesssim1,$ we may write 
$$
\Psi(x_1,\de,\si, s_1,s_2)=s_1x_1+s_2x_1^2\om(\de_1x_1)+\si x_1^n \beta(\de_1 x_1) +(\de_0s_2)^2 Y_3(\de_1x_1,\de_2, \de_0s_2),
$$
where $Y_3$ is smooth and 
$
Y_3(\de_1x_1,\de_2, \de_0s_2)=c_0+O(|\de|),
$
with $c_0:=-1/4b(0,0)\ne 0.$
\end{lemma}
\begin{proof} 
We have 
$$
\Psi(x_1,\de,\si, s_1,s_2)=s_1x_1+s_2x_1^2\om(\de_1x_1)+\si x_1^n \beta(\de_1 x_1) 
+\de_0s_2x^c_2+(x_2^c)^2 b_0(x_1,x_2^c,\de),
$$
so that, by definition, 
$$
Y_3(\de_1x_1,\de_2, \de_0s_2):=Y_2(\de_1x_1,\de_2, \de_0s_2)+Y_2(\de_1x_1,\de_2, \de_0s_2)^2 b_0(x_1,x_2^c,\de)
$$
where  for $\de=0$ we have
$$
Y_3(0,0,0)=Y_2(0,0,0)+Y_2(0,0,0)^2b_0(0,0,0)=-\frac 1{4b(0,0)}\ne 0,
$$
because $Y_2(0,0,0)=-1/(2b(0,0)).$ 
\end{proof} 

Next, we shall verify that $\Psi$  has indeed a singularity of Airy type with respect to the variable $x_1.$ To this end, let us first consider the case where $\de=0.$ Then
$$
\Psi(x_1,0,\si, s_1,s_2):=s_1x_1+s_2x_1^2\om(0)+\si x_1^n\be(0),
$$
and depending again on  the signs of $ s_2\om(0)$ and $\be(0),$  the first  derivative  (with respect to $x_1$)
$$
\Psi'(x_1,0,\si, s_1,s_2)=s_1+2s_2\om(0)x_1+n\si \be(0) x_1^{n-1}
$$  may have a critical point, or not.  If not, 
$\Psi$ will  have at worst  non-degenerate critical points, and this case can be treated again by the method of stationary phase, respectively integrations by parts. We shall therefore concentrate on the case where  $\Psi'$ does have a critical point $x_1^c,$ which will then be given explicitly by
$$
x_1^c=x_1^c(0,\si,s_2):=\Big(-\frac{2\om(0)}{n(n-1)\si\be(0)} \, s_2\Big)^{\frac 1{n-2}}.
$$
Let us assume that $s_2>0$ (the case where it is negative can be treated similarly). By scaling in $x_1,$ we may and shall assume for simplicity that
\begin{equation}\label{normalization1}
-\frac{2\om(0)}{n(n-1)\si\be(0)}=1 \qquad (\mbox {and  } s_2\sim 1).
\end{equation}

Then $x_1^c(0,\si,s_2)= s_2^{\frac 1{n-2}},$ and $|\Psi'''(x_1^c,0,\si,s_1,s_2)|\sim 1.$ Thus, the implicit function theorem  shows that for $\de$ sufficiently small, there is a unique critical point $x_1^c=x_1^c(\de,\si,s_2)$ of $\Psi'$ depending smoothly on $\de,\si$ and $s_2,$ i.e., 
\begin{equation}\label{ai1}
\Psi''(x_1^c(\de,\si,s_2),\de,\si,s_1,s_2)=0.
\end{equation}

\begin{lemma}\label{s11.1}
The phase $\Psi$ given by \eqref{defPsi} can be developed locally  around the critical point $x_1^c$ of $\Psi'$ in the form
$$
\Psi(x_1^c(\de,\si,s_2)+y_1,\de,\si,s_1,s_2)=B_0(s',\de,\si)-B_1(s',\de,\si) y_1+B_3(s_2,\de,\si,y_1) y_1^3,
$$
where $B_0,B_1$ and $B_3$ are smooth functions, and where  $|B_3(s_2,\de,\si,y_1)|\sim 1,$  and indeed 
$$
B_3(s_2,\de,\si,0)=s_2^{\frac{n-3}{n-2}}G_4(s_2,\de,\si),
$$
where $G_4$ is smooth and satisfies
$$
G_4(s_2,0,\si)=\tfrac{n(n-1)(n-2)}{6}\si \be(0).
$$
Moreover, we may write
\begin{equation}\label{ai2}
\begin{cases} \quad  x_1^c(\de,\si,s_2)&=s_2^{\frac 1{n-2}}G_1(s_2,\de,\si),\\
\quad B_0(s',\de,\si)&=s_1s_2^{\frac 1{n-2}}G_1(s_2,\de,\si) - s_2^{\frac n{n-2}} G_2(s_2,\de,\si),\\
 \quad B_1(s',\de,\si)&= -s_1+s_2^{\frac {n-1}{n-2}}G_3(s_2,\de,\si),
\end{cases}
\end{equation}
with smooth functions $G_1, G_2 $ and $G_3$ satisfying 
\begin{equation}\label{ai3}
\begin{cases}  
\quad G_1(s_2,0,\si)&=1,\\
\quad G_2(s_2,0,\si)&= \frac{n^2-n-2}2 \si \be(0),\\
\quad G_3(s_2,0,\si)&= n(n-2)\si\be(0).
\end{cases}
\end{equation}
Notice that all the numbers in \eqref{ai3} are non-zero, since we assume $n\ge 5.$
\end{lemma}

\begin{proof} 
The first statements in \eqref{ai2}, \eqref{ai3} are obvious. Next, by  \eqref{defPsi} and  \eqref{ai0} we have 
\begin{eqnarray*}
B_0(s',\de,\si)&=&\Psi(x_1^c(\de,\si,s_2),\de,\si,s_1,s_2)=s_1s_2^{\frac 1{n-2}}G_1(s_2,\de,\si)\\
&&\hskip-1cm+s_2^{\frac n{n-2}}\Big( G_1(s_2,\de,\si)^2\om(\de_1x_1^c)+\si G_1(s_2,\de,\si)^n \be(\de_1x_1^c)+
\de_0^2s_2^{\frac{n-4}{n-2}}Y_3(\de_1x_1^c,\de_2, \de_0s_2)\Big),
\end{eqnarray*}
where $x_1^c$ is given by the first identity in \eqref{ai2}.  In  combination with \eqref{normalization1}, we thus obtain the second identity in \eqref{ai2} and the third in \eqref{ai3}, because $s_2\sim 1.$

Similarly,
\begin{eqnarray*}
-B_1(s',\de,\si)&=&\Psi'(x_1^c(\de,\si,s_2),\de,\si,s_1,s_2)\\
&=&s_1+ 2s_2x_1^c\om(\de_1x_1^c)+n\si (x_1^c)^{n-1}\be(\de_1 x_1^c)+O(|\de|),
\end{eqnarray*}
which in view of \eqref{normalization1} easily implies the last identities in \eqref{ai2} and \eqref{ai3}.
Finally, when $y_1=0,$  then 
\begin{eqnarray*}
6B_3(s_2,\de,\si,0)&=&\Psi'''(x_1^c(\de,\si,s_2),\de,\si,s_1,s_2)=n(n-1)(n-2)\si \be(0)(x_1^c)^{n-3}+O(|\de|),
\end{eqnarray*}
which shows that $|B_3(s_2,\de,\si,y_1)|\sim 1$ for $|y_1|$ sufficiently small.
\end{proof}

Translating the  coordinate $y_1$ in \eqref{nudehat} by $x_1^c,$ Lemma \ref{s11.1} then allows to re-write
 \eqref{nudehat} also in the following form:
 
\begin{eqnarray}\nn
\widehat{\nu_\de^\la}(\xi)&=& \la^{-1/2}\chi_1(s_1s_3) \chi_1(s_2s_3)\chi_1(s_3) \; e^{-i\la s_3 B_0(s',\de,\si)}\\ 
&&\hskip2cm \int e^{-i\la s_3\Big(B_3(s_2,\de,\si,y_1) y_1^3-B_1(s',\de,\si) y_1\Big)}
\, a_0(y_1,s',\de;\, \la)\,\chi_0(y_1)\, dy_1\, .\label{ai4}
\end{eqnarray}

Here, $\chi_0$ is a  smooth cut-off function supported in sufficiently small neighborhood of the origin, and 
$a_0(y_1,s',\de;\,\la)$  is again a smooth function (possible different from the one in  \eqref{nudehat}), which is uniformly a classical symbol of order $0$ with respect to $\la.$  

\medskip

We shall   make use of the following, more or less classical  lemma, respectively variations  of it,    in the case  of Airy type integrals, i.e., when $B=3$ (compare for instance Lemma 1 in  \cite{randol},   or \cite{duistermaat} for related results). The case of general $B\ge 3$ will become relevant in  \cite{IM-rest2}.  Since we need somewhat more refined results than what can be found in the literature, for instance information on  the asymptotic behavior also under certain  perturbations,  we shall sketch a proof.

\begin{lemma}\label{airy1}
Let $B\ge 3$ be an integer, and let
$$
J(\la,u,s):=\int_\RR e^{i\la( b(t,s)t^B-ut-\sum_{j=2}^{B-1}b_j(u)t^j)} \, a(t,s)\, dt,\quad \la\ge 1, u\in\RR, |u|\lesssim 1,
$$
where $a,b$   are smooth, real-valued functions of $(t,s)$ on an open neighborhood of $I\times K,$ where $I$ is a compact neighborhood of the origin in $\RR$ and $K$ is a compact subset of $\RR^m.$  The functions  $b_j$ are assumed to be real-valued and smooth too. Assume also that $b(t,s)\ne 0$ on $I\times K,$  that $|t|\le \ve$ on the support of $a,$ and that
$$
|b_j(u)|\le C|u|,\qquad j=2,\dots, B-1.
$$
If $\ve>0$ is  chosen sufficiently small and $\la$ sufficiently large, then the following hold true:
 \begin{itemize}
\item[(a)]  If $\la^{(B-1)/B} |u|\lesssim 1,$ then 
$$
J(\la,u,s)=\la^{-\frac 1B} \, g(\la^{\frac {B-1}B}u, \la, s),
$$
where $g(v,\la,s)$ is a smooth function of $(v,\la,s)$ whose derivates of any order are uniformly bounded on its natural domain.
\item[(b)]  If $\la^{(B-1)/B} |u|\gg 1,$ let us assume first that $u$ and $b$ have the same sign, and  that  $B$ is odd.  Then 
\begin{eqnarray*}
&&J(\la,u,s)=\la^{-\frac 12} |u|^{-\frac {B-2}{2B-2}}\, \chi_0\Big(\frac u\ve\Big)\\
&& \times\Big(a_+(|u|^{\frac 1{B-1}},s)\, e^{i\la |u|^{\frac B{B-1}}q_+(|u|^{\frac1{B-1}},s)}+ a_{-}(|u|^{\frac1{B-1}},s) \,e^{i\la |u|^{\frac B{B-1}}q_{-}(|u|^{\frac1{B-1}},s)}\Big)\\
 &&\hskip 4cm +(\la |u|)^{-1} E(\la |u|^{\frac B{B-1}}, |u|^{\frac1{B-1}},s),
\end{eqnarray*}
where $a_\pm,q_\pm$ are smooth functions, and where  $E$ is smooth and satisfies estimates
$$
|\pa_\mu^\al \pa_v^\be\pa_s^\ga E(\mu,v,s)|\le C_{N,\al,\be,\ga} |v|^{-\be} |\mu|^{-N}, \qquad \forall N,\al,\be,\ga\in\NN.
$$Moreover, when $|u|$ is sufficiently small, then
$$
q_\pm(v,s)=\mp \sgn {b(0,s)}|b(0,s)|^{\frac1{B-1}}\rho(v,s),
$$
where $\rho$ is smooth and $\rho(0,s)=(B-1)\cdot B^{-B/(B-1)}.$

Finally, if $u$ and $b$ have opposite signs, then  the same formula remains valid, even with $a_+\equiv 0, a_-\equiv 0.$ And, if $B$ is even, we do have a similar result, but without the presence of the term containing $a_-.$

 \end{itemize}
\end{lemma}

\begin{proof} In the case (a), scaling in $t$ by the factor $\la^{-1/B}$ allows to re-write
$$
J(\la,u,s)=\la^{-\frac1B}\int e^{i( b(\la^{-\frac 1B}t,s)t^B-\la^{\frac{B-1}B}ut -\sum_{j=2}^{B-1}\la^{\frac{B-j}B}b_j(u)t^j)} \, a(\la^{-\frac1B}t,s)\, dt.
$$
Choose a smooth  cut-off function $\chi_0$ on $\RR$ which is identically one on $[-1,1],$ and $M\gg 1,$  and decompose  
$$
\la^{\frac1B}J(\la,u,s)=G_0(\la^{\frac{B-1}B}u,\la,s)+G_\infty(\la^{\frac{B-1}B}u,\la,s),
$$
where, for $|v|\lesssim 1,$ 
\begin{eqnarray*}
G_0(v,\la,s)&:=&\int e^{i( b(\la^{-\frac1B}t,s)t^B-vt-\sum_{j=2}^{B-1}\la^{\frac{B-j}B}b_j(\la^{\frac{1-B}B}v)t^j)} \,\chi_0(\tfrac t M) a(\la^{-\frac1B}t,s)\, dt,\\
G_\infty(v,\la,s)&:=&\int e^{i( b(\la^{-\frac1B}t,s)t^B-vt-\sum_{j=2}^{B-1}\la^{\frac{B-j}B}b_j(\la^{\frac{1-B}B}v)t^j)}  \,(1-\chi_0(\tfrac t M))\, a(\la^{-\frac1B}t,s)\, dt.
\end{eqnarray*}

Notice that for $j\ge2,$ 
$$
|\la^{\frac{B-j}B}b_j(\la^{\frac{1-B}B}v)|\le C \la^{\frac{B-j}B}\la^{\frac{1-B}B}|v|\lesssim \la^{-\frac 1B}.
$$
It is  then easy to see that $G_0$ is a smooth function of $(v,\la,s)$ whose derivates of any order are uniformly bounded on its natural domain, and the same can easily be  verified for $G_\infty$ by means of iterated integrations by parts. This proves (a).

\medskip

In order to prove (b),  consider first the case where $|u|\ge \ve.$ If $\Phi=\Phi(t)$ denotes the complete phase in the oscillatory integral defining $J(\la,u,s),$ recalling that $|t|\le \ve,$  we easily   see that 
$$
|\Phi'(t)|\ge C \la|u|,
$$ 
provided we choose $\ve $ sufficiently small. Integrations by parts then show that we can represent $J(\la,u,s)$ by the third term $(\la |u|)^{-1} E(\la |u|^{\frac B{B-1}}, |u|^{\frac1{B-1}},\la,s).$

\medskip
Let us therefore assume that $|u|<\ve.$ We shall also  assume that $u>0;$ the case $u<0$ can be treated in a similar way. Here, we scale t by the factor $u^{1/(B-1)},$  and re-write
$$
J(\la,u,s)=u^{\frac1{B-1}}\int e^{i\la u^{\frac B{B-1)}}( b(u^{\frac1{B-1}}t,s)t^B-t-\sum_{j=2}^{B-1} u^{-\frac{B-j}{B-1}}b_j(u)t^j)} \, a(u^{\frac1{B-1}}t,s)\, dt.
$$
Again, we decompose this as 
$$
J(\la,u,s)=J_0(\la,u^{\frac1{B-1}},s)+J_\infty(\la,u^{\frac1{B-1}},s),
$$
where, with $v:=u^{\frac1{B-1}},$ 
\begin{eqnarray*}
J_0(\la,v,s)&:=&v\int e^{i\la v^B( b(vt,s)t^B-t-\sum_{j=2}^{B-1} v^{-(B-j)}b_j(v^{B-1})t^j)} \, \chi_0(\tfrac t M)\, a(vt,s)\, dt,\\
J_\infty(\la,v,s)&:=&v\int e^{i\la v^B( b(vt,s)t^B-t-\sum_{j=2}^{B-1} v^{-(B-j)}b_j(v^{B-1})t^j)}  \, (1-\chi_0(\tfrac t M))\, a(vt,s)\, dt.
\end{eqnarray*}
Observe that 
$$
|v^{-(B-j)}b_j(v^{B-1})|\le C v^{j-1}\lesssim \ve^{\frac 1{B-1}}, \qquad j=2,\dots, B-1.
$$
Assume that $\ve$ is sufficiently small.  If $B$ is odd, then, in the first integral $J_0,$  the phase has exactly two non-degenerate critical points $t_\pm(v,s)\sim \pm 1,$  if $b>0,$ and thus the method of stationary phase shows that 
$$
J_0(\la,v,s)=v(\la v^B)^{-\frac12}a_+(v,s) e^{i\la v^Bq_+(v,s)}+v(\la v^B)^{-\frac12}a_{-}(v,s) e^{i\la v^Bq_{-}(v,s)} +
vE_1(\la v^B,v,s),
$$
where $a_\pm$ are smooth functions, and where $E_1$ is smooth and rapidly decaying with respect to the first variable.  If $b<0,$ then there are no critical points, and we get  the term $E_1$ only. Moreover, 
$$
q_\pm(v,s)=b(vt_\pm(v,s),s) t_\pm(v,s)^B-t_\pm(v,s)+O(v).
$$
Note that if $v=0,$ then $t_\pm(0,s)=\pm (Bb(0,s))^{-1/(B-1)},$ so that 
$$
q_\pm(0,s)=\mp\frac {B-1}{B^{\frac B{B-1}} b(0,s)^{\frac1{B-1}}}\ne 0,
$$
which proves the statement about $q_\pm.$ A similar discussion applies when  $B$ is even. In this case, there is only  one critical point, namely  $t_+(v,s).$ 

In the second integral $J_\infty,$ we may apply integrations by parts in order to re-write it as 
$$
J_\infty(\la,v,s):=v(\la v^B)^{-N} \int e^{i\la v^B( b(vt,s)t^B-t-\sum_{j=2}^{B-1} v^{-(B-j)}b_j(v^{B-1})t^j)} \,  a_N(t,v,s)\, dt, \quad N\in\NN,
$$
where $a_N$ is supported where $|t|\ge M$ and $|a_N(t,v,s)|\le C_N |t|^{-2N}.$ Similarly, if we take derivatives with respect to $s,$ we produce additional  powers of $t$ in the integrand, which, however, can be compensated by integrations by parts. Analogous considerations apply to derivatives with respect to $v$ (where we produce negative powers of $v$), and with respect to $\la v^B.$ Altogether, we find that 
$$
J_\infty(\la,v,s) =\frac 1{\la v^{B-1}}  E_2(\la v^B, v,s),
$$
where $E_2$ is smooth and 
$$
|\pa_\mu^\al \pa_v^\be\pa_s^\ga E_2(\mu,v,s)|\le C_{N,\al,\be,\ga} |v|^{-\be} |\mu|^{-N}, \qquad \forall N,\al,\be,\ga\in\NN.
$$
Summing up all terms, and putting $E:=E_1+E_2,$ we obtain the statements in (b).
\end{proof}

The following remark can be verified easily by  well-known versions of the method of stationary phase for oscillatory integrals whose amplitude depends also on the parameter $\la$ as symbols of order $0$ (see, e.g., \cite{sogge}). 
\begin{remark}\label{airy1rem}
We may even allow in Lemma \ref{airy1} that the function $a(t,s)$ also depends on $\la,$ i.e., $a=a(t,s;\,\la),$ in such a way that it is a symbol of order $0$ in $\la,$ uniformly in the other parameters, i.e., 
$$
|\Big(\frac{\pa}{\pa \la}\Big)^\al \Big(\frac{\pa}{\pa t}\Big)^{\beta_1}  \Big(\frac{\pa}{\pa s}\Big)^{\beta_2}a(t,s;\,\la)|\le C_{\al,\beta} (1+\la)^{-\al}
$$
for all $\al, \beta_1,\beta_2\in \NN.$  Then the same conclusions hold, only with $a_{\pm}$ and $E$ depending also additionally on $\la$ as symbols of order $0$ in a uniform way.
\end{remark}

Let us apply this lemma and the remark to the oscillatory integral \eqref{ai4}, with $B=3.$ 
Putting $u:=B_1(s,\de,\si),$ in view of this lemma we shall decompose the frequency support of $\nu_\de^\la$  furthermore into the domain where $\la^{2/3}|B_1(s,\de,\si)|\lesssim 1$ (this is  essentially a conic region in $\xi$-space (cf. \eqref{sdef}), which will be  called  the "region near the  Airy cone"),  and the remaining domain into the conic regions where $(2^{-l}\la)^{2/3}|B_1(s,\de,\si)|\sim 1,$ for 
$
M_0\le  2^l\le \frac \la {M_1}
$
where $M_0,M_1\in\NN$ are sufficiently  large. The  {\it Airy cone}   is given by the equation $B_1=0,$ i.e., 
$$
s_1=s_2^{\frac {n-1}{n-2}}G_3(s_2,\de,\si).
$$

 To this end, we choose smooth cut-off functions $\chi_0$ and $\chi_1$ such that   $\chi_0=1$ on a sufficiently large neighborhood of the origin, and 
$\chi_1(t)$ is supported  where $|t|\sim 1$ and $\sum_{l\in \bZ} \chi_1(2^{-2l/3})=1$ on $\RR\setminus \{0\},$ and define the functions $\nu_{\de,Ai}^\la$ and $\nu_{\de,l}^\la$ by 
\begin{eqnarray*}  
\widehat{\nu_{\de,Ai}^\la}(\xi)&:=&\chi_0\Big(\la^{\frac23}B_1(s',\de,\si)\Big)\, \widehat{\nu_{\de}^\la}(\xi),\\ 
 \widehat{\nu_{\de,l}^\la}(\xi)&:=&\chi_1\Big((2^{-l}\la)^{\frac23}B_1(s',\de,\si)\Big)\, \widehat{\nu_{\de}^\la}(\xi), \qquad M_0\le  2^l\le \frac \la {M_1}, 
\end{eqnarray*}
so that 
\begin{equation}\label{ai5}
\nu_{\de}^\la= \nu_{\de,Ai}^\la+\sum_{M_0\le  2^l\le \frac \la {M_1}}\nu_{\de,l}^\la.
\end{equation}
 Denote by $T^\la_{\de,Ai}$ and $T^\la_{\de,l}$ the convolution operators
 $$
 T_{\de,Ai}^\la \vp:=\vp * \widehat{\nu_{\de,Ai}^\la},\quad T_{\de,l}^\la \vp:=\vp * \widehat{\nu_{\de,l}^\la}.
 $$
 
Since $\de_0=2^{-j},$ we note that in order to prove Proposition \ref{m2_A}, it will suffice to prove the following estimate:

If $p_c:=14/13,$   then 
\begin{equation}\label{restm}
\sum\limits_{2\le \la\le \de_0^{-6}}\|T^\la_{\de,Ai}\|_{p_c\to p_c'}+\Big\|\sum_{M_0\le  2^l\le \frac \la {M_1}}\sum\limits_{2\le \la\le \de_0^{-6}} T^\la_{\de,l}\Big\|_{p_c\to p_c'}
\le C\, \de_0^{-\frac 17},
\end{equation}
provided $\de$ is sufficiently small and $M_0, M_1\in\NN$  are sufficiently large.

\subsection{Estimation of $T^\la_{\de,Ai}$ }

We first consider the region near the Airy cone and  prove the following 
\begin{lemma}\label{s11.3}
There are constants $C_1,C_2$ so that 
\begin{eqnarray}\label{ai6}
\|\widehat{\nu_{\de,Ai}^\la}\|_\infty&\le& C_1 \la^{-\frac 56},\\
\|\nu_{\de,Ai}^\la\|_\infty&\le& C_2 \la^{\frac 76},\label{ai7}
\end{eqnarray}
uniformly in  $\si$ and $\de,$
 provided $\la$ is sufficiently large and $\de$ sufficiently small.
\end{lemma}
Notice that by interpolation (again with $\theta=3/7$)  these estimates imply that 
$$
\|T^\la_{\de,Ai}\|_{p_c\to p_c'}\lesssim (\la^{-\frac 56})^{\frac 47}( \la^{\frac 76})^{\frac37}=\la^{\frac 1 {42}},
$$
so that 
\begin{equation}\label{ai8}
\sum\limits_{2\le \la\le \de_0^{-6}} \|T^\la_{\de,Ai}\|_{p_c\to p_c'}\lesssim \de_0^{-\frac 17},
\end{equation}
which is exactly the estimate that we need  (cf. \eqref{restm}).

\medskip
Let us turn to the proof of Lemma \ref{s11.3}.
The first estimate \eqref{ai6}   is immediate from \eqref{ai4} and Lemma \ref{airy1}.

In order to prove the second estimate,  observe first that  by Lemma \ref{airy1} (a) and the subsequent remark, we may write 
\begin{eqnarray*}
\chi_0(\la^{2/3}B_1(s,\de,\si))\int e^{-i\la s_3\Big(B_3(s_2,\de,\si,y_1) y_1^3-B_1(s',\de,\si) y_1\Big)}
\, a(y_1,s,\de;\,\la)\,\chi_0(y_1)\, dy_1\\
=\la^{-\frac 13}\,\chi_0(\la^{2/3}B_1(s',\de,\si))\, g\Big(\la^{2/3}|B_1(s',\de,\si)|,\la,\de, \si,s\Big),
\end{eqnarray*}
 where $g$ is a smooth function whose derivates of any order are uniformly bounded on its natural domain.

Applying  the Fourier inversion formula to $\nu_{\de,Ai}^\la,$ \eqref{ai4} and this identity  yield that 
\begin{eqnarray*}
\nu_{\de,Ai}^\la(x)=   \iint   \la^{-\frac 12}\, \la^{-\frac 13}\,\chi_0(\la^{2/3}B_1(s',\de,\si))\, \chi_1(s_1s_3)
\chi_1(s_2s_3)\chi_1(s_3) \; e^{i\xi\cdot x}\\
e^{-i\la s_3 B_0(s',\de,\si)} \,  g\Big(\la^{2/3}B_1(s',\de,\si),\xi_3,\de, \si,s\Big)\, d\xi.
\end{eqnarray*}
We again change coordinates from $\xi=(\xi_1,\xi_2,\xi_3)$ to $(s_1, s_2,s_3)$  according to \eqref{sdef}. 

We  then  find  that 
\begin{eqnarray}\nn \label{ai9}
\nu_{\de,Ai}^\la(x)&=&\la^{\frac {13}6}\int  e^{-i\la s_3\Big( B_0(s',\de,\si) -s_1x_1-s_2x_2- x_3\Big)}
\chi_0(\la^{2/3}B_1(s', \de,\si))\\
&&\hskip 4cm      g\Big(\la^{2/3}B_1(s',\de,\si),\la,\de, \si,s\Big) \tilde \chi_1(s)\, ds_1 ds_2 ds_3,
\end{eqnarray}
where 
$$
\tilde \chi_1(s):=\chi_1(s_1s_3) 
\chi_1(s_2s_3) \chi_1(s_3)\, s_3^{2}
$$
localizes to a region where 
$
s_j\sim 1, \quad j=1,2,3.
$

Observe first that when $|x|\gg 1,$ then we easily obtain by means of integrations by parts that
\begin{equation}\label{aii9}
|\nu_{\de,Ai}^\la(x)|\le C_N \la^{-N}, \quad N\in\NN, \mbox { if } |x|\gg 1.
\end{equation}
Indeed, when $|x_1|\gg 1,$ then we integrate by parts repeatedly in $s_1$ to see this, and a similar argument applies when $|x_2|\gg1,$ where we use the $s_2$-integration. Observe that in each step, we gain a factor $
\la^{-1},$ and lose at most $\la^{2/3}.$  Finally,  when $|x_1|+|x_2|\lesssim 1$ and $|x_3|\gg 1,$ then we can integrate by parts in $s_3$ in order to establish this estimate.
\medskip

We may therefore assume now that $|x|\lesssim 1.$

We then perform yet  another change of coordinates, passing from $s'=(s_1,s_2)$ to $(z,s_2),$ where 
$$
z:= \la^{2/3}B_1(s',\de,\si).
$$ 
Applying  \eqref{ai2}, we find that
$$
z=\la^{\frac 23}(-s_1+s_2^{\frac {n-1}{n-2}}G_3(s_2,\de,\si))
$$
so that
\begin{equation}\label{ai10}
s_1=s_2^{\frac {n-1}{n-2}}G_3(s_2,\de,\si)-\la^{-\frac 23}z.
\end{equation}
 In combination with \eqref{ai2}, we thus obtain that
\begin{equation}\label{ai11}
B_0(s,\de,\si)=-\la^{-\frac 23}z\,s_2^{\frac {1}{n-2}}G_1(s_2,\de,\si)+s_2^{\frac {n}{n-2}}
(G_1G_3-G_2)(s_2,\de,\si).
\end{equation}
We may thus re-write 
\begin{eqnarray}\nn \label{ai12}
\nu_{\de,Ai}^\la(x)&=&\la^{\frac {3}2}\int  e^{-i\la s_3\Phi(z,s_2,x_1,\de,\si)}
\, g\Big(z,\la,\de, \si,s_2^{\frac {n-1}{n-2}}G_3(s_2,\de,\si)-\la^{-\frac 23}z,s_2,s_3\Big)\\
&& \hskip3cm  \tilde \chi_1 \Big(s_2^{\frac {n-1}{n-2}}G_3(s_2,\de,\si)-\la^{-\frac 23}z,s_2\Big)\,   
  \, \chi_0(z)\, dz  ds_2 ds_3,
\end{eqnarray}
where
\begin{eqnarray}\nn
\Phi(z,s_2,x_1,\de,\si):=s_2^{\frac {n}{n-2}}(G_1G_3-G_2)(s_2,\de,\si)
- s_2^{\frac {n-1}{n-2}}G_3(s_2,\de,\si)x_1-s_2x_2-x_3\\
+\la^{-\frac 23}z\,(x_1-s_2^{\frac {1}{n-2}}G_1(s_2,\de,\si)). \label{ai13}
\end{eqnarray}

Observe that by \eqref{ai2},  when $\de=0,$ 
\begin{equation}\label{ai14}
(G_1G_3-G_2)(s_2,\de,\si)=\tfrac{n^2-3n+2}2 \si \be(0)\ne 0, \quad  G_3(s_2,\de,\si)= n(n-2)\si\be(0)\ne 0,\end{equation}
since we assume that $n\ge 5,$ and that the exponents $n/(n-2), (n-1)/(n-2)$ and $1$ of $s_2$ which  appear in $\Phi$ (regarding the last term in \eqref{ai13} as an error term) are all different. Moreover, recall that $|x|\lesssim 1.$  It is then  easily seen that this  implies that, when $\de=0,$ 
$$
\sum_{j=1}^3|\pa_{s_2}^j\Phi(z,s_2,x_1,\de,\si)|\sim 1\quad \mbox{for every } s_2\sim 1,
$$
uniformly in $z$ and $\si.$  The same type of estimates then remains valid for $\de$ sufficiently small. We may thus apply the van der Corput type Lemma \ref{corput} to  the $s_2$- integration in \eqref{ai12}, which  in combination with Fubini's theorem yields
$$
\|\nu_{\de,Ai}^\la\|_\infty\le C \la^{\frac 32}\la^{-\frac 13},
$$ hence \eqref{ai7}. This concludes the proof of Lemma \ref{s11.3}.

\subsection{Estimation of $T^\la_{\de,l}$ }

We next regard the region away from  the Airy cone. The study of this region will require substantially more refined techniques.  Let us first note that by \eqref{nudehat} and Fourier inversion we have
\begin{eqnarray}\nn
\nu_{\de,l}^\la(x)&= &\la^3\la^{-1/2}\iint \chi_1(s_1s_3)\chi_1(s_2s_3)\chi_1(s_3) \, \chi_1\Big((2^{-l}\la)^{2/3}B_1(s,\de,\si)\Big)\\
&&\hskip2cm \times e^{-i\la s_3\Big(\Psi(y_1,\de,\si, s_1,s_2)-s_1x_1-s_2x_2-x_3\Big)}
\, a(y_1,\de,\si,s;\, \la)\,\chi_1(y_1)\, dy_1\, ds . \label{nude}
\end{eqnarray}

In order to  indicate the problems that we have to face here,  let us state (without proof) an analogue to Lemma \ref{s11.3}, which we believe  gives essentially optimal estimates (a proof will  implicitly be  contained in the more refined estimates of the next section).

\begin{lemma}\label{s11.4}
There is a constant $C$ so that 
\begin{eqnarray}\label{ai15}
\|\widehat{\nu_{\de,l}^\la}\|_\infty&\le& C 2^{-\frac l 6} \la^{-\frac 56},\\
\|\nu_{\de,l}^\la\|_\infty&\le& C\min\{\la^{\frac 76} 2^{\frac l 3},\frac \la{\de_0}\},\label{ai16}
\end{eqnarray}
uniformly in $\si$ and $\de,$ provided $\de$  is sufficiently small.
\end{lemma}

In order to apply this lemma, let us put $\la=2^r,\,r\in\NN.$ Then, according to \eqref{ai15}, we have 
\begin{equation}\label{6.39I}
\|\widehat{\nu_{\de,l}^\la}\|_\infty\lesssim 2^{-\frac{5r+l}{6}},
\end{equation}
For $k\in\NN$ we therefore define
$$
\nu_{\de,k}:=\sum_{I_k}\nu_{\de,l}^{2^r},
$$
where $I_k:=\{(r,l)\in\NN^2: 5r+l=k, 2^r\le \de_0^{-6}\}.$   Then 
\begin{equation}\label{c6.7}
\sum_{M_0\le  2^l\le \frac \la {M_1}}\, \sum_{2\le \la\le \de_0^{-6}} \nu_{\de,l}^\la=\sum_{k\in\NN}\nu_{\de,k},
\end{equation}
and we have the following consequence of Lemma \ref{s11.4}:
\begin{cor}\label{cai}
There is a  constant $C$ so that 
\begin{eqnarray}\label{6.41I}
\|\widehat{\nu_{\de,k}}\|_\infty&\le& C\,2^{-\frac k 6};\\
\|\nu_{\de,k}\|_\infty&\le& C\,  2^{\frac 29 k}\,\de_0^{-\frac 13},   \label{6.42I}
\end{eqnarray}
uniformly in $\si$ and $\de,$ provided  $\de$ sufficiently small.
\end{cor}

\begin{proof} 
The first estimate \eqref{6.41I} follows immediately from \eqref{6.39I}, because the supports of the functions 
$\{\widehat{\nu_{\de,l}^{2^r}}\}_{r,l}$ are essentially disjoint.

Next, we decompose $I_k=I_k^1\cup I_k^2,$ where
\begin{eqnarray*}
I_k^1&:=&\{(r,l)\in\NN^2: 5r+l=k, \, 2^{r+2l}\le  \de_0^{-6}\},\\
I_k^2&:=&\{(r,l)\in\NN^2: 5r+l=k,\,  \de_0^{-6}< 2^{r+2l},\,   2^r\le \de_0^{-6}\} .
\end{eqnarray*}
Notice that according to \eqref{ai16}, for $(r,l)\in I_k^1$ we have $\la^{\frac 76} 2^{\frac l 3}\le\frac \la{\de_0},$ hence $\|\nu_{\de}^\la\|_\infty\lesssim 2^{7r/6} 2^{l/3}=2^{2k/9} 2^{(r+2l)/18},$ whereas for $(r,l)\in I_k^2$ we have 
$\|\nu_{\de}^\la\|_\infty\lesssim 2^r /\de_0=(2^{2k/9}/\de_0) 2^{-(r+2l)/9},$ so that
$$
\|\nu_{\de,k}\|_\infty\le C\, 2^{\frac 29 k} \sum\limits_{(r,l)\in  I_k^1}  2^{\frac{r+2l}{18}}
+\frac{2^{\frac 29 k}}{\de_0}\sum\limits_{(r,l)\in  I_k^2}2^{-\frac{r+2l}9}.
$$
Comparing the latter sums with one-dimensional geometric series and using that $2^{r+2l}\le  \de_0^{-6}$ in the first sum, and $2^{r+2l}>\de_0^{-6}$ in the second sum, we obtain \eqref{6.42I}.
\end{proof}

Let us denote by $T_{\de,k}$ the convolution operator $\vp\mapsto \vp *\widehat{\nu_{\de,k}}.$ Interpolating the  estimates in the preceding lemma,  again with parameter $\th_c:= 3/7,$ we obtain
$$
\|T_{\de,k}\|_{p_c\to p_c'} \lesssim \de_0^{-\frac 17},
$$
uniformly in $k, $  whereas for $1\le p<p_c$ we get  $\|T_{\de,k}\|_{p\to p'} \lesssim 2^{-\ve k}\de_0^{-\frac 17},$ 
for some $\ve>0$ which depends on $p,$ so that by  \eqref{ai5}, \eqref{ai8} and \eqref{c6.7} 
$$
\Big\|\sum\limits_{2\le \la\le \de_0^{-6}} T^\la_\de\Big\|_{p\to p'}\lesssim  \de_0^{-\frac 17}+ \sum\limits_{k\in\NN} \|T_{\de,k}\|_{p\to p'}
\lesssim \de_0^{-\frac 17}.
$$
 We thus barely fail to establish the estimate \eqref{restm} at the critical exponent $p=p_c.$ 
 
 \medskip In order to prove the estimate \eqref{restm}  also at the endpoint  $p=p_c,$  we need to apply an interpolation argument. Recall that for the Fourier restriction to spheres, the endpoint result had been obtained by Stein using interpolation with analytic families of distributions (cf. \cite{tomas}), and this has become one of  the  standard methods for obtaining endpoint estimates. However, an alternative, real interpolation method has been devised by Bak and Seeger recently in \cite{bs}, which often leads to much shorter proofs and  even optimal results  
  in the scale  of Lorentz spaces.
    
  In \cite{IM-rest2}, we shall make use of this new method in some cases.  Nevertheless, we  shall also encounter further situations which apparently cannot be studied by means of this real interpolation method, but still can by treated by using  complex interpolation. 
  
The  latter applies also to the proof of the endpoint estimate in Proposition \ref{m2_A} (c).   Indeed, what  seems to prevent the application of the real interpolation method is that the (complex) measures $\nu_{\de,k}$ arise from the positive measure $\nu_\de$ by means of spectral localizations to certain frequency regions, i.e., $\nu_{\de,k}=\nu_\de*\psi_{\de,k},$ and the obstacle in applying the method from \cite{bs} 
 is that there is no uniform bound for the $L^1$-norms of the functions $\psi_{\de,k}$ as $k$ tends to infinity.

The proofs based on complex interpolation are technically  involved, and our arguments outlined in the next section can be viewed as prototypical for other  proofs of this kind in  \cite{IM-rest2}.

\setcounter{equation}{0}
\section{The endpoint in Proposition \ref{m2_A} (c): Complex interpolation}\label{complex1}

We keep the notation of the previous section. According to \eqref{ai4} and Lemma \ref{s11.1} we may write 
(recalling that $\xi=\la s_3(s_1,s_2,1)$)
\begin{equation}\label{nulhat}
\widehat{\nu_{\de,l}^\la}(\xi):=\la^{-\frac12} \chi_1\Big((2^{-l}\la)^{\frac23}B_1(s',\,\de,\sigma)\Big)\tilde\chi_1(s)\, e^{-i\la s_3B_0(s',\de,\si)}\,J(\la,s,\de,\sigma),
\end{equation}
where we recall that $\tilde \chi_1$ localizes to a region where 
$
s_j\sim 1,\  j=1,2,3,
$
and where 
$$
J(\la,s,\de,\sigma):=\int e^{-i\la s_3 \tilde\Psi_0(y_1,\de,\si, s_1,s_2) }\,a_0(y_1,s,\de;\,\la)\,\chi_0(y_1)\,dy_1,
$$
with
$$
\tilde\Psi_0(y_1,\de,\si, s_1,s_2) :=B_3(s_2,\de,\sigma,y_1)y_1^3-B_1(s',\de,\si)y_1.
$$
\smallskip
Since $B_1$ is of size $(2^l/\la)^{2/3},$ we scale by the factor $(2^l/\la)^{1/3}$ in the integral defining $J(\la,s,\de,\sigma)$ by putting $y_1=(2^l/\la)^{1/3} u_1,$ and obtain 
$$
J(\la,s,\de,\sigma)=(2^l\la^{-1})^\frac13\int e^{-is_3 2^l\Psi_0(u_1,s',\de,\la,l)}\,a_0\Big((2^l\la^{-1})^\frac13u_1,\,s,\,\de,\,\la\Big)\,\chi_0\Big((2^l\la^{-1})^\frac13u_1\Big)\,du_1,
$$
with
$$
\Psi_0(u_1,s',\de,\la,l):=B_3\Big(s_2,\,\de,\sigma,(2^l\la^{-1})^\frac13u_1\Big)\,u_1^3-(2^l\la^{-1})^{-\frac23}B_1(s',\,\de,\sigma)\,u_1.
$$
Observe that  the coefficients of $u_1$ and of $u_1^3$ in $\Psi_0$ are both of size $1,$ so that  
$\Psi_0$ will have no critical point with respect to $u_1$ unless $|u_1|\sim 1.$  

We  may therefore choose  a smooth cut-off function $\chi_1\in C_0^\infty(\RR)$ supported away from $0$  so that $\Psi_0$ has no critical point outside the support of $\chi_1,$ and  decompose
$$
J:=J(\la,s,\de,\sigma)=J_1+J_\infty,
$$
where $J_1=J_1(\la,s,\de,\sigma)$ is given by
$$
J_1:=(2^l\la^{-1})^\frac13\int e^{-is_3 2^l\Psi_0(u_1,s',\de,\la,l)}\,a_0\Big((2^l\la^{-1})^\frac13u_1,\,s,\,\de,\,\la\Big)\,\chi_0\Big((2^l\la^{-1})^\frac13u_1\Big)\,\chi_1(u_1)\,du_1,
$$
Accordingly, we decompose
$$
\nu_{\de,l}^\la=\nu_{l,1}^\la+\nu_{l,\infty}^\la,
$$
where the summands are defined by 
\begin{eqnarray*}
\widehat{\nu_{l,1}^\la}(\xi)&:=&\la^{-\frac12} \chi_1\Big((2^{-l}\la)^{\frac23}B_1(s',\,\de,\sigma)\Big)\tilde\chi_1(s)\, e^{-i\la s_3B_0(s',\de,\si)}\,J_1(\la,s,\de,\sigma),\\
\widehat{\nu_{l,\infty}^\la}(\xi)&:=&\la^{-\frac12} \chi_1\Big((2^{-l}\la)^{\frac23}B_1(s',\,\de,\sigma)\Big)\tilde\chi_1(s)\, e^{-i\la s_3B_0(s',\de,\si)}\,J_\infty(\la,s,\de,\sigma)
\end{eqnarray*}
(we have dropped the dependence on $\de$ in  order to defray the notation).

\medskip
Let us first  consider the contribution given by the $\nu_{l,\infty}^\la:$
By means of integrations by parts, we easily  obtain  that for every $N\in\NN$ we have 
$
|J_\infty|\lesssim (2^l\la^{-1})^\frac13 2^{-lN},
$
hence 
\begin{equation}\label{6.29I}
\|\widehat{\nu_{l,\,\infty}^\la}\|_\infty \lesssim \la^{-\frac12}(2^l\la^{-1})^\frac13 2^{-lN} \qquad \forall N\in\NN.
\end{equation}
Next,  we may assume that we have chosen $\tilde\chi_1$ so that the Fourier inversion formula reads
$$
\nu_{l,\infty}^\la(x)=\la^3\int_{\RR^3} e^{i\la s_3(s_1x_1+s_2x_2+x_3)} \widehat{\nu_{l,\,\infty}^\la}(\xi)\, ds
$$
(with $\xi=\la s_3(s_1,s_2,1)$). 
We then use the change of variables  from $s'=(s_1,s_2)$ to $(z,s_2),$ where now
$$
z:= (2^{-l}\la)^{\frac 23}B_1(s',\de,\si),
$$ 
and  find that (compare \eqref{ai10})
\begin{equation}\label{ai20}
s_1=s_2^{\frac {n-1}{n-2}}G_3(s_2,\de,\si)-(2^{-l}\la)^{-\frac 23}z,
\end{equation}
and in particular
 \begin{equation}\label{ai21}
B_0(s,\de,\si)=-(2^{-l}\la)^{-\frac 23}z\,s_2^{\frac {1}{n-2}}G_1(s_2,\de,\si)+s_2^{\frac {n}{n-2}}
(G_1G_3-G_2)(s_2,\de,\si).
\end{equation}
Notice that  $2^{l}/\la\le 1/M_1\ll 1.$   And, if we plug in the previous formula for 
$\widehat{\nu_{l,\infty}^\la}$ and write  $\nu_{l,\infty}^\la(x)$ as an oscillatory with respect to the variables $u_1, z, s_2,s_3,$ we see that the complete phase is of the form 
$$
-\la s_3\Big( s_2^{\frac {n}{n-2}}(G_1G_3-G_2)(s_2,\de,\si)- x_1 s_2^{\frac {n-1}{n-2}}G_3(s_2,\de,\si)- s_2x_2-x_3+O(2^{l}\la^{-1}(1+|u_1|^3))\Big),
$$
where according to \eqref{ai14} $|G_1G_3-G_2|\sim 1.$ Observe that the localization given by $\chi_0$ implies that  $2^{l}\la^{-1}|u_1|^3\ll 1.$ Applying again first $N$ integrations by parts with respect to $u_1,$ and then van der Corput's lemma for the integration in $s_2,$ taking into account also the Jacobian  of our change of coordinates to $z,$ we thus see that 
$$
\|\nu_{l,\,\infty}^{\la}\|_\infty \lesssim \la^3  \, \la^{-\frac12}(2^l\la^{-1})^\frac13 2^{-lN}\, 
(2^{-l}\la)^{-\frac 23} \, \la^{-\frac 13}=\la^\frac76 2^{-l(N-1)}.
$$

Interpolating between this estimate and \eqref{6.29I}, we see that the convolution operator $T_{l,\infty}^\la,$ which maps $\vp$ to $\vp* \widehat{\nu_{l,\,\infty}^\la},$ can be estimated by 

$$
\|T_{l,\,\infty}^\la\|_{p_c\to p_c'}\lesssim \la^{-\frac56\frac47+\frac76\frac37}2^{-l}
=\la^{\frac 1{42}}2^{-l},
$$
if we choose $N=2.$ 
This implies the desired estimate
$$
\sum_{M_0\le  2^l\le \frac \la {M_1}}\sum\limits_{2\le \la\le \de_0^{-6}}\|T_{l,\,\infty}^\la\|_{p_c\to p_c'}\lesssim \de_0^{-\frac17}.
$$

\medskip

\subsection{The operators $T_{l,1}^\la$} \label{section7.1}We now turn to the investigation of the  convolution operator $T_{l,1}^\la,$ which maps $\vp$ to $\vp* \widehat{\nu_{l,1}^\la}.$ According to \eqref{restm}, what we need to prove is that the operator 
$$
T_1:=\sum_{M_0\le  2^l\le \frac \la {M_1}}\sum\limits_{2\le \la\le \de_0^{-6}} T^\la_{l,1}
$$
satisfies 
\begin{equation}\label{7.5I}
\|T_1\|_{p_c\to p_c'}\lesssim \de_0^{-\frac 17},
\end{equation}
 with a bound which is independent of $\de$ and $\si.$ 

\medskip
 Now, if  the phase $\Psi_0$ has no critical point on the support of $\chi_1,$ then we  can estimate $J_1$ in the same way as $J_\infty$ before, and can  handle the  operators $T_{l,1}^\la$ as we did  for  the $T_{l,\infty}^\la.$ Let us therefore assume in the sequel that $\Psi_0$ has a critical  point $u_1^c\in \supp \chi_1,$   so that $|u_1|\sim 1.$ 

Applying the method of stationary phase, we then get 
$
|J_1|\lesssim (2^l\la^{-1})^{1/3} 2^{-l /2},
$
hence 
\begin{equation}\label{7.6I}
\|\widehat{\nu_{l,1}^\la}\|_\infty \lesssim \la^{-\frac12}(2^l\la^{-1})^\frac13 2^{-\frac l2} =\la^{-\frac 56} 2^{-\frac l6}=2^{-\frac{k}6},
\end{equation}
where we use the same abbreviations $\la:=2^r,\, k=k(r,l):=5r+l$ as in the previous section
(compare with \eqref{ai15}).

\medskip

In view of this estimate, we define for $\ze$ in the complex strip $\Sigma:=\{\zeta\in \bC: 0\le \Re \zeta\le 1\}$  the following analytic family of measures
$$
\mu_\ze(x):=\gamma(\ze)\,\de_0^{\frac \ze 3}\sum_{M_0\le  2^l\le \frac {2^r}{M_1}}\sum\limits_{2\le 2^r\le \de_0^{-6}} 2^\frac{k(3-7\ze)}{18} \nu_{l,\,1}^{2^r},
$$
where 
$$
\ga(\ze):=\frac{2^{\frac 72(\ze-1)}-1}{2^{-2}-1},
$$
and denote by $T_\ze$  the operator of convolution with $\widehat{\mu_\ze}.$ 
Observe that  for $\ze=\theta_c=3/7,$ we have 
$
T_{\theta_c}=\de_0^{\frac 17}T_1,
$
 so that  by Stein's interpolation theorem \cite{stein-weiss},  \eqref{7.5I}  will follow if we can prove the following estimates on the boundaries of the  strip $\Sigma:$ $ \|T_{it}\|_{1\to \infty} \le C$ and $ \|T_{1+it}\|_{2\to 2} \le C,$ where the constant $C$ is independent of $t\in\RR$ and the  parameters $\de, \si$  (provided $\de$ is sufficiently small). Equivalently, we shall prove
 that
\begin{eqnarray}\label{7.7I}
\|\widehat{\mu_{it}}\|_\infty &\le& C \qquad \forall t\in\RR,\\ 
\|{\mu_{1+it}}\|_\infty &\le& C \qquad \forall t\in\RR. \label{7.8I}
\end{eqnarray}

Since the supports of the functions  $\{\widehat{\nu_{l,1}^{2^r}}\}$ are almost disjoint for $l,r$ in the given range, we see that the first estimate \eqref{7.7I} is an immediate consequence of \eqref{7.6I}.

\medskip
The main problem will consist in estimating  $\|\mu_{1+it}\|_\infty.$  To this end, observe that, again by Fourier inversion, we have (with $\xi=\la s_3(s_1,s_2,1)$)
$$
\nu_{l,1}^\la(x)=\la^3\int_{\RR^3} e^{i\la s_3(s_1x_1+s_2x_2+x_3)} \widehat{\nu_{l,1}^\la}(\xi)\, ds.
$$

Using once again the change of variables  from $s_1$ to $z,$  so that $z= (2^{-l}\la)^{\frac 23}B_1(s',\de,\si)$ and $s_1=s_2^{\frac {n-1}{n-2}}G_3(s_2,\de,\si)-(2^{-l}\la)^{-\frac 23}z,$
we   find that (compare \eqref{ai20}, \eqref{ai21})
\begin{eqnarray}\nonumber
\nu_{l,1}^\la(x)&=&\la^\frac322^l \int e^{-is_3\Phi_1(x,u_1,z,s_2,\de,\la,l)}\,a\Big((2^l\la^{-1})^\frac13 u_1,z,s_2,s_3,\de;\,\la\Big)\\
&&\hskip2cm\times\,  \chi_1(u_1)\chi_1(z)\chi_1(s_2) \chi_1(s_3) \,du_1dz ds_2ds_3, \label{measure1l}
\end{eqnarray}
where $\Phi_1=\Phi_1(x,u_1,z,s_2,\de,\la,l)$ is given by
\begin{eqnarray}\nonumber
\Phi_1&:=&2^l\Big(B_3\big(s_2,\,\de,\sigma,(2^l\la^{-1})^\frac13u_1\big)\,u_1^3-z\,u_1\Big)
\\ \label{7.10I}
&+&\la \Big( s_2^{\frac {n}{n-2}}(G_1G_3-G_2)(s_2,\de,\si)-  s_2^{\frac {n-1}{n-2}}G_3(s_2,\de,\si)\,x_1- s_2x_2-x_3\Big)\\ \nonumber
&+&\la (2^l\la^{-1})^\frac23 z\, \Big(x_1-s_2^\frac1{n-2}G_1(s_2,\,\de,\,\sigma)\Big).\nonumber
\end{eqnarray}

Moreover, 
 $a(v,u_1,z,s_2,s_3,\de;\,\la)$  is a smooth function  which is  uniformly a classical  symbol of order $0$ with respect to $\la.$  
 
 Notice that, in order to defray the notation,  we have suppressed here the dependence on $\si,$ which we shall do so also in the sequel.
 
\subsection{Preliminary reductions}\label{ss7.2}

 Assume now first that $|x|\gg 1.$ If $|x_1|\ll |(x_2,x_3)|,$ then we easily see by means of integrations by parts in \eqref{measure1l} with  respect to the variables $s_2$ or  $s_3$  that $|\nu_{l,1}(x)|\lesssim \la^{-N}$ for every $N\in\NN,$ and if $|x_1|\gtrsim |(x_2,x_3)|,$ then we  easily obtain $|\nu_{l,1}(x)|\lesssim (\la (2^l\la^{-1})^\frac23 )^{-N},$  by means of integrations by parts in $z.$ Since $2^l\le \la,$  it follows easily that there are constants $A\ge 1$ and $C$ such that 
 $
 \sup_{|x|\ge A}\sup_{t\in\RR} |\mu_{1+it}(x)|\le C,
 $
 uniformly in $\de$ and $\si.$

  \medskip
From now  on we shall therefore  assume that $|x|\le A.$  For such $x$ fixed, we decompose the support of $\chi_1(s_2)$  into the subset $L_{II}$  of all $s_2$ such that
$$
\ve (2^l\la^{-1})^\frac13 <|x_1-s_2^\frac1{n-2}G_1(s_2,\,\de,\,\sigma)|<\frac 1\ve (2^l\la^{-1})^\frac13, 
$$
and the  complementary subsets  $L_I$  where  $|x_1-s_2^\frac1{n-2}G_1(s_2,\,\de,\,\sigma)|\ge  (2^l\la^{-1})^\frac13/\ve $ and $L_{III}$ where $|x_1-s_2^\frac1{n-2}G_1(s_2,\,\de,\,\sigma)|\le \ve (2^l\la^{-1})^\frac13 .$ Here, $\ve>0$ will be a sufficiently small fixed number. 

 If we restrict the set of integration in \eqref{measure1l}  to these  subsets with respect to the variable $s_2,$  we obtain corresponding measures $ \nu_{l,I}^\la,\nu_{l,II}^\la$ and $\nu_{l,III}^\la$ into which $\nu_{l,1}^\la$ decomposes, i.e., 
$$
\nu_{l,1}^\la=\nu_{l,I}^\la+\nu_{l,II}^\la+\nu_{l,III}^\la.
$$

Observe  also that $|\la (2^l\la^{-1})^\frac23 \,(x_1-s_2^\frac1{n-2}G_1(s_2,\,\de,\,\sigma))|\ge 2^l$ if and only if  
$|x_1-s_2^\frac1{n-2}G_1(s_2,\,\de,\,\sigma)|\ge  (2^l\la^{-1})^\frac13.$

Thus, if $s_2\in L_I,$  the last term in \eqref{7.10I} becomes dominant as a function of $z,$ provided we choose $\ve $ sufficiently small.  Consequently, the phase has no critical point as a function of $z,$  and 
 applying  $N$  integrations by parts in $z,$ we may estimates
 \begin{eqnarray*}
|\nu_{l,I}^\la(x)|\lesssim \la^{\frac 32}2^l \int\limits_{\{s_2:\la^{\frac 13}2^{\frac{2l}3}\,|x_1-s_2^{\frac {1}{n-2}}G_1(s_2,\de,\si)|\ge C 2^l, \, |s_2|\sim 1\}}
\frac{ds_2} {\Big(\la^{\frac 13}2^{\frac{2l}3}\,|x_1-s_2^{\frac {1}{n-2}}G_1(s_2,\de,\si)|\big)^{N}} \\
\lesssim  \la^{\frac 32}2^l \int\limits_{\la^{\frac 13}2^{\frac{2l}3}\,|v|\ge C 2^l}
\frac {dv}{(\la^{\frac 13}2^{\frac{2l}3}\,|v|)^{N}}\lesssim  \la^{\frac 32} 2^l \,  (\la^{\frac 13}2^{\frac{2l}3})^{-1} 2^{(1-N)l}
= \la^{\frac 76}2^{(\frac 43-N)l} .
\end{eqnarray*}
Similarly, if $s_2\in L_{III},$  the first  term in \eqref{7.10I} becomes dominant as a function of $z,$ and thus $N$  integrations by parts in $z$ and the fact that the $s_2$-integral is restricted to a set of size $(2^l\la^{-1})^\frac13$
 yield the same estimate
 $$
 |\nu_{l,III}^\la(x)|\lesssim \la^{\frac 32}2^l 2^{-Nl}   (2^l\la^{-1})^\frac13=  \la^{\frac 76}2^{(\frac 43-N)l} .
 $$
 This implies  the desired estimate
\begin{eqnarray*}
&&\Big|\gamma(1+it)\,\de_0^{\frac {1+it} 3}\sum_{M_0\le  2^l\le \frac {2^r} {M_1}}\sum\limits_{2\le 2^r\le \de_0^{-6}} 2^\frac{k(3-7(1+it))}{18} (\nu_{l,I}^{2^r}+\nu_{l,III}^{2^r})(x)\Big|\\
&&\lesssim \de_0^\frac13\sum_{M_0\le  2^l\le \frac \la {M_1},\, 1\ll \la<\de_0^{-6}} 2^{-\frac{2k}9}
(|\nu_{l,I}^\la(x)|+|\nu_{l,III}^\la(x)|)\\
 &&\lesssim \de_0^\frac13\sum_{M_0\le  2^l,\, \la<\de_0^{-6}} \la^\frac1{18}2^{(\frac{10}9-N)l}\lesssim1,
\end{eqnarray*}
if we choose $N\ge 2.$

\subsection{The  region where $|x_1-s_2^\frac1{n-2}G_1(s_2,\,\de,\,\sigma)|\sim (2^l\la^{-1})^\frac13$}
We are thus left with the measures $\nu_{l,II}^\la(x)$ and the corresponding family  of measures
$$
\mu^{II}_{1+it}(x):=\gamma(1+it)\,\de_0^{\frac {1+it} 3}\sum_{M_0\le  2^l\le \frac {2^r}{M_1}}\sum\limits_{2\le 2^r\le \de_0^{-6}} 2^\frac{-k(4+7it)}{18} \nu_{l,II}^{2^r}.
$$
In order to establish the estimate \eqref{7.8I}, we still need  prove that there is constant $C$ such that 
\begin{equation}\label{mu-II-est}
|\mu^{II}_{1+it}(x)|\le C,
\end{equation}
where $C$ is independent of $t,x,\de$ and $\si.$  
Note that $\pa_{s_2} (s_2^\frac1{n-2}G_1(s_2,\,\de,\,\sigma))\sim 1$ because $s_2\sim1$ and $G_1(s_2,0,\sigma)=1 $ (compare \eqref{ai3}). Therefore, the relation $|x_1-s_2^\frac1{n-2}G_1(s_2,\,\de,\,\sigma)|\sim (2^l\la^{-1})^\frac13$ can be re-written as $|s_2-\tilde G_1(x_1,\,\de,\,\sigma)| \sim (2^l\la^{-1})^\frac13,$ where $\tilde G_1$ is again a smooth function.
If we write  
$$s_2=(2^l\la^{-1})^\frac13v+\tilde G_1(x_1,\,\de,\,\sigma),
$$
 then this means that $|v|\sim 1.$ We  shall therefore change variables  from $s_2$ to $v$ in the sequel. 
  
In these new variables, the phase function $\Phi_1=\Phi_1(x,u_1,z,s_2,\de,\la,l)$  is given by 
$$\Phi_2(x,u_1,z,v,\de,\la,l):=
\Phi_1\Big(x,u_1,z, (2^l\la^{-1})^\frac13v+\tilde G_1(x_1,\de,\sigma),\de,\la,l\Big).
$$
This is a function of the form
\begin{eqnarray*}
\Phi_{2}&=&2^l\Big(\tilde B_3((2^l\la^{-1})^\frac13u_1,\,(2^l\la^{-1})^\frac13v,x_1,\de,\sigma) u_1^3- z(u_1- H(v,x_1,(2^l\la^{-1})^\frac13,\de,\sigma)\Big)\\
&+&R(v,x,\de,\la),
\end{eqnarray*}
where $\tilde B_3,H$  and $R$ are smooth, and where  $R(v,x,\de,\la)$ is the sum of all  terms not depending on $u_1$ and $z.$  Moreover, $|\tilde B_3|\sim 1.$ Note  also that $u_1\sim1,\, |v|\sim1.$  More precisely,  after this change of variables, $\nu_{l,II}^\la(x)$ assumes the form
\begin{eqnarray}\nonumber
\nu_{l,II}^\la(x)&=&\la^\frac76  2^{\frac 43 l} \int e^{-is_3\Phi_2(x,u_1,z,v,\de,\la,l)}\,a\Big((2^l\la^{-1})^\frac13 u_1,z,v,x_1,s_3,\de;\, \la\Big)\\
&&\hskip2cm\times\,  \chi_1(u_1)\chi_1(z)\tilde\chi_1(v) \chi_1(s_3) \,du_1dz dvds_3, \label{7.11I},
\end{eqnarray}
where $a$ is again smooth and uniformly a classical symbol  of order $0$ with respect to $\la$  (in order to defray our notation, we shall here and in the sequel usually denote such symbols by the letter $a$, even if they may be different  from one instance of occurance to another). Moreover, $\tilde\chi_1(v)$ is smooth and supported in a region where $|v|\sim1.$
\medskip

Observe next  that the function $\Phi_{2}$ has at worst a non-degenerate critical point  $(u_1^c,z^c)$ with respect to the variables  $(u_1,z),$  and that   the  Hessian matrix at such a   point is of the form 
$2^l\left(
\begin{array}{ccc}
 \al &   -1&   \\
  -1& 0  &   \\
\end{array}
\right),$
where $|\al|\lesssim 1,$  so that in particular the Hessian determinant is of size $2^{2l.}$ If there is no critical point, we can again integrate by parts and obtain estimates which are stronger than needed,  so let us assume that there is a critical point. We may than apply the method of stationary phase to the  integration in the variables $(u_1,z).$ This leads to the following  new expression for $\nu_{l,II}^\la(x):$
\begin{eqnarray}\label{7.12I}
&&\nu_{l,II}^\la=\la^\frac76  2^{\frac l3 } \mu_l^\la, \qquad\mbox{with}\\
\mu_l^\la(x)&:=&  \int e^{-i\la s_3\Phi_3(x,v,\de,\la,l)}\,a\Big((2^l\la^{-1})^\frac13,v,x_1,s_3,\de;\, \la, 2^l\Big)
 \tilde\chi_1(v) \chi_1(s_3) \,dvds_3, \nonumber
\end{eqnarray} 
up to an error term which is of order $\la^\frac76 2^{-2l/3}$   and which will therefore  be ignored (compare the discussion  in Subsection \ref{ss7.2}). Here, $a$ is again a  smooth  function  which is uniformly  a  classical symbol of order $0$ with respect to each of the last two variables. 
Moreover,  the phase is given by $\Phi_3(x,v,\de,\la,l):=(1/\la)\Phi_2(x,u^c_1,z^c,v,\de,\la,l).$ 

Notice that \eqref{7.12I} does already imply the first estimate in \eqref{ai16}.

\medskip
In order to compute $\Phi_3(x,v,\de,\la,l)$ more explicitly, observe first that the value of a  function at a critical point is invariant under changes of coordinates. Since we had switched from the coordinates $(y_1,s_1),$  in which $\Phi_1$ is given by the function
$$
\Phi_0(x,y_1,s_1,s_2,\de,\la):=\Psi(y_1,\de,\si,s_1,s_2) -s_1x_1-s_2x_2-x_3
$$
(compare \eqref{nude})
to the coordinates $(u_1,z),$ this means that we can also write
$$
\Phi_3(x,v,\de,\la,l)=\Phi_0(x,y^c_1,s^c_1,s_2,\de,\la),
$$
where $(y^c_1,s^c_1)$ denotes the critical point of $\Phi_0$ with respect to the variables $(y_1,s_1).$  This formula turns out to be better suited, since, according to Lemma \ref{Psi1}, we may write
\begin{eqnarray*}
\Phi_0(x,y_1,s_1,s_2,\de,\la)&=&s_1y_1+s_2y_1^2\om(\de_1y_1)+\si y_1^n \beta(\de_1 y_1) \\
&+&(\de_0s_2)^2 Y_3(\de_1y_1,\de_2, \de_0s_2)-s_1x_1-s_2x_2-x_3
\end{eqnarray*}

To this phase, we can apply the following lemma (with $\xi:=y_1,\eta:=s_1, $ and $\zeta=x_1$), whose proof is straight- forward:
\begin{lemma}\label{phiatcri}
Let $\phi=\phi(\xi,\eta)$ be a smooth, real  function on $\RR^2,$  of the form 
$$
\phi(\xi,\eta)=\xi\eta+f(\xi)-\eta \zeta,
$$
with $\zeta\in\RR.$  Then $\phi$ has a unique critical point  given by $(\xi^c,\eta^c):=(\zeta,-f'(\zeta),$ and 
then $\phi(\xi^c,\eta^c)=f(\zeta).$
\end{lemma}

This yields 
$$
\Phi_3(x,v,\de,\la,l)=s_2x_1^2\om(\de_1x_1)+\si x_1^n \beta(\de_1 x_1) 
+(\de_0s_2)^2 Y_3(\de_1x_1,\de_2, \de_0s_2)-s_2x_2-x_3,
$$
and, passing back to the coordinate $v$ in place of $s_2,$ we obtain
\begin{eqnarray}\nonumber
\Phi_3(x,v,\de,\la,l)&=&\Big((2^l\la^{-1})^\frac13v+\tilde G_1(x_1,\,\de,\,\sigma)\Big) \,x_1^2\om(\de_1x_1)+\si x_1^n \beta(\de_1 x_1) \\ \nonumber
&\phantom{}&\hskip-2cm +\,\de_0^2\Big((2^l\la^{-1})^\frac13v+\tilde G_1(x_1,\,\de,\,\sigma)\Big)^2 Y_3\Big(\de_1x_1,\de_2, \de_0((2^l\la^{-1})^\frac13v+\tilde G_1(x_1,\,\de,\,\sigma))\Big)\\ \nonumber
&\phantom{}&\hskip-2cm -\,\Big((2^l\la^{-1})^\frac13v+\tilde G_1(x_1,\,\de,\,\sigma)\Big)x_2-x_3.\nonumber
\end{eqnarray}
Expanding this with respect to $ (2^l\la^{-1})^\frac13 v,$  we see that $\Phi_3$ is of the form
\begin{eqnarray}\nonumber
\Phi_3(x,v,\de,\la,l)&=& \tilde B_0(x,\de,\si)+(2^l\la^{-1})^\frac13\tilde B_1(x,\de,\si)  v\\
&+&\de_0^2 (2^l\la^{-1})^\frac 23 \tilde B_2(x,\de_0((2^l\la^{-1})^\frac13v,\de,\si)\, v^2 , \label{7.13I}
\end{eqnarray}
with smooth function $\tilde B_j(x,\de,\si),$ and where $|\tilde B_2(x,\de_0((2^l\la^{-1})^\frac13v,\de,\si)|\sim 1.$  Recall also that $|v|\sim 1,$ and notice that when $\de_0=0,$ then $\Phi_3$ is a quadratic polynomial in $v,$ and thus, for  $\de_0$ sufficiently small, $\Phi_3$ is a small perturbation of this quadratic polynomial.

Observe that if $\la \de_0^2 (2^l\la^{-1})^\frac 23\gg 1,$ then we can apply van der Corput's lemma in $v$, which yields the estimate 
\begin{equation}\label{nulII}
|\nu_{l,II}^\la(x)|\lesssim \la^\frac76  2^{\frac l3 } \Big(\la \de_0^2 (2^l\la^{-1})^\frac 23\Big)^{-\frac 12}=\frac \la{\de_0}
\end{equation}
(notice that this verifies  the second estimate in \eqref{ai16}!).

We shall therefore distinguish between the cases where $\la 2^{2l}\lesssim \de_0^{-6},$ and where $\la 2^{2l}\gg \de_0^{-6}.$

Observe also that $2^{\frac76 r}  2^{\frac l3 } 2^{\frac{-k4}{18}}=2^{\frac{r+2l}{18}}.$ It will therefore be convenient  to  put $m:=r+2l,$ so that $r=m-2l.$ We may then re-write (compare \eqref{7.12I})

$$
\mu^{II}_{1+it}(x)=\gamma(1+it)\,\de_0^{\frac {1+it} 3}\sum_{M_2\le  2^m\le \ve_1\de_0^{-18}}2^{\frac{m}{18}(1-35it)} \sum\limits_{\max\{M_0,\de_0^3 2^{\frac m 2}\}\le 2^l \le \ve_0 2^{\frac m3}} 2^{it\frac{7}{2}l}\, \mu_l^{2^{m-2l}}(x),
$$
where $M_2:=M_1 M_0^3, \ve_1:=M_1^{-2}$ and $\ve_0:=M_1^{-1/3}.$  Notice also that the condition $\la 2^{2l}\lesssim \de_0^{-6}$ then reads as $2^m\lesssim \de_0^{-6}.$ We shall therefore decompose 
\begin{equation}\label{7.15I}
\mu^{II}_{1+it}=\mu^{II,1}_{1+it}+\mu^{II,2}_{1+it},
\end{equation}
where
\begin{eqnarray*}
\mu^{II,1}_{1+it}(x)&:=&\gamma(1+it)\,\de_0^{\frac {1+it} 3}\sum_{M_2\le  2^m\le M_0^2\de_0^{-6}}2^{\frac{m}{18}(1-35it)} \sum\limits_{M_0 \le 2^l \le \ve_0 2^{\frac m3}} 2^{it\frac{7}{2}l}\, f_{m,x}(2^l)\\
\mu^{II,2}_{1+it}(x)&:=&\gamma(1+it)\,\de_0^{\frac {1+it} 3}\sum_{M_0^2\de_0^{-6}<  2^m\le \ve_1\de_0^{-18}}2^{\frac{m}{18}(1-35it)} \sum\limits_{\de_0^3 2^{\frac m 2}\le 2^l \le \ve_0 2^{\frac m3}} 2^{it\frac{7}{2}l}\, f_{m,x}(2^l).
\end{eqnarray*}
where we have written  $f_{m,x}(2^l):=\mu_l^{2^{m-2l}}(x).$ Recall from  \eqref{7.12I},\eqref{7.13I} that 
\begin{eqnarray}\nonumber
f_{m,x}(2^l)&=& \int e^{-i s_3 \tilde\Phi_3(x,v,\de,m,2^l)}\,a\Big(2^{l-\frac m3},v,x_1,s_3,\de;\, 2^{m-2l},  2^l\Big)
\tilde\chi_1(v) \chi_1(s_3) \,dvds_3, \\ \label{7.16I}
\phantom{} && \\ 
\tilde\Phi_3&:=& 
 2^{m-2l}\tilde B_0(x,\de,\si)+2^{-l}2^{\frac {2m}3} \tilde B_1(x,\de,\si)  v+
 \de_0^2 2^{\frac m3}\tilde B_2(x,\de_02^{-\frac m3} 2^lv,\de,\si)\, v^2. \nonumber
\end{eqnarray}

\medskip
In several cases the summation in $l$    will require the  use of some  cancellation properties. The following simple lemma, which will also turn out to be useful in many other situations, exploits such cancellations:

\begin{lemma}\label{simplesum}
Let $Q=\prod_{j=1}^nÊ[-R_k,R_k]\subset \RR^n$ be a compact cuboid, with $R_k>0, k=1,\dots, n,$ and let $H$ be a $C^1$-function on an open neighborhood of $Q.$ Moreover, let $\al, \be^1, \dots, \be^n\in\RR^\times$  be given. For any given real numbers $a_1,\dots,a_n\in \RR^\times$ and $M\in \NN$ we then put 
\begin{equation}\label{Ft1}
F(t):= \sum_{l=0}^{M} 2^{i \al lt}
(H\chi_Q)\Big(2^{\be^1l}a_1,\dots, 2^{\be^n l}a_n\Big).
\end{equation}
Then there is a constant $C$ depending on $Q$ and the numbers $\al$ and $\be^k,$ but not on  $H,$ the $a_k,$ $M$ and $t,$   such that 
\begin{equation}\label{simsum}
|F(t)|\le C\frac {\|H\|_{C^1(Q)}} {|2^{i\al t}-1|}, \qquad \mbox{ for all } t\in\RR, a_1,\dots a_2\in \RR^\times \mboxÊ{ and  } \ M\in \NN.
\end{equation}
\end{lemma}

\begin{proof}  
For   $y=(y_1,\dots, y_n)$  in an open neighborhood of $Q,$  Taylor's integral formula allows to  write 
 $ H(y)=H(0)+\sum_{k=1}^n y_k H_k(y),
 $
 with continuous functions $H_k$ whose $C^0$-norms on  $Q$ are controlled by the $C^1(Q)$-norm of $H.$   Accordingly, we shall decompose $F(t)=F_0(t)+\sum_k F_k(t),$ where
 \begin{eqnarray*}
F_0(t)&:=&H(0)\sum_{l=0}^{M} 2^{i \al lt}
\chi_Q \Big(2^{\be^1l}a_1,\dots, 2^{\be^n l}a_n\Big),\\
F_k(t)&:=&\sum_{l=0}^{M} 2^{i \al lt}
(y_kH_k\chi_Q)\Big(2^{\be^1l}a_1,\dots, 2^{\be^n l}a_n\Big), \qquad k=1,\dots, n.
\end{eqnarray*}
It will therefore suffice to establish estimates of the form \eqref{simsum} for each of these functions $F_0$ and  $F_k, \, k=1,\dots, n.$ We begin with $F_0.$ 

Observe that in the  sum defining $F_0(t)$ we are effectively summing over an ``interval'' $l\in\{M_1,\dots,M_2\},$ where $M_1,M_2\in\NN$ depend on $M,$ the $a_k$'s and the $\be^k$'s, so that 
$$F_0(t)=H(0) \frac{2^{i\al (M_2+1)}- 2^{i\al M_1}}{2^{i\al t}-1}.
$$
This implies an estimate of the form \eqref{simsum} for $F_0(t).$  Next, if $k\ge 1,$ then trivially 
$$
|F_k(t)|\le C' \sum_{\{l:2^{\be^k l}|a_k|\le R_k\}} 2^{\be^k l}|a_k|\le C R_k,
$$
by summing a geometric series. Again this implies an estimate of the form \eqref{simsum}. 
\end{proof} 

\begin{remark}\label{simsumrem}
The estimate in \eqref{simsum} can be sharpened as follows (this will become relevant in \cite{IM-rest2}):

Assume that there are constants $\epsilon\in]0,1]$ and $C_k, \, k=1,\dots, n,$ such that
\begin{equation}\label{simsumb}
\int_0^1\Big|\frac{\pa H}{\pa y_k}(sy)\Big|\, ds \le C_k |y_k|^{\epsilon-1}, \qquad \mbox{ for all } y\in Q.
\end{equation}
Then, under the hypotheses of Lemma \ref{simplesum}, there is a constant $C$ depending on $Q,$ the numbers $\al$ and $\be^k$ and $\epsilon,$ but not on  $H,$ the $a_k,$ $M$ and $t,$   such that 
\begin{equation}\label{simsum2}
|F(t)|\le C\frac {|H(0)|+\sum_k C_k} {|2^{i\al t}-1|}, \qquad \mbox{ for all } t\in\RR, a_1,\dots a_2\in \RR^\times \mboxÊ{ and  } \ M\in \NN.
\end{equation}
\end{remark}

Indeed, Taylor's integral formula  and \eqref{simsumb} imply that 
$
|y_kH_k(y)| \le C_k|y_k|^\epsilon,
$
which suffices to concude in a similar way as before.

\subsubsection{Estimation of $\mu^{II,1}_{1+it}(x)$: The contribution by those $m$ for which  $2^m\le M_0^2\de_0^{-6}$}
For such $m$ we have $\de_0^2 2^{m/3}\lesssim 1,$ so that the last term in \eqref{7.16I}  in the phase $\tilde\Phi_3$ can be included into the amplitude of $f_{m,x}$  and we may re-write  $f_{m,x}(2^l)$as an oscillatory integral of the form
$$
f_{m,x}(2^l)= \int e^{-i s_3 \Phi_4(x,v,\de,m,2^l)}\,a\Big(\de_0^2 2^{\frac m3}, 2^{l-\frac m3},v,x_1,s_3,\de
;\,2^{m-2l},2^l \Big)
 \tilde\chi_1(v) \chi_1(s_3) \,dvds_3, 
 $$
where 
$$
\Phi_4(x,v,\de,m,2^l):=2^{m-2l} \tilde B_0(x,\de,\si)+ 2^{-l}2^{\frac {2m}3}\tilde B_1(x,\de,\si)  v.
$$

Observe also that it will here suffice to prove that 
\begin{equation}\label{7.18I}
\Big|\ga(1+it)\sum_{M_0\le  2^l\le \ve_0 2^{\frac m 3}}  2^{it\frac{7}{2}l}\,f_{m,x}(2^l)\Big| \le C ,
\end{equation}
with $C$ independent of $m,x,t,$ etc., because  this will immediately  imply that $|\mu^{II,1}_{1+it}(x)|\le C'.$ 
\medskip

Now, recall first that $a$ is a classical symbol of order $0$ with respect to both last variables, so that we may write 
$$
a\Big(\de_0^2 2^{\frac m3}, 2^{l-\frac m3},v,x_1,s_3,\de
;\,2^{m-2l},2^l \Big)=g\Big(\de_0^2 2^{\frac m3}, 2^{l-\frac m3},v,x_1,s_3,\de\Big) +O((2^{m-2l})^{-1}+2^{-l}),
$$
where the first term $g$ is the leading homogeneous term of order $0$ of $a,$  hence a smooth functions of all its variables, and the constant in the error term is independent of the other variables appearing here. 

Since we are summing only over $l$'s for which $ 2m-2l \ge 0$ and $l\ge 0,$ we see that the contributions by the  term $O((2^{m-2l})^{-1}+2^{-l})$ in \eqref{7.18I} can be estimated in the desired way. With a slight abuse of notation, let us therefore from now on assume that 
\begin{eqnarray}\nonumber
f_{m,x}(2^l)&=& \int e^{-i s_3 \Phi_4(x,v,\de,m,2^l)}\,g\Big(\de_0^2 2^{\frac m3}, 2^{l-\frac m3},v,x_1,s_3,\de\Big)\\
&&\hskip4cm\times\,  \tilde\chi_1(v) \chi_1(s_3) \,dvds_3. \label{7.17I}
\end{eqnarray} 

Given $x,$ consider  first those $l$ for which  $|2^{-l}2^{\frac {2m}3}\tilde B_1(x,\de,\si)|\ge 1.$ Integration by parts in $v$ then implies that 
$$
|f_{m,x}(2^l)|\le \frac C{|2^{-l}2^{\frac {2m}3}\tilde B_1(x,\de,\si)|}.
$$
Summing a geometric series, we thus see that 
$\sum_l |f_{m,x}(2^l)|\lesssim 1$ for the sum over these $l's.$ 

Similarly, if we consider those $l$ for which $|2^{-l}2^{\frac {2m}3}\tilde B_1(x,\de,\si)|< 1$ and \hfill  \newline$|2^{m-2l}\tilde B_0(x,\de,\si)|\gg 1,$ by means of an integration by parts in $s_3$ we find that 
$$
|f_{m,x}(2^l)|\le \frac C{|2^{m-2l}\tilde B_0(x,\de,\si)|},
$$
and again the according sum in $l$ is uniformly bounded.

\medskip
We may therefore restrict ourselves in the sequel to the set of  those $l$ for which 
$|2^{-l}2^{\frac {2m}3}\tilde B_1(x,\de,\si)|\lesssim 1$  and  $|2^{m-2l}\tilde B_0(x,\de,\si)|\lesssim 1.$ 
In this case, \eqref{7.17I} shows that 
$$
f_{m,x}(2^l)=H\Big(2^{-2l}2^m \tilde B_0(x,\de,\si), 2^{-l}2^{\frac {2m}3}\tilde B_1(x,\de,\si), 2^l 2^{-\frac m3}\Big),
$$
where $H$ is a smooth function of its (bounded) variables. Indeed, $H$ will also depend on $m,x,\de$ etc., but in such a way that its $C^1$- norm on compact sets is uniformly bounded. This shows that the contribution of the $l's$ that we are are here considering to the sum in Ê\eqref{7.18I} leads to a sum of the form \eqref{simsum}, with $\al:=7/2,$ and where the cuboid $Q$ is defined by the  following set of restrictions, for suitable $R_1,R_2>0:$
\begin{eqnarray*}
|y_1|=|2^{-2l}2^m \tilde B_0(x,\de,\si)|\le R_1,\quad |y_2|=|2^{-l}2^{\frac {2m}3}\tilde B_1(x,\de,\si)|\le R_1, \quad |y_3|=|2^l 2^{-\frac m3}|\le \ve_0.
\end{eqnarray*}

Finally, since  $\ga(1+it)=(2^{i\frac 72 t}-1)/(2^{-2}-1),$   we see that \eqref{7.18I} is an immediate consequence of Lemma \ref{simsum}.

\subsubsection{Estimation of $\mu^{II,2}_{1+it}(x)$: The contribution by those $m$ for which  $2^m>M_0^2\de_0^{-6}$}
For such  $m$ we have $\de_0^2 2^{m/3}\gg 1.$  

\medskip We shall have to distinguish three further subcases. Let us first assume that $ 2^{\frac {2m}3}2^{-l} |\tilde B_1(x,\de,\si)|\gg \de_0^2 2^{\frac m3}$ in \eqref{7.16I}. An integration by parts in $v$ then shows that 
$$
|f_{m,x}(2^l)|\lesssim \Big( 2^{\frac {2m}3}2^{-l}|\tilde B_1(x,\de,\si)|\Big)^{-1}.
$$
The summation over those $l$ for which $ 2^{\frac {2m}3}2^{-l} |\tilde B_1(x,\de,\si)|\gg \de_0^2 2^{\frac m3}$  can therefore be estimated by a constant times $\Big(\de_0^2 2^{\frac m3}\Big)^{-1},$ so that  the contribution of the corresponding   $f_{m,x}(2^l)$ to  $\mu_{1+it}^{II,2}(x)$  can be estimated by
\begin{equation}\label{7.19I}
C\de_0^{\frac 13}\sum_{2^m> M_0^2\de_0^{-6}} 2^{\frac m{18}} \Big(\de_0^2 2^{\frac m3}\Big)^{-1}\\
\lesssim 1.
\end{equation}

\medskip

Assume next that $ 2^{-l} 2^{\frac {2m}3} |\tilde B_1(x,\de,\si)|\lesssim \de_0^2 2^{\frac m3} ,$ but 
$ 2^{m-2l}|\tilde B_0(x,\de,\si)|\gg  \de_0^2 2^{\frac m3}.$  Then an integration by parts in $s_3$ shows that 
$$
|f_{m,x}(2^l)|\lesssim \Big(2^{m-2l}|\tilde B_0(x,\de,\si)|\Big)^{-1},
$$
so that we can argue in the same way is in the preceding subcase to see that  the contribution of the corresponding   $f_{m,x}(2^l)$ to  $\mu_{1+it}^{II,2}(x)$  is again uniformly bounded with respect to $t,x,\de$ and $\si.$

\medskip 
We may thus  assume  that $  2^{\frac {2m}3}2^{-l} |\tilde B_1(x,\de,\si)|\lesssim \de_0^2 2^{\frac m3} $  and   $ 2^{m-2l}|\tilde B_0(x,\de,\si)|\lesssim  \de_0^2 2^{\frac m3}.$  Then we may  re-write 
$$
f_{m,x}(2^l)= \int e^{-i s_3\de_0^2 2^{\frac m3} \Phi_5(x,v,\de,m,2^l)}\,a\Big(2^{l-\frac m3},v,x_1,s_3,\de;\, 2^{m-2l}, 2^l
\Big)
\tilde\chi_1(v) \chi_1(s_3) \,dvds_3,
$$
where 
$$
\Phi_5:=  \tilde B_2(x,\de_02^{l-\frac m3} v,\de,\si)\, v^2+\de_0^{-2} 2^{\frac m3-l}  \tilde B_1(x,\de,\si)  v+ 
 2^{\frac {2m}3-2l}\de_0^{-2} \tilde B_0(x,\de,\si).
 $$
 Observe also that here $|\tilde\Phi_5(x,v,\de,m,2^l)|\lesssim 1.$  
 
 \medskip
 Let us first  consider those $l$ for which  $ 2^{\frac {2m}3}2^{-l}  |\tilde B_1(x,\de,\si)|\ll \de_0^2 2^{\frac m3}.$ Then the coefficient of $\Phi_5$ of the linear term in $v$  is small, so that  we may change variables from $v$ to $\Phi_5(x,v,\de,m,v),$ which then easily shows that $f_{m,x}(v)$  is of the form
 $$
 f_{m,x}(2^l)=F\Big(\de_0^2 2^{\frac m3};\, 2^{\frac {2m}3-2l}\de_0^{-2} \tilde B_0(x,\de,\si),  \de_0^{-2} 2^{\frac m3-l} \tilde B_1(x,\de,\si) ,2^{l-\frac m3},\de; \, 2^{m-2l}, 2^l\Big),
 $$
 where $F$ is a smooth function which is a Schwartz function with respect to the first  variable, whose Schwartz norms  are each uniformly bounded with respect to the other variables. Moreover, $F$ is uniformly a classical symbol of order $0$ in both of the last two variables.  Similar statements apply  also to the partial derivatives of $F.$  
 
 This clearly implies that  $ |f_{m,x}(2^l)|\lesssim (\de_0^2 2^{\frac m3})^{-N}$  for every $N\in\NN.$  However, such an estimates is not sufficient in order to control the summation in $l.$   
 
 We therefore isolate the leading homogeneous  term  of order $0$ of $F$ with respect to the last two variables, which gives a smooth function 
 $$
 h\Big(\de_0^2 2^{\frac m3};\, 2^{\frac {2m}3-2l}\de_0^{-2} \tilde B_0(x,\de,\si),  \de_0^{-2} 2^{\frac m3-l} \tilde B_1(x,\de,\si) ,2^{l-\frac m3},\de\Big)
 $$
 of bounded variables, and the remainder terms, which clearly can be estimated by a constant times 
 $(\de_0^2 2^{\frac m3})^{-N}((2^{m-2l})^{-1}+2^{-l}).$ The second factor allows to sum in $l,$ and then the first factor (choosing $N=1$) leads again  to an  estimate  of the form \eqref{7.19I} for the contribution by the remainder terms. 
 
 \medskip
In order to control the main term given by the function $h,$ we shall  again  apply again Lemma \ref{simsum}.

Let us here define a cuboid $Q$ by the  following set of restrictions, for suitable $R_1,\ve_2>0:$
\begin{eqnarray*}
&&|y_1|=|2^{-2l} 2^{\frac {2m}3}\de_0^{-2} \tilde B_0(x,\de,\si)|\le R_1, \quad |y_2|=|2^{-l}\de_0^{-2} 2^{\frac m3} \tilde B_1(x,\de,\si) |\le \ve_2, \\
&&|y_3|=|2^l 2^{-\frac m3}|\le \ve_0, \quad |y_4|=|2^{-l}\de_0^3 2^{\frac m2}|\le 1
\end{eqnarray*}
(the last condition stems for the additional summation restriction in the definition of $\mu_{1+it}^{II,2}(x)$),
and let us define 
$H_{m,\de}(y_1,\dots, y_4):=h(\de_0^2 2^{\frac m3};\, y_1,y_2,y_3,\de).
$
Then (choosing $N=1$)
$$\|H_{m,\de}\|_{C^1(Q)}\le C (\de_0^2 2^{\frac m3})^{-1},
$$
and thus Lemma \ref{simsum} implies that the sum over the $l$'s in the definition of $\mu_{1+it}^{II,2}(x)$ can be estimated by $C (\de_0^2 2^{\frac m3})^{-1},$ so that the  remaining sum in $m$ can again be estimated by the expression in \eqref{7.19I}.  This concludes the discussion of  also this subcase.

  \medskip
 
We are thus eventually reduced to  those $l$'s for which  $  2^{\frac {2m}3}2^{-l} |\tilde B_1(x,\de,\si)|\sim \de_0^2 2^{\frac m3} \gg 1$  and   $ 2^{m-2l}|\tilde B_0(x,\de,\si)|\lesssim  \de_0^2 2^{\frac m3}.$ Assume more precisely that we consider here pairs $(m,l)$ for which 
\begin{equation}\label{7.20I}
\frac 1 A \de_0^2 2^{\frac m3}\le 2^{\frac {2m}3}2^{-l} |\tilde B_1(x,\de,\si)|\le A \de_0^2 2^{\frac m3},
\end{equation}
where $A\gg 1$ is a fixed constant. 
 In this situation, the phase $\Phi_5$ will have only non-degenerate critical points of size $1$ as a function of $v,$ or none. The latter case can be treated as before, so assume that we have a critical point $v^c$ such that $|v^c|\sim 1$ when $v\sim 2^l.$ Then we may apply the method of stationary phase in $v$ in \eqref{7.16I}, which leads to the following estimate for  $f_{m,x}(2^l):$
$$
|f_{m,x}(2^l)|\lesssim \Big(\de_0^2 2^{\frac m3}\Big)^{-\frac 12}.
$$
But, given $m,$ \eqref{7.20I} means that we are summing over at most $\log A^2$ different $l$'s, and thus  the  contribution of those   $f_{m,x}(2^l)$ which  we are considering here to the sum forming  $\mu_{1+it}^{II,2}(x)$ can be estimated by 
$$
C\log A^2 \,\de_0^{\frac 13}\sum_{2^m> M_0^2\de_0^{-6}} 2^{\frac m{18}} \Big(\de_0^2 2^{\frac m3}\Big)^{-\frac 12} \lesssim 1.
$$
Combining this estimate with the previous ones, we see that we can bound $|\mu^{II,2}_{1+it}(x)|\le C,$  with a constant $C$ which is independent of $t,x, \de$ and $\si.$  This concludes the proof of the estimate \eqref{mu-II-est}, hence of \eqref{7.8I},  \eqref{7.5I}, and  consequently of Proposition \ref{m2_A} (c).

\setcounter{equation}{0}
\section{Proof of Proposition \ref{m2_A} (a),(b): Complex interpolation}\label{complex2}

For the proofs of parts (a) and (b) of Proposition \ref{m2_A} we shall make use of  similar interpolation schemes. A crucial result  for part (a) will also be the following analogue to Lemma \ref{simplesum},  for oscillatory  double-sums. Its proof follows similar ideas, but is technically more involved and therefore postponed to the Appendix  in Section \ref{doublesumproof}.

\begin{lemma}\label{doublesum}
Let $Q=\prod_{j=1}^nÊ[-R_k,R_k]\subset \RR^n$ be a compact cuboid, with $R_k>0, k=1,\dots, n,$ and let $H$ be a $C^2$-function on an open neighborhood of $Q.$ Moreover, let $\al_1,\al_2\in \QQ^\times$ and $\be^k_1,\be^k_2\in \QQ$  such that the vectors $(\al_1,\al_2)$ and  $(\be^k_1,\be^k_2) $ are  linearly independent, for every $k=1,\dots,n,$  i.e., 
\begin{equation}\label{8.1ind}
\al_1\be^k_2-\al_2\be^k_1\ne 0, \qquad k=1,\dots,n. 
\end{equation}
For any given real numbers $a_1,\dots,a_n\in \RR^\times$ and $M_1,M_2\in \NN$ we then put 
\begin{equation}\label{Ft}
F(t):= \sum_{m_1=0}^{M_1}\sum_{m_2=0}^{M_2} 2^{i(\al_1m_1+\al_2m_2)t}
(H\chi_Q)\Big(2^{(\be^1_1 m_1+\be^1_2 m_2)}a_1,\dots, 2^{(\be^n_1 m_1+\be^n_2 m_2)}a_n\Big).
\end{equation}
Then there is a constant $C$ depending on $Q$ and the numbers $\al_i$ and $\be^k_i,$ but not on $H,$  the $a_k,$ $M_1,M_2$ and $t,$  and a number $N\in\NN^\times$ depending on the $\be^k_i,$ such that 
\begin{equation}\label{8.9I}
|F(t)|\le C\frac {\|H\|_{C^2(Q)}}{|\rho(t)|}, \qquad \mbox{ for all } t\in\RR, a_1,\dots a_2\in \RR^\times \mboxÊ{ and  } \ M_1,M_2\in \NN,
\end{equation}
where $\rho(t):=\prod_{\nu=1}^N\tilde\rho(\nu t),$ with 
$$
\tilde\rho(t):=(2^{i\al_1 t}-1)(2^{i\al_2 t}-1)\prod_{k=1}^n(2^{i(\al_1\be^k_2-\al_2\be^k_1)t}-1)
$$
\end{lemma}

\begin{remark}\label{doublean}
For $\zeta\in \CC$ and $0<\theta<1,$    let us put 
$$
\tilde\ga(\zeta):=(2^{\al_1 (\zeta-1)}-1)(2^{\al_2 (\zeta-1)}-1)\prod_{k=1}^n(2^{(\al_1\be^k_2-\al_2\be^k_1)(\zeta-1)}-1)
$$
Then $\tilde\ga(\nu \th))\ne 0,$ so that 
$\ga_\th(\zeta):=\prod_{\nu=1}^N(\tilde\ga(\nu \zeta)/\tilde\ga(\nu \th))
$
 is a well-defined entire analytic  function such that $\ga_\th(\th)=1.$  Moreover, for $\zeta$ in the complex strip  $\Sigma:=\{\zeta\in \bC: 0\le \Re \zeta\le 1\},$ this function is uniformly bounded, and 
 $
 \ga_\th(1+it)=c_\th\rho(t),
 $
 so that 
 \begin{equation}\label{8.11I}
\Big| \ga(1+it) F(t)\Big| \le C \qquad \mbox{ for all } t\in\RR, a_1,\dots a_2\in \RR^\times \mboxÊ{ and  } \ M_1,M_2\in \NN,
\end{equation}
if $F(t)$ is defined as in Lemma \ref{doublesum}.
\end{remark}

\subsection{ Estimate \eqref{restm3} in Proposition \ref{m2_A} (a) } Recall that $\de_0=2^{-j},$ and that 
$$
\nu_{\de,j}^{V}=\sum_{\la_1=2^{M+j}}^{2^{2j}} \sum_{\la_3=(2^{-M-j}\la_1)^2}^{2^{-M} \la_1} \nu^{(\la_1,\la_1,\la_3)}_j
$$
(in this notation, summation is always meant to be over dyadic $\la_j$'s),
and  that, by \eqref{hatnu3},  $\|\widehat{\nu^{\la}_{j}}\|_\infty\lesssim \la_1^{-1/2} \la_3^{-1/2}.$ We  therefore define here for $\ze$ in the strip $\Sigma=\{\zeta\in \bC: 0\le \Re \zeta\le 1\}$  an analytic family of measures by 
$$
\mu_\ze(x):=\gamma(\ze)\,\de_0^{\frac \ze 3} \sum_{k_1=M+j}^{2j} \sum_{k_3=-2M+2k_1-2j}^{-M+k_1}2^\frac{(3-7\zeta)k_1}6 2^\frac{(3-7\zeta)k_3}6\nu^{(2^{k_1},2^{k_1},2^{k_3})}_j,
$$
where $\ga(\ze)$ is an entire function which  will serve a similar role as the function $\ga(z)$ in Subsection \ref{section7.1}. Its precise definition will be given later (based on Remark \ref{doublean}). It will again be uniformly bounded on $\Sigma,$  and such that  $\ga(\th_c)=\ga(3/7)=1.$ 
\medskip

By $T_\ze$  we denote the operator of convolution with $\widehat{\mu_\ze}.$ 
Observe that  for $\ze=\theta_c=3/7,$ we have $\mu_{\th_c}=\de_0^{\frac 17}\nu_{\de,j}^{V},$ hence 
$
T_{\theta_c}=2^{-\frac j7}T_{\de, j}^V,
$
 so that, again  by Stein's interpolation theorem,  \eqref{restm3}  will follow if we can prove the following estimates on the boundaries of the  strip $\Sigma:$ 
\begin{eqnarray}\label{8.1I}
\|\widehat{\mu_{it}}\|_\infty &\le& C \qquad \forall t\in\RR,\\ 
\|{\mu_{1+it}}\|_\infty &\le& C \qquad \forall t\in\RR. \label{8.2I}
\end{eqnarray}

As before, the  first estimate \eqref{8.1I} is an immediate consequence of  the estimates \eqref{hatnu3},  so let us concentrate of \eqref{8.2I}, i.e., assume that $\zeta=1+it,$ with $t\in \RR.$ We then have to  prove that there is constant $C$ such that 
\begin{equation}\label{8.3I}
|\mu_{1+it}(x)|\le C,
\end{equation}
where $C$ is independent of $t,x,\de$ and $\si.$


Let us  introduce the  measures $ \mu_{\la_1,\la_3}$ given by
$$
\mu_{\la_1,\la_3}(x):=(\la_1\la_3)^{-\frac23} \nu_j^{(\la_1,\la_1,\la_3)}(x), 
$$
which allow to re-write 
\begin{equation}\label{mula13}
\mu_{1+it}(x)=\gamma(1+it
)\,\de_0^{\frac {1+it} 3} \sum_{\la_1=2^{M}\de_0^{-1}}^{\de_0^{-2}} \sum_{\la_3=2^{-2M}(\de_0\la_1)^2}^{2^{-M} \la_1}
(\la_1\la_3)^{-\frac 76 it} \mu_{\la_1,\la_3}(x).
\end{equation}

Notice that according to  Remark \ref{rem5.3I}
\begin{eqnarray}\nonumber
\mu_{\la_1,\la_3}(x)&=&\la_1^\frac 43\la_3^\frac 13 \int \check\chi_1\Big({\la_1}(x_1-y_1)\Big) \, \check\chi_1\Big({\la_1}(x_2-\de_0y_2-y_1^2\om(\de_1y_1))\Big)\\
&&\check\chi_1\Big({\la_3}\Big(x_3-b_0(y,\de)\,y_2^2
-\si y_1^n \beta(\de_1 y_1)\big)\Big)\nonumber
\,\eta(y,\de) \,dy,
\end{eqnarray}
where $\eta$ is supported where $y_1\sim 1$ and $|y_2|\lesssim  1.$
Assume first that $|x|\gg 1.$ Since $\check\chi_1$ is rapidly decreasing, after scaling in $y_1$ by the factor $1/\la_1,$  we then easily see that 
$|\mu_{\la_1,\la_3}(x)|\le C_N\la_1^\frac 13\la_3^{-N}$ for every $N\in\NN.$ Since $ 2^j\lesssim  \la_1\le{2^{2j}}$ and  $(2^{-j}\la_1)^2\lesssim \la_3\ll \la_1$ in the sum defining $\mu_{1+it}(x),$ this easily implies \eqref{8.3I}.

\medskip
From now on, we may and shall therefore assume that $|x|\lesssim1$. 

By means of the change of variables $y_1\mapsto x_1-y_1/\la_1, \  y_2\mapsto y_2/\la_3^{1/2}$ and Taylor expansion around $x_1$  we  may re-write 
$\mu_{\la_1,\la_3}(x)=\la_1^\frac 13\la_3^{-\frac16} \tilde\mu_{\la_1,\la_3}(x),$ with 
\begin{equation}\label{mulala}
\tilde\mu_{\la_1,\la_3}(x):=\iint  \check\chi_1(y_1)F_\de(\la_1,\la_3,x,y_1,y_2)\, dy_1dy_2,
\end{equation}

where
\begin{eqnarray*}
F_\de(\la_1,\la_3,x,y_1,y_2)&:=&\eta(x_1-\la_1^{-1}y_1,\, \la_3^{-\frac12} y_2,\de) \,\check\chi_1(D-Ey_2+r_1(y_1))\\
&\times& \check\chi_1\Big(A-y_2^2b_0(x_1-\la_1^{-1}y_1,\la_3^{-\frac12}y_2,\de)+\la_3\la_1^{-1}r_2(y_1)\Big).
\end{eqnarray*}
Here,  the quantities 
\begin{eqnarray}\nonumber
A&=&A(x,\la_3,\de):=\la_3Q_A(x),\ D=D(x,\la_1,\de):=\la_1Q_D(x),\  E=E(\la_1,\la_3,\de):=\de_0\la_1\la_3^{-\frac12},\\
&& \qquad \mbox{  with}\quad  Q_A(x):=x_3-\sigma x_1^n\beta(\de_1x_1), \ Q_D(x):=x_2-x_1^2\omega(\de_1x_1), \label{ADE}
\end{eqnarray}
do not depend  on $y_1,y_2,$ 
and $r_i(y_1)=r_i(y_1;\la_1^{-1}, x_1,\de),\,  i=1,2, $ are smooth functions of $y_1$  (and $\la_1^{-1}$ and $x_1$)  satisfying estimates of the form
\begin{equation}\label{8.5I}
|r_i(y_1)|\le C|y_1|, \ 
\quad \left|\left(\frac{\pa}{\pa (\la_1^{-1})}\right)^lr_i(y_1; \la_1^{-1}, x_1,\de)\right|\le C_l |y_1|^{l+1}\quad \mbox{for every}\quad l\ge1.
 \end{equation}
Notice that we may here assume that $|y_1|\lesssim \la_1,$ because of our assumption  $|x|\lesssim1$ and the support properties of $\eta.$ It will also be important to observe that 
$E=\de_0\la_1\la_3^{-\frac12}\le 2^{M/2}$ for the index set of $\la_1,\la_3$ over which we sum in \eqref{mula13}.

\medskip
In order to verify \eqref{8.3I}, given $x,$ we shall split the sum in \eqref{mula13} into three parts, according to whether  $|A(x,\la_3,\de)|\gg 1,$ or  $|A(x,\la_3,\de)|\lesssim1 $ and $|D(x,\la_1,\de)|\gg1$,  or $|A(x,\la_3,\de)|\lesssim1$ and $|D(x,\la_1,\de)|\lesssim1.$

\medskip

\noi {\bf 1. The part where  $|A|\gg 1.$  } Denote by $\mu_{1+it}^1(x)$ the contribution to $\mu_{1+it}(x)$ by the terms for which $|A(x,\la_3,\de)|> K,$  where $K\gg 1$ is a  large constant.
We claim that 
\begin{equation}\label{8.6I}
|\tilde\mu_{\la_1,\la_3}(x)|\lesssim \frac {1}{|A|^{\frac 12}}, \qquad \mbox{if} \ |A|=|A(x,\la_3,\de)|> K,
\end{equation}
provided $K$ is sufficiently large. 
This estimate will imply the right kind of estimate
\begin{eqnarray*}
|\mu_{1+it}^1(x)| &\lesssim& \de_0^{\frac 13} \sum_{\{\la_3: 1\le \la_3\le \de_0^{-2}, \, \la_3|Q_A(x)|\ge K\}}\sum_{\la_1\le \de_0^{-1} \la_3^{\frac 12}}\frac {\la_1^{\frac 13} \la_3^{-\frac 16}}{(\la_3|Q_A(x)|)^{\frac 12}}\\
&\lesssim& \sum_{\{\la_3: 1\le \la_3\le \de_0^{-2}, \, \la_3|Q_A(x)|\ge K\}}   \frac 1 {(\la_3|Q_A(x)|)^{\frac 12}}
\lesssim \frac 1{K^{\frac 12}},
\end{eqnarray*}
since we are summing over dyadic $\la_3$'s. 

\medskip 
In order to verify \eqref{8.6I}, observe first that  an easy van der Corput type estimate for the integration in $y_2$ (making use of the last factor of $F_\de$)  allows to estimate
$$
\int |F_\de(\la_1,\la_3,x,y_1,y_2)|\, dy_2\le C,
$$
where the constant $C$ is independent of $y_1,x,\la$ and $\de$  (recall that $|b_0|\sim 1$!).
Let $\ve >0.$ It follows in particular that the  contribution of the  region where  $|y_1|\gtrsim |A|^\ve$ to $\tilde \mu_{\la_1,\la_3}$ can  be estimated by the right-hand side of \eqref{8.6I}, because of the Schwartz- factor $\check \chi_1(y_1)$  in the double integral defining $\tilde\mu_{\la_1,\la_3}(x).$ 

Let us  thus consider the part of  $\tilde \mu_{\la_1,\la_3}(x)$ given by integrating over the region where  $|y_1|\le C |A|^\ve ,$ where $C$ is a fixed positive  number.  Here, according to \eqref{8.5I} we have $|r_2(y_1)|\lesssim |A|^\ve,$ and hence $|A+\la_3\la_1^{-1}r_2(y_1)|\sim |A|,$ is we choose for instance $\ve=1/2$  and $K$ sufficiently large.

Then an easy estimation  for the $y_2$-integration leads to 
$$
\int\Big|\check\chi_1\Big(A-y_2^2b_0(x_1-\la_1^{-1}y_1,\la_3^{-\frac12}y_2,\de)+\la_3\la_1^{-1}r_2(y_1)\Big)\Big| \, dy_2
\lesssim |A|^{-\frac 12},
$$
and  integrating subsequently in $y_1$ over the region   $|y_1|\le C |A|^\ve ,$  we again arrive at the right-hand side of \eqref{8.6I}.

\medskip
\noi {\bf 2. The part where $|A|\lesssim1 $ and $|D|\gg1.$ }  Denote by $\mu_{1+it}^2(x)$ the contribution to $\mu_{1+it}(x)$ by the terms for which $|A(x,\la_3,\de)|\le K$  and $|D(x,\la_1,\de)|> K.$
We claim that here
\begin{equation}\label{8.8I}
|\tilde\mu_{\la_1,\la_3}(x)|\lesssim \frac {1}{|D|}, \qquad \mbox{if} \ |D|=|D(x,\la_1,\de)|> K,
\end{equation}
provided $K$ is sufficiently large. 
It is again easy to see that  this estimate will imply the right kind of estimate for $|\mu_{1+it}^2(x)|$  (just interchange the roles of $A$ and $D$ and of $\la_1$ and $\la_3$ in the arguments of  the previous situation).

\medskip
In order to prove \eqref{8.8I},  consider first the contribution  to   $\tilde \mu_{\la_1,\la_3}(x)$ given by integrating over the region where  $|y_1|\ge C |D|^\ve ,$ where $C$ is a fixed positive  number. Arguing in the same way as in the previous situation, we find that this part can be estimated by the right-hand side of \eqref{8.8I}.

Next, we consider the contribution  to   $\tilde \mu_{\la_1,\la_3}(x)$ given by integrating over the region where  $|y_1|< C |D|^\ve $  and $|y_2|\gg C |D|^\ve.$ According to \eqref{8.5I}, we then have that $|r_j(y_1)|\lesssim |D|^\ve,\ j=1,2,$ so that we may assume that  $|A+\la_3\la_1^{-1}r_2(y_1)|\ll |D|^\ve,$  hence 
$$| A-y_2^2b_0(x_1-\la_1^{-1}y_1,\la_3^{-\frac12}y_2,\de)+\la_3\la_1^{-1}r_2(y_1) |\gtrsim  |D|^{2\ve}.
$$
This easily implies that also this part of $\tilde \mu_{\la_1,\la_3}(x)$ can be estimated by the right-hand side of \eqref{8.8I}.

\medskip
What remains is  the contribution by the region where  $|y_1|< C |D|^\ve $  and $|y_2|<C |D|^\ve$  (with $C$ isufficiently large, but fixed). Since $E\ll 1,$ we here have that $D-E y_2+r_1(y_1)|\gtrsim |D|,$ and again we see that we can estimate by the right-hand side of \eqref{8.8I}.

\medskip
\noi {\bf 3. The part where $|A|\lesssim  1$ and $|D|\lesssim  1.$ } 
 Denote finally  by $\mu_{1+it}^3(x)$ the contribution to $\mu_{1+it}(x)$ by the terms for which $|A(x,\la_3,\de)|\le K$  and $|D(x,\la_1,\de)|\le K.$ In this case, it is easily seen from formula \eqref{mulala} and \eqref{8.5I} that 

$$
\tilde\mu_{\la_1,\la_3}(x)=\tilde J(A,D, E,\,\la_1^{-1}, \la_3^{-\frac13}, \la_3\la_1^{-1}),
$$
where $\tilde J$ is a smooth function of all bounded variables, hence 
$$
\de_0^{\frac 13}\mu_{\la_1,\la_3}(x)=E^\frac13\,J(A,D, E^\frac 13,\,\la_1^{-1}, \la_3^{-\frac13}, \la_3\la_1^{-1}),
$$ 
where again $J$ is a smooth function.

Let us write $\la_1=2^{m_1}, \la_3=2^{m_2},$ with $m_1,m_2\in \NN.$  In combination with \eqref{mula13} we  then see that 
$\de_0^{-it/3}\, \mu_{1+it}(x)$ can be written  in the form \eqref{Ft}, with $(\al_1,\al_2):=(-\frac 76,-\frac 76)$ and $M_1=\de_0^{-2},M_2:=
2^{-M} \de_0^{-2}.$ The cuboid $Q$ is defined by the  following set of restrictions:
\begin{eqnarray*}
&&|y_1|=|\la_3Q_A(x)|\le K,\quad |y_2|=\la_1|Q_D(x)|\le K, \quad |y_3|=|E^{\frac 13}|= \la_1^{\frac 13} \la_3^{-\frac 16}\de_0^{\frac 13} \le 2^{\frac M3},\\
&&|y_4|=\la_1^{-1} \le 1,\quad |y_5|=\la_3^{-\frac 13}\le 1, \quad |y_6|=|\la_1^{-1}\la_3|\le 2^{-M},
\quad |y_7|=|\la_1^2\la_3^{-1}\de_0^2|\le 2^M, \\
&&|y_8|=|\la_1^{-1}\de_0^{-1}|\le 2^{-M}.
\end{eqnarray*}
The first three conditions arise from our assumptions $|A|\lesssim1,|D|\lesssim 1, |E|\lesssim 1,$ and the last three from the restrictions on the summation indices in \eqref{mula13}. Moreover,  for the function $H$ in Lemma \ref{doublesum}, we my choose $H(y_1,\dots, y_8):=y_3J(y_1,\dots, y_6).$ The corresponding vectors  $(\be^k_1,\be^k_2)$ are given by 
$(0,1), (1,0),(1/3,-1/6), (-1,0),(0,-1/3), (2,-1),$ $(-1,1)$ and $(-1,0).$ Therefore, if we choose for  $\ga(\ze)$  the corresponding function $\ga_{3/7}(\ze)$ of  Remark \ref{doublean}, then Lemma \ref{doublesum} shows that indeed also 
$\mu_{1+it}^3(x)$ satisfies the estimate \eqref{8.3I}. 

This concludes the proof of Proposition \ref{m2_A} (a).

\bigskip

\subsection{ Estimate \eqref{restm2} in Proposition \ref{m2_A} (b)}

Recall that $\de_0=2^{-j},$ and that
$$
\nu_{\de,j}^{VI}:=\sum_{\la_1=2^{M+j}}^{2^{2j}} \sum_{\la_3=2}^{(2^{-M-j}\la_1)^2} \nu^{(\la_1,\la_1,\la_3)}_j
$$
(in this notation, summation is always meant to be over dyadic $\la_j$'s),
and  that, by \eqref{hatnu3},  $\|\widehat{\nu^{\la}_{j}}\|_\infty\lesssim \la_1^{-1/2} \la_3^{-1/2}.$ We  therefore define here for $\ze$ in the strip $\Sigma=\{\zeta\in \bC: 0\le \Re \zeta\le 1\}$  an analytic family of measures by 
$$
\mu_\ze(x):=\gamma(\ze)\,\de_0^{\frac \ze 3} \sum_{k_1=M+j}^{2j} \sum_{k_3=1}^{-2M+2k_1-2j}2^\frac{(3-7\zeta)k_1}6 2^\frac{(3-7\zeta)k_3}6\nu^{(2^{k_1},2^{k_1},2^{k_3})}_j,
$$
where here we need to put 
$$
\ga(\ze):=\frac {2^{\frac 72 (1-z)}-1}3.
$$ 
By $T_\ze$  we denote the operator of convolution with $\widehat{\mu_\ze}.$ 
Observe that  for $\ze=\theta_c=3/7,$ we have $\mu_{\th_c}=\de_0^{\frac 17}\nu_{\de,j}^{VI},$ hence 
$
T_{\theta_c}=2^{-\frac j7}T_{\de, j}^{VI},
$
 so that, arguing exactly as in the preceding subsection by means of Stein's interpolation theorem,  \eqref{restm2}  will follow if we can prove that there is constant $C$ such that  


\begin{equation}\label{8.14I}
|\mu_{1+it}(x)|\le C,
\end{equation}
where $C$ is independent of $t,x,\de$ and $\si.$

As before, we introduce the  measures $ \mu_{\la_1,\la_3}$ given by
$$
\mu_{\la_1,\la_3}(x):=(\la_1\la_3)^{-\frac23} \nu_j^{(\la_1,\la_1,\la_3)}(x), 
$$
which allow to re-write 
\begin{equation}\label{8.15I}
\mu_{1+it}(x)=\gamma(1+it
)\,\de_0^{\frac {1+it} 3} \sum_{\la_1=2^{M+j}}^{2^{2j}} \sum_{\la_3=2}^{(2^{-M-j}\la_1)^2}
(\la_1\la_3)^{-\frac 76 it} \mu_{\la_1,\la_3}(x).
\end{equation}

Recall also that according to Remark \ref{rem5.3I}
\begin{eqnarray}\nonumber
\mu_{\la_1,\la_3}(x)&=&\la_1^\frac 43\la_3^\frac 13 \int \check\chi_1\Big({\la_1}(x_1-y_1)\Big) \, \check\chi_1\Big({\la_1}(x_2-\de_0y_2-y_1^2\om(\de_1y_1))\Big)\\
&&\check\chi_1\Big({\la_3}(x_3-b_0(y,\de)\,y_2^2
-\si y_1^n \beta(\de_1 y_1)\Big)\nonumber
\,\eta(y,\de) \,dy,
\end{eqnarray}
where $\eta$ is supported where $y_1\sim 1$ and $|y_2|\lesssim  1.$
Assume first that $|x|\gg 1.$  If $|x_1|\gg 1$ or $|x_2|\gg 1,$ this easily implies that 
$|\mu_{\la_1,\la_3}(x)|\le C_N\la_1^{-N}\le (\la_1\la_2)^{-N/2}$ for every $N\in\NN,$ because $\la_1\gg \la_3^{1/2}.$   Thus \eqref{8.14I} follows immediately.

And, if $|x_3|\gg 1,$ we may estimate the last factor in the integrand by $C_N\la_3^{-N},$ and then easily obtain that $|\mu_{\la_1,\la_3}(x)|\le C_N\la_1^{4/3}\la_3^{1/3-N} \, \la_1^{-1}(\la_1\de_0)^{-1}=2^j\la_1^{-2/3}\la_3^{1/3-N}.$ Summing first over all $\la_1\gg 2^j \la_3^{1/2},$ and then over $\la_3,$ we find that $|\mu_{1+it}(x)|\lesssim \de_0^{1/3} 2^{j/3}\lesssim 1.$

\medskip
From now on, we may and shall therefore assume that $|x|\lesssim1.$

By means of the change of variables $y_1\mapsto x_1-y_1/\la_1, \  y_2\mapsto y_2/(\de_0\la_1)$ we re-write 
$\mu_{\la_1,\la_3}(x)=\de_0^{-1}\la_1^{-\frac 23}\la_3^{\frac13}\, \tilde\mu_{\la_1,\la_3}(x),$ with 
\begin{equation}\label{mulala2}
\tilde\mu_{\la_1,\la_3}(x):=\iint  \check\chi_1(y_1)\tilde F_\de(\la_1,\la_3,x,y_1,y_2)\, dy_1dy_2,
\end{equation}

where
\begin{eqnarray*}
\tilde F_\de(\la_1,\la_3,x,y_1,y_2)&:=&\eta(x_1-\la_1^{-1}y_1,\,  \de_0^{-1}\la_1^{-1}y_2,\de) \,\check\chi_1(D-y_2+r_1(y_1))\\
&\times&\check\chi_1\Big(A+y_2^2 E \,b_0(x_1-\la_1^{-1}y_1,\de_0^{-1}\la_1^{-1}y_2,\de)+\la_3\la_1^{-1}r_2(y_1)\Big).
\end{eqnarray*}
The  quantities 
\begin{eqnarray}\nonumber
&& A:=\la_3Q_A(x),\quad D:=\la_1Q_D(x),\quad  E:=\frac{\la_3}{(\de_0\la_1)^{2}},\\ 
&& \mbox{  with}\quad  Q_A(x):=x_3-\sigma x_1^n\beta(\de_1x_1),\quad  Q_D(x):=x_2-x_1^2\omega(\de_1x_1), \label{ADE2}
\end{eqnarray}
appearing here again do not depend  on $y_1,y_2,$ and the functions $r_i(y_1)$ are as before (i.e., they are indeed smooth functions  of $y_1, \la_1^{-1}, x_1$ and $\de,$ and  satisfy again  estimates of the form \eqref{8.5I}. Notice that also  here we have that $\la_3/\la_1\ll 1.$ 
Recall also  that we may assume that $|y_1|\lesssim \la_1,$ because of our assumption  $|x|\lesssim1$ and the support properties of $\eta,$ and that $\de_0^{-1}\la_1^{-1}\ll 1.$ Observe finally   that our summation conditions imply that  $E\ll 1.$ 

 Notice  also that the first factor $\check\chi_1(y_1)$ in \eqref{mulala2} in combination with the second factor of $F_\de$ clearly allow for a uniform estimate
$$
|\tilde\mu_{\la_1,\la_3}(x)|\lesssim 1, \quad \mbox{hence}\quad \de_0^{\frac 13}|\mu_{\la_1,\la_3}(x)|\lesssim \Big(\frac {\la_3^{\frac 12}}{\de_0\la_1}\Big)^{\frac 23}.
$$
However, these estimate are not quite sufficient in order to prove estimate \eqref{8.15I}, and so we need to improve on them. The second estimate suggest to introduce new  dyadic summation variables $\la_0,\la_4$ in place of $\la_1,\la_3$ so that 
\begin{equation}\label{nesuva}
\la_3=\la_4^2 \quad \mbox{and } \quad \la_1=\frac {\la_0\la_4}{\de_0},
\end{equation}
for in these new variables we would have $\de_0^{\frac 13}|\mu_{\la_1,\la_3}(x)|\lesssim \la_0^{-2/3}.$ 

More precisely, recalling that $\la_3=2^{k_3},$ we decompose the summation over $k_3$ in \eqref {8.15I}  into two arithmetic progressions, by writing $k_3=2k_4+i,$ with $i\in\{0,1\}$ fixed for each of these progressions. Since all of these sums can be treated in essentially the same way, let us assume for simplicity that $i=0,$ so that $k_3=2k_4.$ Putting $\la_4:=2^{k_4}$  and $\la_0:=2^{k_0},$ and writing $k_1:= k_0+k_4+j,$ we indeed obtain \eqref{nesuva}. Replacing without loss of generality the sum  over the dyadic $\la_3$ in \eqref {8.15I} by the sum over the corresponding arithmetic progression with $i=0,$ it is also easy to check that the summation restrictions $2^{M+j}\le \la_1\le  2^{2j}$ and 
$2\le \la_3\le ( 2^{-M-j}\la_1)^2$ are equivalent to the conditions 
$$
2^M\le \la_0\le (2 \de_0)^{-1}, \qquad 2\le \la_4\le (\de_0\la_0)^{-1}.
$$
We may thus estimate in \eqref{8.15I}
$$
|\mu_{1+it}(x)|\le \sum_{\la_0=2^M}^{(2 \de_0)^{-1}}\la_0^{-\frac 23}\Big |\ga(1+it)
\sum_{\la_4=2}^{(\de_0\la_0)^{-1}}\la_4^{-\frac 72 it}\tilde\mu_{\frac{\la_0\la_4}{\de_0},\la_4^2}(x)\Big|.
$$
For $\la_0$ and $x$ fixed, let us put 
\begin{eqnarray*}
f_{\la_0,x}(\la_4)&:=& \tilde\mu_{\frac{\la_0\la_4}{\de_0},\la_4^2}(x),\\
\rho_{t,\la_0}(x)&:=&\ga(1+it)
\sum_{\la_4=2}^{(\de_0\la_0)^{-1}}\la_4^{-\frac 72 it}f_{\la_0,x}(\la_4).
\end{eqnarray*}
The previous estimate shows that in order to verify \eqref{8.14I},  it will suffice to prove the following uniform estimate: there exist constants $C>0$ and $\epsilon\ge 0$ with $\epsilon <2/3,$ so that for all $x$ with   $|x|\lesssim 1$ and $\de$ sufficiently small we have 
\begin{equation}\label{8.19I}
|\rho_{t,\la_0}(x)|\le C \la_0^\epsilon \quad \mbox{for } \quad 2^M\le \la_0\le (2 \de_0)^{-1}.
\end{equation}

\medskip
In order to prove this, observe that by \eqref{mulala2}
\begin{equation}\label{fla0}
f_{\la_0,x}(\la_4)=\iint  \check\chi_1(y_1)F_\de(\la_0,\la_4,x,y_1,y_2)\, dy_1dy_2,
\end{equation}
where
\begin{eqnarray*}
F_\de(\la_0,\la_4,x,y_1,y_2)&:=&\eta(x_1-\de_0(\la_0\la_4)^{-1}y_1,\,  (\la_0\la_4)^{-1}y_2,\de) \,\check\chi_1(D-y_2+r_1(y_1))\\
&\times&\check\chi_1\Big(A+y_2^2 E \,b_0(x_1-\de_0(\la_0\la_4)^{-1}y_1,(\la_0\la_4)^{-1}y_2,\de)+\la_3\la_1^{-1}r_2(y_1)\Big)
\end{eqnarray*}
and
\begin{eqnarray}\nonumber
&&A=A(x,\la_4,\de)=\la_4^2Q_A(x),\quad D=D(x,\la_0,\la_4,\de)=\frac{\la_0\la_4}{\de_0}Q_D(x),
\quad  E=E(\la_0)=\frac{1}{\la_0^{2}},\\ 
&&\qquad \mbox{  with}\quad  Q_A(x):=x_3-\sigma x_1^n\beta(\de_1x_1), \quad Q_D(x):=x_2-x_1^2\omega(\de_1x_1), \label{ADE3}
\end{eqnarray}

Given $x$ and $\la_0,$ we shall split the summation in $\la_4$ into sub-intervals, according to whether $|D|\gg 1 ,$ 
$|D|\lesssim 1$ and $|A|\gg 1,$ or $|D|\lesssim 1$ and $|A|\lesssim 1.$

\medskip
\noi {\bf 1. The part where $|D|\gg 1$. } Denote by $\rho_{t,\la_0}^1(x)$ the contribution to $\rho_{t,\la_0}(x)$ by the terms for which  $|D|\gg 1.$ 

We first consider the contribution to $f_{\la_0,x}(\la_4)$ given by integrating in \eqref{fla0} over the region where $|y_1|\gtrsim |D|^\ve$ (where $\ve>0$ is assumed to be sufficiently small).  Here, the rapidly decaying  first factor $\check\chi_1(y_1)$ in \eqref{mulala2} leads to an improved estimate of this contribution of the order $|D|^{-N}$ for every $N\in \NN,$ which allows to sum over the dyadic $\la_4$ for which $\la_4|\la_0Q_D(x)/\de_0|=|D|\gg 1,$ and the contribution to  $\rho_{t,\la_0}(x)$ is of order $O(1),$ which is stronger than what is needed in \eqref{8.19I}.

We may therefore restrict ourselves in the sequel to the region where $|y_1|\ll |D|^\ve.$ Observe that, because of \eqref{8.5I},   this implies in particular  that $|r_i(y_1)|\ll |D|^\ve, \, i=1,2.$  By looking at the second factor in $F_\de,$ we  again see that the contribution by the regions where in addition $|y_2|<|D|/2,$ or  $|y_2|>3|D|/2,$ is again of the  order $|D|^{-N}$ for every $N\in \NN,$  and their contributions to 
$\rho^1_{t,\la_0}(x)$ are again admissible.

\medskip
What remains is the region where  $|y_1|\ll |D|^\ve$  and $|D|/2\le |y_2|\le 3|D|/2.$ In addition, we may assume that $y_2$ and $D$ have the same sign, since otherwise we can estimate as before. Let us therefore assume, e.g.,  that  $D>0,$ and that $D/2\le y_2\le 3D/2.$ 

The change of variables $y_2\mapsto Dy_2$ then allows to re-write the corresponding contribution to 
$f_{\la_0,x}(\la_4)$ as 
\begin{equation}\label{mulala3}
\tilde f_{\la_0,x}(\la_4):=D\int_{|y_1|\ll |D|^\ve} \int_{1/2\le y_2\le 3/2} \check\chi_1(y_1)\tilde F_\de(\la_0,\la_4,x,y_1,y_2)\, dy_2dy_1,
\end{equation}
where here
\begin{eqnarray*}
&&\tilde F_\de(\la_0,\la_4,x,y_1,y_2):=\eta(x_1-\de_0(\la_0\la_4)^{-1}y_1,\,  (\la_0\la_4)^{-1}Dy_2,\de) \,\check\chi_1(D-Dy_2+r_1(y_1))\\
&&\hskip1cm \times\, \check\chi_1\Big(A+y_2^2 ED^2 \,b_0(x_1-\de_0(\la_0\la_4)^{-1}y_1,(\la_0\la_4)^{-1}Dy_2,\de)+\la_3\la_1^{-1}r_2(y_1)\Big)
\end{eqnarray*}
Recall also  that $|b_0|\sim 1,$  and notice that, according to Remark \ref{rem5.3I}, $|\pa_{y_2}b_0|\lesssim \de_0\de_2\ll 1.$ In combination with the localization given by $\eta,$ this shows that,  given $y_1,$ we may  change variables from $y_2$ to 
$z:=y_2^2 ED^2 \,b_0(x_1-\de_0(\la_0\la_4)^{-1}y_1,(\la_0\la_4)^{-1}Dy_2,\de),$ and use the last factor of $\tilde F_\de$ in order to estimate the integral in $y_2$ (respectively $z$) by $C|ED^2|^{-1}.$  Subsequently, we may estimate the integration with respect to $y_1$ by means of the factor $\check\chi_1(y_1),$ and find that 
$$
|\tilde f_{\la_0,x}(\la_4)|\le C \frac D{|ED^2|}=C\frac 1{|ED|}.
$$
Interpolating this with the trivial estimate $|\tilde f_{\la_0,x}(\la_4)|\le C $  leads to 
$$
|\tilde f_{\la_0,x}(\la_4)|\le C\frac 1{|ED|^{\frac{\epsilon}2}}=C \la_0^{\epsilon} |D|^{-\frac{\epsilon}2},
$$
where we chose $\epsilon>0$ so that $\epsilon<2/3.$ The factor $ |D|^{-\epsilon/2}$ then allows to sum in $\la_4,$ and we see that altogether  we arrive at the  estimate $|\rho^1_{t,\la_0}(x)|\le C \la_0^\epsilon.$ 
This completes the proof of estimate \eqref{8.19I} in this first case.

\medskip
\noi {\bf 2. The part where $|D|\lesssim 1$ and $|A|\gg 1.$}  Denote by  $\rho_{t,\la_0}^2(x)$ the contribution to $\rho_{t,\la_0}(x)$ by the terms for which  $|D|\lesssim 1$ and $|A|\gg 1.$ Arguing in a similar way as in the previous case, only with $D$ replaced by $A,$ we see that we may restrict to the regions where $|y_1|\lesssim |A|^{\ve}$ and $|y_2|\lesssim |A|^{\ve}$  (where $\ve>0$ is any  fixed, positive constant). In the remaining regions, we can gain a factor $C_N|A|^{-N}$ in the estimate of $f_{\la_0,x}(\la_4)$ in a trivial way. But, if $|y_1|\lesssim |A|^{\ve}$ and $|y_2|\lesssim |A|^{\ve},$ and if $\ve>0$ is  sufficiently small, then 
$$
\Big|A+y_2^2 E \,b_0(x_1-\de_0(\la_0\la_4)^{-1}y_1,(\la_0\la_4)^{-1}y_2,\de)+\la_3\la_1^{-1}r_2(y_1)\Big|\gtrsim |A|,
$$
and thus we obtain an  estimate of the same kind, i.e.,
$$
| f_{\la_0,x}(\la_4)|\le C_N |A|^{-N} \qquad \mbox{ for every } \ N\in\NN.
$$
Summing over all dyadic $\la_4$ such that $\la_4^2|Q_A(x)|=|A|\gg 1,$ this implies $|\rho_{t,\la_0}^2(x)|\le C.$

\medskip
\noi {\bf 3. The part where $|D|\lesssim 1$ and $|A|\lesssim 1.$}  Denote by  $\rho_{t,\la_0}^3(x)$ the contribution to $\rho_{t,\la_0}(x)$ by the terms for which  $|D|\le K $ and $|A|\le K,$ where $K>0$ is a sufficiently large constant.   Observe that $\rho_{t,\la_0}^3(x)$ can again be estimated by means of Lemma \ref{simplesum}. Indeed, the cuboid $Q$ will here be defined by means of the conditions $|D|\le K, |A|\le K$ and $w_1:=\la_1^{-1}\de_0^{-1}=\la_4^{-1}\la_0^{-1}\le  2^{-M-1},w_2:=\la_3/\la_1=\la_4(\de_0/\la_0)\le 2^{-2M},  \la_4(\de_0\la_0)\le 1$ (compare  also the properties of the  functions $r_i(y_1)$) and if we define $M:=1/(\de_0\la_0), \al:= -7/2$ and 
\begin{eqnarray*}
&&H_{x,\de}(A,D,E,w_1,w_2)\\
&&:=\iint  \check\chi_1(y_1) \eta(x_1-\de_0w_1y_1,\,  w_2y_2,\de)
\check\chi_1(D-y_2+ r_1(y_1; \de_0w_1,x_1,\de))\\
&&\quad \times\check\chi_1\Big(A+y_2^2 E \,b_0(x_1-\de_0w_1y_1,w_1y_2,\de)+ w_2r_2(y_1; \de_0w_1,x_1,\de)\Big)\, dy_1dy_2,
\end{eqnarray*}
then \eqref{fla0}  shows that   $f_{\la_0,x}(\la_4)=H_{x,\de}(A,D,E,w_1,w_2)$, and $\ga(1+it)^{-1} \rho^3_{t,\la_0}(x)$ is an oscillatory sum of the form \eqref{Ft1} (with summation index $l:=k_4$).   Moreover, one easily checks that 
$$
\|H_{x,\de}\|_{C^1(Q)}\le C,
$$
with a constant $C$ which does not depend on $x$ and $\de.$  Applying Lemma \ref{simplesum}, we therefore obtain the estimate  $|\rho_{t,\la_0}^3(x)|\le C.$ This completes the proof of estimate \eqref{8.19I}, and hence also the proof of  Proposition \ref{m2_A} (b).

\setcounter{equation}{0}
\section{The case  when $\hl(\phi)\ge 2:$ preparatory results}\label{prep}

Recall that $h=h(\phi)>2$ when $\hl\ge 2,$ and that we assume that the original coordinates $x$ are linearly adapted, so that $d=\hl\ge 2.$ Moreover, based on  Varchenko's algorithm,  we can   locally  find an adapted  coordinate system
$y_1= x_1, \ y_2= x_2-\psi(x_1)$
 for  the function $\phi$  near the origin. In these  coordinates, $\phi$ is given by
$ \pad(y):=\phi(y_1,y_2+\psi(y_1))$ (cf. \eqref{adaptco},\eqref{phia1}).



Also recall that   the vertices of the Newton polyhedron $\N(\pad)$ of $\pad$ are assumed to be the   points
$(A_l,B_l), \  l=0,\dots,n,$ so that the Newton polyhedron $\N(\pad)$ is the convex hull
of the set $\bigcup_{l} ((A_l,B_l)+\bR_+^2),$ where $A_{l-1}<A_{l}$ for every $l\ge 1.$ Moreover, 
$L_l:=\{(t_1,t_2)\in \bR^2:\ka^l_1t_1+\ka^l_2 t_2=1\}$ denotes the line passing through the  points $(A_{l-1},B_{l-1})$ and $(A_l,B_l),$ and  $a_l={\ka^l_2}/{\ka^l_1}.$
 The $a_l$ can be identified as the distinct leading exponents of all the roots of $\pad$ in case that $\pad$ is analytic (see Section 3 of \cite{IKM-max}), and the cluster of roots whose  leading exponent in their Puiseux series expansion is given by  $a_l$ is associated to the edge $\ga_l=[(A_{l-1},B_{l-1}) ,(A_l,B_l)]$ of $\N(\pad).$

As  before, following Subsection 8.2 of \cite{IKM-max}, we choose the integer $l_0\ge 1$ such that
 $$
-\infty=:a_0<\dots <a_{l_0-1}\le m<a_{l_0}<\dots< a_l<a_{l+1}<\dots<a_n.
$$
 As has been shown in Section 3 of \cite{IKM-max},
the vertex  $(A_{l_0-1},B_{l_0-1})$ lies  strictly above the bisectrix, i.e.,  $A_{l_0-1}<B_{l_0-1},$ since the original coordinates $x$ were assumed to be non-adapted.

\medskip
Following in a  slightly modified way the discussion in Section  3 of \cite{IKM-max}  we single out a particular edge by fixing the corresponding index $l_\pr\ge l_0$:
\medskip

\noi{\bf Cases:}
\begin{enumerate}
 \item[(a)] If the principal face $\pi(\pad)$ of $\N(\pad)$ is a compact edge, we choose $l_\pr$ so  that the edge $\ga_{l_\pr}=[(A_{l_\pr-1},B_{l_\pr-1}) ,(A_{l_\pr},B_{l_\pr})]$ is the principal face $\pi(\pad)$ of the Newton polyhedron of $\pad.$

 \item[(b)] If $\pi(\pad)$ is the vertex $(h,h),$  we choose $l_\pr$ so that  $(h,h)=(A_{l_\pr-1},B_{l_\pr-1}).$ Then$(h,h)$ is the right endpoint of the compact edge $\ga_{l_\pr-1}.$

\item[(c)] If the principal face $\pi(\pad)$ is unbounded, i.e.,  a half-line given by  $t_1\ge A$ and  $t_2=h:=B,$ with $A<B,$  then we distinguish two subcases:
\begin{enumerate}

 \item[(c1)] If  the point $(A,B)$ is the right endpoint of a compact edge of $\N(\pad)$,  then  we choose again $l_\pr$ so that  this edge is given by  $\ga_{l_\pr-1}.$

   \item[(c2)] Otherwise,  $(A,B)$ is the only vertex of $\N(\pad),$ i.e., $\N(\pad)=(A,B)+\RR^2_+.$     \end{enumerate}
\end{enumerate}

We also put 
\begin{equation}\label{aa}
a:=\begin{cases} a_{l_\pr}   &   \mbox{ in Case (a)}; \\
 a_{l_\pr-1}   &   \mbox{ in Case (b) and Case (c1)}; \\
  m   &   \mbox{ in Case (c2)}.
\end{cases}
\end{equation}

Following  \cite{IKM-max} and \cite{IM-uniform},  in the cases (a) - (c1) we shall decompose the  domain \eqref{3.1} in which $\rho_1$ is supported into   subdomains
 $$
D_l:=\{\x:\ve_l  x_1^{a_l}< |x_2-\psi(x_1)|\le N_l x_1^{a_l}\},\quad l=l_0,\dots,l_\pr-1,
$$
which correspond to the  $\ka^l$-homogeneous domains $
D^a_l:=\{\y:\ve_l  y_1^{a_l}< |y_2|\le N_l y_1^{a_l}\}$
in  our adapted coordinates $y,$   and  intermediate ``transition''  domains
$$
E_l:=\{\x:N_{l+1} x_1^{a_{l+1}}<|x_2-\psi(x_1)| \le \ve_l x_1^{a_l}\},
$$
where $l=l_0,\dots,l_\pr-1$ in Case (a), and $l=l_0,\dots,l_\pr-2$ in all other cases, as well as the  ``first'' transition domain
$$
E_{l_0-1}:=\{\x:N_{l_0} x_1^{a_{l_0}}<|x_2-\psi(x_1)| \le \ve_{l_0} x_1^{m}\},
$$
corresponding to the $y$-domains $E^a_l:=\{\y:N_{l+1} y_1^{a_{l+1}}<|y_2| \le \ve_l y_1^{a_l}\},$  respectively  $E^a_{l_0-1}:=\{\y:N_{l_0} y_1^{a_{l_0}}<|y_2| \le \ve_{l_0} y_1^{m}\}.
$
Here,  the $\ve_l>0 $ are  small and the $N_l>0$ are  large parameters to be determined later. We remark that the domain $E_{l_0-1}$ can be written like $E_l$ with $l=l_0-1$ if we replace, with some slight abuse of notation, $a_{l_0-1}$ by $m$ and $\ka_{l_0-1}$ by $\ka.$ We shall make use of this unified way of describing $E_l$ in the sequel.

What will remain after removing these domains  is a domain of the form
\begin{equation}\label{restdomain2}
D_\pr:=\begin{cases}  
\{ \x:|x_2-\psi(x_1)|\le N x_1^a\} &   \mbox{ in Case (a)};   \\
      \{ \x:|x_2-\psi(x_1)|\le \ve x_1^{a}\}, &   \mbox{ in all other cases},
\end{cases}
\end{equation}
where $N$ is sufficiently large and $\ve$ sufficiently small.

\smallskip

In the cases (c1) and (c2), we shall furthermore regard the  domains 
\begin{equation}\label{gentrans}
E_{l_\pr-1}:=D_\pr=\{ \x:|x_2-\psi(x_1)|\le \ve x_1^{a}\}
\end{equation}
as ``generalized'' transition domains. Notice that  in the Case (c2) this domain  will cover the domain in \eqref{3.1}, since here $a=m,$  so that the proof of Proposition \ref{nearjet} will be complete once we shall have  handled all these transitions domains in the next section. In a similar way, the discussion of  Case (c1) will be complete once we have handled the domains  $E_l$ and $D_l.$  This will eventually reduce our problem to studying the domain $D_\pr$ in the cases (a) and (b).

\setcounter{equation}{0}
\section{Restriction estimates in the transition domains when $\hl(\phi)\ge 2$ }\label{transition}

\bigskip

Following a standard approach, we would like to study the contributions of the domains $E_l$ by means of a decomposition of the corresponding $y$-domains $E^a_l$ into dyadic rectangles. These rectangles correspond to a kind of ``curved boxes'' in the original coordinates $x,$ so that we  cannot achieve
 the localization to them by means of Littlewood-Paley decompositions in the variables $x_1$ and $x_2.$ However, the following lemma shows that this localization can nevertheless be induced by means of  Littlewood-Paley decompositions in the variables $x_1$ and $x_3.$ 
 
 We shall formulate this lemma for a general smooth, finite type function $\Phi$ with $\Phi(0,0)=0$ and $\nabla\Phi(0,0)=0$ in place of $\pad,$ since it will by applied not only to $\pad.$ However, we shall keep the notation introduced for $\pad,$ denoting for instance by $(A_l,B_l), \  l=0,\dots,n$ the vertices of   the Newton polyhedron of $\Phi,$ by $\ka^l$ the weight associated to the edge $\ga_l=[(A_{l-1},B_{l-1}) ,(A_l,B_l)],$ etc..

\begin{lemma}\label{equivcond}
For $l\ge l_0,$ 
let $[(A_{l-1},B_{l-1}) ,(A_l,B_l)]$ and $[(A_{l},B_{l}) ,(A_{l+1},B_{l+1})]$ be two subsequent compact edges of $\N(\Phi),$ with common vertex $(A_{l},B_{l}),$ and associated weights $\ka^l$ and $\ka^{l+1}.$ Recall also that $a_l=\ka^l_2/\ka_1^l <a_{l+1}=\ka^{l+1}_2/\ka_1^{l+1}.$ 
For  a given $M >0,$ and $\delta>0$ sufficiently small, consider the domain
$$
E^a:=\{\y: 0<y_1<\delta,\, 2^{M}y_1^{a_{l+1}}<|y_2|\le 2^{-M} y_1^{a_l}\}.
$$

(a) There is a constant $C>0$ such that 
\begin{equation}\label{edgeterm}
\Phi(y)=c_{A_{l},B_{l}}y_1^{A_l}y_2^{B_l}\Big(1 +O(\delta^C+2^{-M})\Big) \quad \mboxÊ{on  }\  E^a,
\end{equation}
where $c_{A_{l},B_{l}}$ denotes the Taylor coefficient of $\Phi$ corresponding to $(A_{l},B_{l}).$ More precisely, 
$\Phi(y)=c_{A_{l},B_{l}}y_1^{A_l}y_2^{B_l}(1 +g(y)),$ where $|g^{(\beta)}(y)|\le C_\beta (\delta^C+2^{-M})|y_1^{-\beta_1}y_2^{-\beta_2}|$ for every multi-index $\beta\in\NN^2.$
\smallskip

(b) For $M, j\in \bN$ sufficiently large, the following conditions are equivalent:
 \begin{itemize}
\item[(i)] $y_1\backsim 2^{-j},\, \y\in E^a$ and $2^{A_lj+B_lk}\Phi(y)\backsim 1$;\
\item[(ii)] $y_1\backsim 2^{-j},\, y_2\backsim 2^{-k}$ and $a_lj+M\le k \le a_{l+1}j-M$.
 \end{itemize}
 \smallskip
 \noindent
Moreover, if we set $\phi_{j,k}(x):= 2^{A_lj+B_lk}\Phi(2^{-j} x_1, 2^{-k}x_2),$ then under the previous conditions we have that $\phi_{j,k}(x)=c_{A_{l},B_{l}}x_1^{A_l}x_2^{B_l}\Big(1 +O(2^{-Cj}+2^{-M})\Big)$ on the set where $x_1\sim 1, |x_2|\sim 1,$ in the sense of the $C^\infty$ - topology.
 
 \medskip 
 The statements in (a) and (b) remain valid also in the case where $l=l_0-1.$
\end{lemma}

\proof

When $\Phi$ is analytic, these results have  essentially been proven in Section 8.3 of \cite{IKM-max}, at least implicitly. We shall here give an elementary proof which works also for smooth functions $\Phi.$

We begin with the case where $l>l_0. $ 
Notice first that (b) is an immediate consequence of (a). In order to prove (a), let us denote by $\Phi_N$ the Taylor polynomial of degree $N$ of $\Phi$  centered at the origin. Since $(\Phi-\Phi_N)\y=O(|y_1|^N+|y_2|^N),$ it is easily seen that $y_1^{-A_l}y_2^{-B_l}(\Phi-\Phi_N)\y =O(2^{-B_lM})$ on $E^a,$ provided $N$ is sufficiently large and $\delta$ small. It therefore suffices to prove \eqref{edgeterm} for $\Phi_N$ in place of $\Phi.$ 

If $\Phi(\y\sim\sum_{\al_1,\al_2=0}^\infty c_{\al_1,\al_2} y_1^{\al_1} y_2^{\al_2}$ is the Taylor series of 
of $\Phi$ centered at  the origin, then we decompose  the polynomial $\Phi_N$ as $\Phi_N=P^+ +P^-, $ where 
$$
P^+\y:=\sum_{\al_1+\al_2\le N, \al_2>B_l}c_{\al_1,\al_2} y_1^{\al_1} y_2^{\al_2}, \quad P^-\y:=\sum_{\al_1+\al_2\le N, \al_2\le B_l}c_{\al_1,\al_2} y_1^{\al_1} y_2^{\al_2}.
$$
Let $(\al_1,\al_2)$  be one of the multi-indices appearing in $P^-,$ and assume it is  different from $(A_l,B_l).$
 Let $\y\in E^a,$ and assume, for notational convenience, that $y_2>0.$  Since clearly $A_l,B_l>0,$  we have 
$$
\frac {y_1^{\al_1} y_2^{\al_2}}{y_1^{A_l}y_2^{B_l}}=y_1^{\al_1-A_l} y_2^{\al_2-B_l}\le y_1^{\al_1-A_l} 
\Big(2^My_1^{a_{l+1}}\Big)^{\al_2-B_l}=2^{(\al_2-B_l)M} y_1^{\al_1+a_{l+1}\al_2-(A_l+a_{l+1}B_l)}.
$$
It is easy to  see that $A_l+a_{l+1}B_l=A_{l+1}+a_{l+1}B_{l+1},$  so that 
\begin{equation}\label{6.2}
\frac {y_1^{\al_1} y_2^{\al_2}}{y_1^{A_l}y_2^{B_l}}\le 2^{(\al_2-B_l)M} y_1^{\al_1+a_{l+1}\al_2-(A_{l+1}+a_{l+1}B_{l+1})}.
\end{equation}

But, since $\ga_{l+1}$ is an edge of $\N(\Phi),$ we have that $\ka^{l+1}_1\al_1+\ka^{l+1}_2 \al_2\ge 1,$ i.e., 
$\al_1+a_{l+1} \al_2\ge (\ka^{l+1}_1)^{-1},$ whereas $A_{l+1}+a_{l+1}B_{l+1}=(\ka^{l+1}_1)^{-1}.$ Thus, \eqref{6.2}
 implies that $y_1^{\al_1} y_2^{\al_2}\le \, 2^{(\al_2-B_l)M}\,y_1^{A_l}y_2^{B_l},$ so that  $y_1^{\al_1} y_2^{\al_2}\le \, 2^{-M}\, y_1^{A_l}y_2^{B_l}$ when $\al_2<B_l.$ And, when $\al_2=B_l,$ then  $(\al_1,\al_2)$  lies in the interior of $\N(\Phi),$ so that $\al_1+a_{l+1}\al_2-(A_{l+1}+a_{l+1}B_{l+1})>0,$ hence $y_1^{\al_1} y_2^{\al_2}\le \, \delta^C\, y_1^{A_l}y_2^{B_l}$ for some positive constant $C.$ 
 
 The estimates of the derivatives of $g(y)=\Phi(y)/ c_{A_{l},B_{l}}y_1^{A_l}y_2^{B_l}-1$ follow in a very similar way.
 
 The terms in $P^+$ can be estimated analogously, making use here of the estimates  $y_2\le 2^{-M} y_1^{a_l}$ and $\ka^{l}_1\al_1+\ka^{l}_2 \al_2\ge 1.$  This proves (a).
 
 \medskip
 Finally, if $l=l_0,$ exactly the same arguments work, if we re-define $a_{l_0-1}$ to be  $m$ and $\ka_{l_0-1}$  to be $\ka,$  since $\ka_2/\ka_1=m.$
\qed

A similar result applies also to the  generalized transition domains $E_{l_\pr-1}$ arising in the cases (c1) and (c2), provided we can factor the root $y_2=0$ to its given order, which applies in particular when  $\Phi$ is real-analytic (some easy examples show that it may be false  otherwise). Recall that in these cases, the principal face of $\N(\pad)$ is  an unbounded half-line with left endpoint $(A,B).$ More generally, we have the following result:

\begin{lemma}\label{equivcond2}
Assume that $(A,B)$ is a vertex of $\N(\Phi)$ such that the unbounded horizontal half-line with left endpoint $(A,B)$ is a face of $\N(\Phi),$ and assume in addition that $\Phi$ factors as $\Phi\y=y_2^B \Upsilon\y,$  with a smooth function $\Upsilon.$ Moreover, let  $L_\ka:=\{(t_1,t_2)\in \bR^2:\ka_1t_1+\ka_2 t_2=1\}$ be a non-horizontal supporting  line for $\N(\Phi)$ (i.e., $\ka_1>0$) passing through $(A,B),$ and let $a:=\ka_1/\ka_1.$ 
We then put 
$$
E^a:=\{\y: 0<y_1<\delta,\,|y_2|\le 2^{-M} y_1^{a}\}.
$$

(a) There is a constant $C>0$ such that 
\begin{equation}\label{edgeterm2}
\Phi(y)=c_{A,B}\,y_1^{A}y_2^{B}\Big(1 +O(\delta^C+2^{-M})\Big) \quad \mboxÊ{on  }\  E^a,
\end{equation}
where $c_{A,B}$ denotes the Taylor coefficient of $\Phi$ corresponding to $(A,B).$ More precisely, 
$\Phi(y)=c_{A,B}y_1^{A}y_2^{B}(1 +g(y)),$ where $|g^{(\beta)}(y)|\le C_\beta (\delta^C+2^{-M})|y_1^{-\beta_1}y_2^{-\beta_2}|$ for every multi-index $\beta\in\NN^2.$

\smallskip

(b) For $M, j\in \bN$ sufficiently large, the following conditions are equivalent:
 \begin{itemize}
\item[(i)] $y_1\backsim 2^{-j},\, \y\in E^a$ and $2^{Aj+Bk}\Phi(y)\backsim 1$;\
\item[(ii)] $y_1\backsim 2^{-j},\, y_2\backsim 2^{-k}$ and $aj+M\le k $.
 \end{itemize}
  Moreover, if we set $\phi_{j,k}(x):= 2^{Aj+Bk}\Phi(2^{-j} x_1, 2^{-k}x_2),$ then under the previous conditions we have that $\phi_{j,k}(x)=c_{A,B}x_1^{A}x_2^{B}\Big(1 +O(2^{-Cj}+2^{-M})\Big)$ on the set where $x_1\sim 1, |x_2|\sim 1,$ in the sense of the $C^\infty$ - topology.
\end{lemma}

\proof
It suffices again to prove (a). 

 By our assumption, $\Phi\y=y_2^B\Upsilon\y,$ so that 
$\Phi(y)/y_1^{A}y_2^{B}=\Upsilon(y)/y_1^A.$ Approximating $\Upsilon$ by its Taylor polynomial of sufficiently high degree, we again see that we may reduce to the case where $\Upsilon,$ hence $\Phi,$  is a polynomial. Then let $(\al_1,\al_2)$ be any point different from $(A,B)$ in its Taylor support. Since $\al_2\ge B,$ assuming again that $y_2>0,$ we see that 
$$
\frac {y_1^{\al_1} y_2^{\al_2}}{y_1^{A}y_2^{B}}=y_1^{\al_1-A} y_2^{\al_2-B}\le y_1^{\al_1-A} 
\Big(2^{-M}y_1^{a}\Big)^{\al_2-B}=2^{-(\al_2-B)M} y_1^{\al_1+a\al_2-(A+a B)}.
$$
Moreover, clearly $\al_1+a\al_2\ge A+a B,$ and $\al_1+a\al_2> A+a B$ when $\al_2=B.$ We can thus  argue in a very similar way as in the proof of Lemma \ref{equivcond} to finish the proof.

\qed

Let us now fix $l\in \{l_0-1,\dots, l_\pr-1\},$ and consider  the corresponding  (generalized)  transition domain $E_l$ from Section \ref{prep}, which  can be written as 
$$
E_l=\{\x: Nx_1^{a_{l+1}}<|x_2-\psi(x_1)| \le \ve x_1^{a_l}\},
$$
where, with some slight abuse of notation, we have again re-defined $a_{l_0-1}:=m,$  and put $a_{l\pr}:=\infty$ in the cases (c1) and (c2), so that $x_1^{a_{l_\pr}}:=0,$ by definition. 

Following \cite{IKM-max}, we shall localize to the  domain $E_l$ by means of a cut-off function
$$
\tau_l\x:=\chi_0\Big(\frac {x_2-\psi(x_1)}{\ve x_1^{a_l}}\Big)(1-\chi_0)\Big(\frac {x_2-\psi(x_1)}{N x_1^{a_{l+1}}}\Big),
$$
where $\chi_0\in C^\infty_0(\RR)$ is again supported in $[-1,1]$ and $\chi_0\equiv 1$ on $[-1/2,1/2]$ (actually, $\chi_0$ may depend on $l$). In Case (c), when  
$l=l_\pr-1$ and $a_{l_\pr}=\infty,$ the second factor has to be interpreted as $1,$ i.e., 
$$
\tau_{l_\pr-1}\x=\chi_0\Big(\frac {x_2-\psi(x_1)}{\ve x_1^{a}}\Big).
$$

Recall that $\phi$ is assumed to satisfy Condition (R).
\begin{proposition}\label{inter1}
Let $l\in \{l_0-1,\dots, l_\pr-1\}.$ 
Then, if $\ve>0$ is chosen sufficiently small and $N>0$ sufficiently large, 
$$
\Big(\int_S |\widehat f|^2\,d\mu^{\tau_l}\Big)^{1/2}\le C_{p} \|f\|_{L^p(\RR^3)},\qquad  f\in\S(\RR^3),
$$
whenever $p'\ge p'_c .$
\end{proposition}

\begin{proof}

Consider  partitions of unity
$\sum_{j} \chi_j(s)=1$ and  $\sum_{k} \tilde\chi_{j,k}(s)=1$ on  $\RR\setminus\{0\}$ with $\chi, \tilde\chi \in C_0^\infty(\RR)$ supported in  $[-2,-1/2]\cup [1/2, 2]$ respectively  $ [-2^{B_l},-2^{-B_l}]\cup[2^{-B_l},2^{B_l}],$
where $\chi_j(s):=\chi(2^js)$  and, for $j$ fixed,   $\tilde\chi_{j,k}(s):=\chi(2^{A_lj+B_lk}s),$ and let
$$
\chi_{j,k}(x_1,x_2,x_3):=\chi_j(x_1)\tilde\chi_{j,k}(x_3)= \chi(2^jx_1)\tilde\chi(2^{A_lj+B_lk}x_3)\, ,\quad j,k\in\bZ.
$$
Notice here  that $B_l>B_{l+1}\ge 0$. 
We next put $\mu_{j,k}:=\chi_{j,k}\mu^{\tau_l},$  and assume that $\mu$ has sufficiently small support near the origin.  Then clearly $\mu_{j,k}=0,$ unless $j\ge j_0,$ where $j_0>0$ is a large number which we can still  choose suitably later. But then, according to Lemma \ref{equivcond}, we may  assume in addition that
\smallskip

\begin{equation}\label{indecond}
a_lj+M\le k\le a_{l+1}j-M,
\end{equation}
where $M$ is a large number. 
\smallskip
Indeed, we may choose  $N:= 2^M$ and $\ve:=2^{-M},$ and then Lemma \ref{equivcond} (b) shows that $\mu_{j,k}=0$ for all  pairs $(j,k)$ not satisfying \eqref{indecond}. Notice that this also implies that $k\ge k_0$ for some large number $k_0.$ Observe also  that the measure $\mu_{j,k}$ is supported over a ``curved box'' given by $x_1\sim 2^{-j}$ and $|x_2-\psi(x_1)|\lesssim 2^{-k}.$ This shows that the localization that we have achieved by means of the cut-off function $\chi_{j,k}$ is very similar to the localization that we could have imposed by means of the cut-off function $\chi(2^{-j}x_1)\chi\Big(2^{-k}(x_2-\psi(x_1)\Big).$

\smallskip

Then, applying again Littlewood-Paley theory, now  in the variables $x_1$ and $x_3,$ and interpolating with the trivial $L^1\to L^\infty$ estimate for the Fourier transform,  we see that in order to prove Proposition \ref{inter1}, it suffices to prove uniform  restriction estimates for the measures $\mu_{j,k}$ at the critical exponent,  i.e., that
\begin{equation}\label{6.5}
\int_S |\widehat f|^2\,d\mu_{j,k}\le C\|f\|^2_{L^{p_c}(\RR^3)},\qquad\mbox{when $(j,k)$ satisfies  \eqref{indecond}  and} \   j\ge j_0,
\end{equation}
provided $M$ and $j_0$ are chosen sufficiently large.

We introduce the normalized  measures
$\nu_{j,k}$ given by
$$
\laa \nu_{j,k} ,\, f\ra:=\int f(x_1, \, 2^{mj-k}x_2+x_1^{m}\om(2^{-j}x_1),\, \phi_{j,\,k}\x)\, a_{j,\,k}(x)\, dx,
$$
where
\begin{eqnarray*}
a_{j,\,k}(x)=\eta\Big(2^{-j} x_1,\,2^{-k} x_2 +\psi(2^{-j}x_1)\Big)&\Big[\chi_0\left(2^{a_lj+M-k}\dfrac{x_2}{ x_1^{a_l}}\right)(1-\chi_0)\left(2^{a_{l+1}j-M-k}\dfrac{x_2}{ x_1^{a_{l+1}}}\right)\Big]\\
&\times\,\chi(x_1)\tilde\chi\Big(\phi_{j,\,k}\x\Big).
\end{eqnarray*}
Here, according to Lemma \ref{equivcond}, the functions $\phi_{j,k}$ satisfy 
$$
\phi_{j,\,k}\x=cx_1^{A_l}x_2^{B_l}+O(2^{-M}) \quad \mbox{in }\quad C^\infty
$$
 on domains where $x_1\sim 1, |x_2|\sim 1,$ and the amplitude $a_{j,k}$ in the integral above is supported in such a domain.


Observe that
\begin{equation}\label{6.6}
\laa \mu_{j,k} ,\, f\ra=2^{-j-k}\int f(2^{-j}y_1, 2^{-mj} y_2,2^{-(A_lj+B_lk)}y_3)\, d\nu_{j,k}(y),
\end{equation}
which follows easily by means of a change to adapted coordinates in the integral defining the measure $\mu_{j,k}$ and scaling in $x_1$ by the factor $2^{-j}$ and in $x_2$ by the factor $2^{-k}.$

We observe that the measure $\nu_{j,k}$ is supported on the surface given by
$$
S_{j,k}:=\{(x_1, \, 2^{mj-k}x_2+x_1^m\om(2^{-j}x_1),\, \phi_{j,\,k}\x):  x_1\sim 1\sim x_2\}.
$$
which is a  small perturbation of the limiting surface
$$
S_{\infty}:=\{(x_1, \, x_1^m\om(0),\, cx_1^{A_l}x_2^{B_l}):  x_1\sim 1\sim x_2\},
$$
since $mj-k\le a_l j-k\le -M$ because of \eqref{indecond}. Notice also that $|\pa(cx_1^{A_l}x_2^{B_l})/\pa x_2|\sim 1,$ since $B_l\ge 1.$ This show that $S_{\infty}$ and hence also $S_{j,k}$ (for $j$ and $M$ suffciently large) is a 
smooth hypersurface with one non-vanishing principal  curvature (with respect to $x_1$) of size $\sim 1.$  This implies that
$$
|\widehat{ \nu_{j,k}}(\xi)|\le C (1+|\xi|)^{-1/2},
$$
uniformly in $j$ and $k.$

Moreover, the total variations of the measures $\nu_{j,k}$ are uniformly bounded,
i.e.,
$\sup_{j,k} \|\nu_{j,k}\|_1<\infty.$

We may thus apply again   Greenleaf's result \cite{greenleaf} in order to prove that
 \begin{equation}\label{6.7}
\int |\widehat f|^2\,d\nu_{j,k}\le C\, \|f\|^2_{L^{p}(\RR^3)}
\end{equation}
holds, whenever $p'\ge 6,$ with a constant $C$ which is independent of $j,k.$ Since $p_c'\ge 2d+2\ge 6,$ this holds in particular for $p=p_c.$  Re-scaling this estimate by means of \eqref{6.6}, this implies
\begin{equation}\label{6.8}
\int |\widehat f|^2\,d\mu_{j,k}\le C 2^{-j-k+2\frac {(m+1+A_l)j+B_lk}{p'_c}} \|f\|^2_{L^{p_c}(\RR^3)}.
\end{equation}

But,  we can write $k$ in the form $k=\th a_lj+(1-\th)a_{l+1}j+\tilde M $ with $0\le \th \le1$ and $|\tilde M|\le M$. Then 
\begin{eqnarray*}
-j-k+&&2\frac {(m+1+A_l)j+B_lk}{p'_c}=-j\theta\left[1+a_l- 2\frac {m+1+A_l+a_lB_l}{p'_c}\right]\\
&&-j(1-\th)\left[1+a_{l+1}- 2\frac {m+1+A_l+a_{l+1}B_l}{p'_c}\right]+\left(-1+2\frac { B_l}{p'_c}\right)\tilde M.
\end{eqnarray*}

Recall next that by the definition of the $r$-height and the critical exponent $p'_c,$ , we have $p'_c\ge 2(h_l+1)$ whenever $l\ge l_0.$ 
And, \eqref{hl} shows that 
\begin{equation}\label{hl+1}
h_l+1= \frac{1+(1+m)\ka_1^l}{|\ka^l|}= \frac{m+1+\frac 1{\ka_1^l}}{1+a_l}.
\end{equation}
Moreover, we have seen in the proof of Lemma \ref{equivcond} that $A_l+a_lB_l=1/\ka_1^l,$ so that 
$$
2(h_l+1)=2\frac {m+1+A_l+a_lB_l}{1+a_l}.
$$
We thus find that $1+a_l- 2(m+1+A_l+a_lB_l)/p'_c\ge 0.$   Arguing in a similar way for $l+1$ in place of $l,$ by using that  $p'_c\ge 2(h_{l+1}+1)$ and $A_l+a_{l+1}B_l=1/\ka_1^{l+1}$ we also see that $1+a_{l+1}- 2(m+1+A_l+a_{l+1}B_l)/p'_c\ge 0.$

Consequently, the exponent on the right-hand side of the estimate \eqref{6.8} is uniformly bounded from above, which verifies the claimed estimate \eqref{6.5}.

\medskip
Assume next that $l=l_0-1.$ Observe that in this case, by following Varchenko's algorithm one observes that the left endpoint $(A_{l_0-1},B_{l_0-1})$ of the edge $[(A_{l_0-1},B_{l_0-1}),$ $(A_{l_0},B_{l_0})]$ of the Newton polyhedron of $\pad$ belongs also to the Newton polyhedron of $\phi$ and lies on the principal line $L=L_\ka$ of $\N(\phi),$ whose slope is the reciprocal of $\ka_2/\ka_1=m.$ Thus, if we formally replace $h_{l_0-1}$  by $d$ in the previous argument (compare also Remark \ref{r1} (a)), it is easily seen that the previous argument works in exactly the same way.

\medskip

What remains to be considered are the generalized transition domains $E_{l_\pr-1}$ in the cases (c1) and (c2). Observe that in this case Condition (R) implies that $\Phi:=\pad$ satisfies the factorization  hypothesis of Lemma \ref{equivcond2}. We may therefore argue  in a similar way as before, by applying Lemma \ref{equivcond2} in place of Lemma \ref{equivcond},   and obtain the estimate
\begin{equation}\label{6.10}
\int_S |\widehat f|^2\,d\mu_{j,k}\le C 2^{-j-k+2\frac {(m+1+A)j+Bk}{p'_c}} \|f\|^2_{L^{p_c}(\RR^3)},
\end{equation}
where here $B=h$ is the height of $\phi,$ and where now we may only assume that 
\begin{equation}\label{indecond2}
a_lj+M\le k
\end{equation}
Since, by the definition of the $r$-height, we have $p_c'\ge 2h_{l_\pr-1}+2=2B$ (compare \eqref{hl}), we see that 
$-1+\frac{2B}{p_c'}\le 0.$
We may thus estimate the exponent in \eqref{6.10} by 
\begin{eqnarray*}
-j-k+2\frac {(m+1+A)j+Bk}{p'_c}&\le& -j\left[a+1- 2\frac {m+1+A+aB}{p'_c}\right]+\Big(-1+\frac{2B}{p_c'}\Big)M\\
&\le& -j\frac{a+1}{p'_c}\left[p'_c- 2\frac {m+1+A+aB}{a+1}\right].
\end{eqnarray*}
And, in the case (c1), arguing as before  we see that $2(m+1+A+aB)/(a+1)=2(h_{l_\pr}+1)\le p'_c.$

Finally, in the case (c2), we have $m=a.$ Moreover, the point  $(A,B)$ lies on the principal line $L$ of $\N(\phi),$ so that $\ka_1 A+\ka_2B=1,$ i.e., $A+aB=1/\ka_1.$ This shows that  
$$2\frac{m+1+A+aB}{a+1}=2(1+\frac{1}{\ka_1+\ka_2})=2(1+d)\le p'_c.
$$
We thus see that the uniform estimate \eqref{6.5} is valid also for the generalized transition domains.
\end{proof}

\setcounter{equation}{0}
\section{Restriction estimates in the  domains $D_l,\ l<l_\pr,$ when $\hl(\phi)\ge 2$ }\label{homogeneous}

\bigskip
We shall now consider the domains $D_l, \quad l=l_0,\dots, l_\pr-1,$ from Section \ref{prep}, which are homogeneous in the adapted coordinates.  Following again \cite{IKM-max} we can  localize to these domains by means of  cut-off functions
$$
\rho_l\x:=\chi_0\Big(\frac {x_2-\psi(x_1)}{N x_1^{a_l}}\Big)-\chi_0\Big(\frac {x_2-\psi(x_1)}{\ve x_1^{a_{l}}}\Big),\qquad l=l_0,\dots,l_\pr-1,
$$
where $\chi_0$ is as in the previous section. Recall that such domains do  appear only  in the cases (a), (b) and (c1).

\begin{proposition}\label{s7.1}
Let  $\hl(\phi)\ge 2,$ and assume that  $l<l_\pr.$ Then, if $\ve>0$ is chosen sufficiently small and $N>0$ sufficiently large, 
$$
\Big(\int_S |\widehat f|^2\,d\mu^{\rho_l}\Big)^{1/2}\le C_{p} \|f\|_{L^p(\RR^3)},\qquad  f\in\S(\RR^3),
$$
whenever $p'\ge p'_c .$
\end{proposition}

\begin{proof}

Similarly to the proof of Proposition \ref{awayjet}, we denote by $\{\de_r\}_{ r>0}$ the dilations associated to the weight $\ka^l,$ i.e., $\de_{r}y:= (r^{\ka_1^l }y_1,\, r^{\ka_2^l }y_2),$ where by $y$ we again denote our adapted coordinates.
Recall that the $\ka^l$-principal part $\pad_{\ka^l}$ of $\pad$ is homogeneous of degree one with respect to these dilations, and that we are interested in a  $\ka^l$- homogeneous domain of the form
$D^a_l=\{\y:0<y_1<\de, \, \ve  y_1^{a_l}< |y_2|\le N x_1^{a_l}\}$ with respect to the $y$-coordinates, where $\de>0$ can still be chosen as small as we please.
\medskip

We shall prove that, given any real number $c_0$ with $\ve\le c_0\le N,$ there is some $\ve'>0$ such that the desired restriction estimate holds true on the domain $D(c_0)$ in $x$-coordinates corresponding to the homogeneous domain 
$$
D^a(c_0):=\{\y:0<y_1<\de, \,  |y_2-c_0 y_1^{a_l}|\le \ve' y_1^{a_l}\}
$$
in $y$-coordinates. Since we can cover the closure of $D^a_l$ by a finite number of such narrow domains, this will imply Proposition \ref{s7.1}.

\smallskip
 We can essentially localize to a domain $D(c_0)$ by means of a cut-off function 
$$
\rho_{(c_0)}\x:=\chi_0\Big(\frac {x_2-\psi(x_1)-c_0x_1^{a_l}}{\ve'x_1^{a_l}}\Big).
$$
Let us again fix a suitable smooth cut-off function $\chi\ge 0$ on $\RR^2$ supported in an annulus $\A:=\{x\in \RR^2:1/2\le |y|\le R\}$ such that the functions $\chi^a_k:=\chi\circ \de_{2^k}$ form a partition of unity. In the original coordinates $x,$ these correspond to the functions $\chi_k(x):=\chi^a_k(x_1,x_2-\psi(x_1)).$ We then decompose the measure $\mu^{\rho_{(c_0)}}$ dyadically as 
\begin{eqnarray}\label{7.1}
\mu^{\rho_{(c_0)}}=\sum_{k\ge k_0}\mu_k,
\end{eqnarray}
where $\mu_k:= \mu^{\chi_k\rho_{(c_0)}}.$  Notice that by choosing the support of $\eta$ sufficiently small, we can choose $k_0\in \NN$ as large as we need. It is also important to observe that  this decomposition can essentially we achieved by means of a dyadic decomposition with respect to the variable $x_1,$  which again allows to apply Littlewood-Paley theory! 

Moreover, changing to adapted coordinates in the integral defining $\mu_k$ and scaling by $\de_{2^{-k}}$ we find that
\begin{eqnarray*}
\laa\mu_k,\, f\ra &=&2^{-k|\ka^l|}\int f(2^{-\ka_1^l k}x_1,\, 2^{-\ka_2^l k}x_2+2^{-m\ka_1^l k}x_1^m\om(2^{-\ka_1^lk} x_1),\, 2^{-k}\phi_k(x))\\
&&\hskip7cm  \eta(\de_{2^{-k}}x)\chi(x)\,\chi_0\Big(\frac {x_2-c_0x_1^{a_l}}{\ve'x_1^{a_l}}\Big)\, dx,
\end{eqnarray*}
where
\begin{equation}\label{7.2}
\phi_k(x):=2^k\pad(\de_{2^{-k}}x)=\pad_{\ka^l} (x)+ \mbox{error terms of order } O(2^{-\delta k})
\end{equation}
with respect to the $C^\infty$ topology (and  $\delta>0).$

We consider the corresponding normalized  measure $\nu_k$ given by
\begin{eqnarray*}
\laa\nu_k,\, f\ra :=\int f(x_1,\, 2^{(m\ka_1^l-\ka_2^l) k}x_2+x_1^m\om(2^{-\ka_1^lk} x_1),\, \phi_k(x))
\, \tilde\eta(x)\, dx,
\end{eqnarray*}
with amplitude $\tilde\eta(x):=\eta(\de_{2^{-k}}x)\chi(x)\chi_0\Big((x_2-c_0x_1^{a_l})/(\ve'x_1^{a_l})\Big).$

Observe  that the support of the integrand is contained in the thin neighborhood 
$$U(v):=\A\cap \{(x_1,\,x_2):  |x_2-c_0x_1^{a_l}|\le 2\ve'x_1^{a_l}\}$$
 of 
 $v=v(c_0):=(1,c_0),$
and that the measure $\nu_k$ is supported on the  hypersurface 
$$
S_k:=
\{g_k(x_1,x_2):=(x_1,\,2^{(m\ka_1^l-\ka_2^l) k}x_2+x_1^m\om(2^{-\ka_1^lk} x_1),\, \phi_k\x): \x\in U(v)  \},
$$
which, for $k$ sufficiently large, is a small perturbation of the limiting variety 
$$
S_\infty:=
\{g_\infty\x:=(x_1,\,\om(0)x_1^m\,, \pad_{\ka^l}(x)): \x\in U(v)  \},
$$
since $m\ka_1^l-\ka_2^l<a_l\ka_1^l-\ka_2^l=0$ and since $\phi^k$ tends to $\pad_{\ka^l}$ because of \eqref{7.2}.
The corresponding limiting measure will be denoted by $\nu_\infty.$

By Littlewood-Paley theory (applied to the variable $x_1$) and interpolation,  in order to prove  the desired restriction estimates for the measure $\mu^{\rho_{(c_0)}},$ it suffices again to prove uniform restriction estimates for the measures $\mu_k,$ i.e., 
\begin{equation}\label{hestimate1}
\Big(\int |\widehat f|^2\,d\mu_k)^{1/2}\le C\,   \|f\|_{L^{p_c}}.
\end{equation}
with a constant $C$ not depending on $k\ge k_0.$
We shall obtain these by first proving restriction estimates for the measures $\nu_k.$ 

\smallskip
Indeed, we shall prove that for $\ve'$ sufficiently small,  the estimate 
\begin{equation}\label{1homoestimate}
\Big(\int |\widehat f|^2\,d\nu_k\Big)^{1/2}\le C \|f\|_{L^{p_c}}
\end{equation}
holds true, with a constant $C$ which  does not depend on $k$. Then, after re-scaling, estimate \eqref{1homoestimate} implies the  following estimate for $\mu_k:$
\begin{equation}\label{1orhomoestimate}
\Big(\int |\widehat f|^2\,d\mu_k\Big)^{1/2}\le C\,  2^{-k\Big(\frac{|\ka^l|}{2}-\frac{\ka_1^l(1+m)+1}{p_c'}\Big)} \|f\|_{L^{p_c}}.
\end{equation}
But, by \eqref{hl} (resp. \eqref{hl+1})Êwe have that 
$$
\frac{|\ka^l|}{2}-\frac{\ka_1^l(1+m)+1}{p_c'}=\frac{|\ka^l|}2\Big(1-\frac {2(h_l+1)}{p_c'}\Big),
$$
where, by definition, $p_c'\ge 2(h_l+1).$ This shows that the exponent on the right-hand side of 
\eqref{1orhomoestimate} is less or equal to zero, which verifies \eqref{hestimate1}.

\medskip
We turn to the proof of \eqref{1homoestimate}. Recall that $v=(1,c_0).$ Depending on the behavior of $\pad_{\ka^l}$ near $v,$ we shall distinguish between two cases.  

\medskip

{\bf 1. Case.} $\pa_2\pad_{\ka^l}(v) \neq 0$.  
This assumption implies  that we may use $y_2:=\pad_{\ka^l}(x_1,\, x_2)$ in place of $x_2$ as a new coordinate for $S_\infty$  (which thus is a hypersurface, too), and then also for $S_k,$ in place of $x_2,$ provided $\ve'$ is chosen small enough and $k$ sufficiently large. Since $x_1\sim 1$ on $U(v),$ this then shows that $S_k$ is a hypersurface with one non-vanishing principal curvature. Therefore  we can
again apply Greenleaf's restriction theorem from  \cite{greenleaf} and  obtain that  for $p'\ge 6$  and $k$ sufficiently large the estimate
$$
\Big(\int |\widehat f|^2\,d\nu_k\Big)^{1/2}\le C_{p} \|f\|_{L^p}
$$
holds true, with a constant $C_p$ which  does not depend on $k$. This applies in particular to $p=p_c,$ which gives \eqref{1homoestimate}.

\medskip

{\bf 2. Case.} $\pa_2\pad_{\ka^l}(v)=0.$ Then $v=(1,c_0)$ is a real root of $\pa_2\pad_{\ka^l},$ of multiplicity, say,  $B-1\ge 1,$  so that a Taylor expansion with respect to $x_2$ around $c_0$ and homogeneity show that 
$$
\pa_2\pad_{\ka^l}(x_1,\, x_2)=(x_2-c_0x_1^{a_l})^{B-1} \tilde Q\x,
$$
where $\tilde Q$ is a $\ka^l$-homogenous  smooth function  in $U(v)$   such that $\tilde Q(v)\neq 0.$ Integrating in $x_2,$ and making  again use of the $\ka^l$-homogeneity of $\pad_{\ka^l},$ we find that 
\begin{equation}\label{7.6}
\pad_{\ka^l}(x_1,\, x_2)=(x_2-c_0x_1^{a_l})^B x_2^{B_l} Q\x+c_1x_1^{1/\ka_1^l},
\end{equation}
where  $Q$ is a $\ka^l$-homogenous  smooth function such that $Q(1,c_0)\neq0$ and $Q(1,0)\ne 0$ (recall  that $c_0\ne 0$). Here, $c_1\in\RR$ could possibly be zero (iff $\nabla \pad_{\ka^l}(v)=0$).

We claim that 
\begin{equation}\label{7.7}
B<d/2,
\end{equation}
where again $d=d(\phi).$ Indeed, observe first that  the vertex $(A_l,B_l)$ lies above or on the bi-sectrix, so that $1=\ka^l_1A_l+\ka^l_2B_l\le (\ka^l_1+\ka^l_2)B_l=B_l/d_l,$ where $d_l:=d_h(\pad_{\ka^l})$ denotes the homogenous distance of  $\pad_{\ka^l}.$  But, since $a_l>m,$ so that the edge $\ga_l$ is less steep than the line $L$ (which intersects the bi-sectrix at $(d,d)$),  we have $d_l>d,$ hence $B_l>d.$ Note that for the same reason, $1/\ka_2>1/\ka^l_2.$  
Because $\pad_{\ka^l}$ is ${\ka^l}$- homogeneous of degree 1, by \eqref{7.6}  we thus have 
$$
1\ge (B_l+B)\ka^l_2>(d+B)\ka^l_2,
$$
which implies that 
$$
B<\frac 1{\ka^l_2}-d\le \frac 1{\ka_2}-\frac 1{\ka_1+\ka_2}=\frac d m \le \frac d2.
$$
Let us localize to frequencies  of size $\la> 1$ by putting 
$$
\widehat{\nu_k^\la}(\xi):=\chi_1\Big(\frac \xi{\la}\Big)\widehat{\nu_k}(\xi),
$$
 where $\chi_1$ is  a smooth bump function supported where $|\xi|\sim 1.$ We claim that the measures $\nu_k^\la$ satisfy the following estimates, uniformly in $k\ge k_0,$ provided $k_0$ is sufficiently large and $\ve'$ sufficiently small:
\begin{eqnarray}\label{7.8}
\|\widehat{\nu_k^\la}\|_\infty&\le& C \la^{-1/B}\,;\\ 
\|\nu_k^\la\|_\infty&\le& C \la^{2-1/B}\,. \label{7.9}
\end{eqnarray}
Indeed,  
$$
\widehat{\nu_k^\la}(\xi)=\chi_1\Big(\frac \xi{\la}\Big)\, \int  e^{-i\Big[\xi_1x_1+\xi_2 \Big(2^{(m\ka_1^l-\ka_2^l) k}x_2+x_1^m\om(2^{-\ka_1^lk} x_1)\Big)+ \xi_3\phi_k(x)\Big]}
\, \tilde\eta(x)\, dx,
$$
which, in the limit as $k\to\infty,$ simplifies as 
$$
\widehat{\nu_\infty^\la}(\xi)=\chi_1\Big(\frac \xi{\la}\Big)\,\int  e^{-i[\xi_1x_1+\xi_2\om(0)x_1^m+ \xi_3\pad_{\ka^l}(x)]}
\, \tilde\eta(x)\,  dx.
$$
Now, if $|\xi_3|\ge c |(\xi_1,\xi_2)|,$ then an application of van der Corput's lemma to the integration in $x_2$ yields 
$|\widehat{\nu_\infty^\la}(\xi)|\lesssim |\xi_3|^{-1/B}$ (cf. \eqref{7.6}), and if $|\xi_3|\ll |(\xi_1,\xi_2)|,$ we may apply van der Corput's lemma to the $x_1$-integration and obtain $|\widehat{\nu_\infty^\la}(\xi)|\lesssim |(\xi_1,\xi_2)|^{-1/2}.$ Since $B\ge 2,$ and because  van der Corput's estimates are  stable under small perturbations, we thus obtain \eqref{7.8}.

In order to verify  \eqref{7.9}, observe that 
\begin{eqnarray*}
\nu_\infty^\la(x_1,x_2,x_3)=\la^3\int \widehat{\chi_1}(\la(x_1-y_1),\la(x_2-\om(0)y_1^m), \la(x_3-\pad_{\ka^l}(y_1,y_2))\,\tilde\eta(y)\, dy_1dy_2,
\end{eqnarray*}
hence 
\begin{eqnarray*}
|\nu_\infty^\la(x_1,x_2,x_3)|\le \la^3\int \rho(\la x_1-\la y_1)\,\rho(\la x_3-\la\pad_{\ka^l}(y_1,y_2))\,|\tilde\eta|(y_1,y_2)\, dy_1dy_2,
\end{eqnarray*}
where $\rho$ and $\eta_1$ are  suitable, non-negative Schwartz functions, and $\eta_1$ localizes again to $U(v).$ However, since $|\pa_2^B\pad_{\ka^l}(y_1,y_2))|\simeq 1$ on the domain of integration, classical  sublevel estimates, originating in work by van der Corput \cite{vdC} (see also \cite{arhipov}, and \cite{ccw},\cite{grafakos}),  essentially would imply that the integral with respect to $y_2$ can be estimates by $O(\la^{-1/B}),$ uniformly in $y_1$ and $\la$  (at least, if $\rho$ had compact support). To  be more precise, we can argue  as follows: By means of Fourier inversion, re-write 
\begin{eqnarray*}
&&\int \rho(\la x_1-\la y_1)\,\rho(\la x_3-\la\pad_{\ka^l}(y_1,y_2))\,|\tilde\eta|(y_1,y_2)\, dy_1dy_2,\\
&=&\int \rho(\la x_1-\la y_1)\,\hat \rho(s) e^{is(\la x_3-\la\pad_{\ka^l}(y_1,y_2))}\,|\tilde\eta|(y_1,y_2)\, dy_2\, ds \,dy_1,
\end{eqnarray*}
and then apply again van der Corput's estimate to the $y_2$-integration. This yields 
\begin{eqnarray*}
&&\Big|\int \rho(\la x_1-\la y_1)\,\rho(\la x_3-\la\pad_{\ka^l}(y_1,y_2))\,|\tilde\eta|(y_1,y_2)\, dy_1dy_2\Big|,\\
&\lesssim &\int \rho(\la x_1-\la y_1)\,|\hat \rho(s)| (1+\la|s|)^{-1/B}\,|\tilde\eta|(y_1,y_2)\, dy_2\, ds \,dy_1,
\end{eqnarray*}
which is easily estimated by $C\la^{-1-1/B},$ so that we obtain $|\nu_\infty^\la(x_1,x_2,x_3)|\le C\la^{2-1/B}.$ Observing that our argument is again stable under small perturbations, we thus obtain \eqref{7.9}.

Interpolating  the estimates \eqref{7.8} and \eqref{7.9}, it is again easily seen that we can sum the corresponding estimates over dyadic $\la$'s and obtain the  $L^p$-$L^2$ restriction estimate
$$
\Big(\int |\widehat f|^2\,d\nu_k\Big)^{1/2}\le C_{p} \|f\|_{L^p}
$$
whenever $p'> 4B,$ uniformly in $k,$ for $k$ sufficiently large. 

The restriction estimates above  are valid in particular for $p'= p'_c,$ since, by \eqref{7.7}, $B<d/2,$ so that $ p'_c\ge 2d+2>4B.$ We have thus again verified \eqref{1homoestimate}.
\end{proof}

In combination with Proposition \ref{inter1}, we immediately obtain

\begin{cor}\label{c2}
The restriction estimate in Proposition \ref{nearjet} holds true in the Case (c), i.e., when the principal face of the Newton polyhedron of $\pad$ is unbounded.
\end{cor}

\begin{remark}\label{tors} When $\hl\ge 5,$ then  the subcase of Case 2,  where  $\pa_2\pad_{\ka^l}(v)=0$  and $\pa_1\pad_{\ka^l}(v)\ne0,$ could  be handled alternatively by means of  Drury's Fourier restriction theorem for curves with non-vanishing torsion (cf. Theorem 2 in \cite{drury}).  
This approach will allow to treat the analogous case also for the remaining domain $D_\pr,$  provided $\hl\ge 5,$  since it does not require the condition $B<d/2,$ which may not  hold true in $D_\pr,$ 
\end{remark}

Indeed, if $\pa_1\pad_{\ka^l}(v)\ne0,$ then  $c_1\ne 0$ in \eqref{7.6}. Moreover,  
\begin{equation}\label{torsion}
2\le m <a_l=\ka_2^l/\ka_1^l<1/\ka_1^l,
\end{equation}
since $\ka_1^lA_{l-1}+\ka_2^lB_{l-1}=1$ with $B_{l-1}\ge h>1,$ so that $\ka_2^l<1$.
Observe next  that $F: (x_1,c)\mapsto(x_1,cx_1^{a_l})$ provides local smooth coordinates near $v=(1,c_0),$ since the Jacobian $J_F$ of $F$ at the point $(1,c_0)$ is given by $J_F(1, c_0)=1.$ We may therefore fibre the variety $S_\infty$ into the family of curves 
$$
\ga_c(x_1) := g_\infty(F(x_1,c)=(x_1,\om(0)x_1^m,\pad_{\ka^l}(F(x_1,c)),\qquad c\in V(c_0),
$$
where $V(c_0)$ is a sufficiently small neighborhood of $c_0,$
provided  $\ve'$ is chosen sufficiently small. But, \eqref{torsion} implies that the curve $\ga_{c_0}(x_1)=(x_1,\om(0)x_1^m,c_1x_1^{1/\ka_1^l})$ has non-vanishing torsion near $v_1,$ since $v_1\ne0,$ and so the same is true for the curves $\ga_c$ when $c$ is sufficiently close to $c_0.$  

If we fibre in a similar way the surface $S_k$ into the family of  curves 
$$
\ga^k_c(x_1) := g_k(F(x_1,c), \qquad c\in V(c_0),
$$
then for $k$ sufficiently large and $V(c_0)$ sufficiently small, these curves will have non-vanishing torsion uniformly bounded from above and below, and  the measure $\nu_k$ will decompose into the direct integral
\begin{eqnarray*}
\laa\nu_k,\, f\ra =\iint f(\ga^k_c(x_1) )\, \tilde\eta(x_1,c)\, dx_1\, dc=\int_{V(c_0)}\int_{W(v_1)} f\, d\Gamma_c\, dc,
\end{eqnarray*}
where $\tilde\eta$ is a smooth function with compact support in $W(v_1)\times V(c_0)$  and $W(v_1)$ a sufficiently small neighborhood of $v_1,$ where $d\Gamma_c$ is a measure which has a smooth density with respect to the arclength measure on the curve $\ga^k_c.$

We may  thus apply Drury's Fourier restriction theorem for curves with non-vanishing torsion (cf. Theorem 2 in \cite{drury}) to the measures
$d\Gamma_c$ and obtain that
$$
\Big(\int_{W(v_1)} |\hat f|^2\, d\Gamma_c\Big)^{\frac 12}\le C_{p}\|f\|_{L^p(\RR^3)},
$$
 provided $p'>7$ and $2\le p'/6,$ i.e., if  $p'\ge 12.$ 
The  constant $C_{p}$ will then be  independent of $c$ provided the neighborhoods $V(c_0)$  and $W(v_1)$ are sufficiently small and $k$ is sufficiently large.  But, if $\hl\ge 5, $ then we do have $p_c'\ge 2(\hl+1)\ge 12,$ so that we 
do obtain  estimate \eqref{1homoestimate}  also in this way.

\setcounter{equation}{0}
\section{Restriction estimates in the  domain $D_\pr$  when $\hl(\phi)\ge  5$ }\label{nearrootjet}

\bigskip
What remains to be studied is the piece of the surface $S$ corresponding to the domain $D_\pr,$ in the cases (a) and (b) of Section \ref{prep}, i.e.,
\begin{equation}\label{restdomain3}
D_\pr:=\begin{cases}  
\{ \x:|x_2-\psi(x_1)|\le N x_1^a\} &   \mbox{ in Case (a)},   \\
      \{ \x:|x_2-\psi(x_1)|\le \ve x_1^{a}\}, &   \mbox{ in Case (b)},
\end{cases}
\end{equation}
where $N$ is sufficiently large in Case (a),  and $\ve$ may be assumed to be sufficiently small in Case (b). 
  Our goal will to prove 

\begin{proposition}\label{rdomain1}
Assume that $\hl(\phi)\ge 5,$ and that we are in Case (a) or (b) of Section \ref{prep}. When $N$ is sufficiently large in Case (a),  respectively  $\ve$ is sufficiently small in Case (b), then
$$
\Big(\int_{D_{\pr}} |\widehat f|^2\,d\mu^{\rho_{l_\pr}}\Big)^{1/2}\le C_{p} \|f\|_{L^p(\RR^3)},\qquad  f\in\S(\RR^3),
$$
whenever $p'\ge p'_c .$
\end{proposition}

In the domain $D_{\pr},$ the upper bound $B<d/2$ for the multiplicity $B$ of real roots will in general no longer be true, not even the weaker condition $B<h^r(\phi)/2,$ which would still suffice for the previous argument,  as the following example shows.

\begin{example}\label{ex2}
{\rm
$$
\phi\x:=(x_2-x_1^2-x_1^3)(x_2-x_1^2-x_1^4)^3.
$$
Here, $\phi_\pr\x=(x_2-x_1^2)^4,$ the multiplicity of the root $x_1^2$ satisfies $4>d(\phi)=8/3,$ so that the coordinates $\x$ are not adapted to $\phi.$ Adapted coordinates are given by $y_1:=x_1,\ y_2:=x_2-x_1^2,$ and in these coordinates $\phi$ is given by 
$$
\pad\y= (y_2-y_1^3)(y_2-y_1^4)^3.
$$
$\N(\pad)$ has three vertices $(A_0,B_0):=(0,4),\ (A_1,B_1):=(3,3)$ and $(A_2,B_2):=(0,15),$ with corresponding edges 
$\ga_1:=[(0,4),(3,3)]$ and $\ga_2:=[(3,3),(0,15)],$ and  associated weights $\ka^1:=(1/12,1/4)$ and $\ka^2:=(1/15,4/15).$  Moreover, one easily computes by means of \eqref{hl} that $h_1=11/4$ and $h_2=13/5.$ We thus see that  $h^r(\phi)=h_1=11/4.$ The multiplicity of the root $y_1^3$ associated to the first edge $\ga_1$ lying above the bi-sectrix is $1<(8/3)/2$ and thus satisfies the condition \eqref{7.7}, whereas the root Ê$y_1^4$ of multiplicity $B=3$ associated to the edge $\ga_2$ below the bi-sectrix does not even satisfy $B<h^r(\phi),$ since $3> 11/4.$ 

}
\end{example}
\medskip

The study of the domain $D_\pr$   will therefore require finer decompositions  into further transition and homogeneous domains (with respect to further weights).  These will be devised  by means of an iteration  scheme, resembling somewhat  Varchenko's algorithm for the construction of adapted coordinates. Note  that the latter algorithm also shows that the principal root jet $\psi$ is actually a polynomial
\begin{equation}\label{polrootjet}
\psi(x_1)=cx_1^m+\dots +c_\pr x_1^a
\end{equation}
of degree $a=a_{l_\pr}$ in the cases (a) and (b) (cf. \cite{IM-ada}).

\subsection{First step of the algorithm}

\medskip
Let us begin with  Case (a), where  $D_\pr=\{\x: 0<x_1<\de, \, |x_2-\psi(x_1)|\le Nx_1^{a}\},$ with a possibly large constant $N>0.$ We then put  $D_{(1)}:=D_\pr,$  $\phi^{(1)}:=\pad, \, \psi^{(1)}:=\psi$ and $a_{(1)}:=a, \ka^{(1)}:=\ka^{l_\pr},$ so that $D_{(1)}$ can be re-written as 
$$
D_{(1)} =\{\x: 0<x_1<\de, \, |x_2-\psi^{(1)}(x_1)|\le Nx_1^{a_{(1)}}\}.
$$

   As in the discussion of the domains $D_l$ in the previous section, we can cover the domain $D_{(1)}$ by finitely many narrow domains of the form
$$
D_{(1)}(c_0):=\{\x:0<x_1<\de, \,  |x_2-\psi(x_1)-c_0 x_1^{a_{(1)}}|\le \ve x_1^{a_{(1)}}\},
$$
where $\ve>0$ can be chosen as small as we need, and where $0\le c_0\le N.$  Fix any of these domains, and put again $v:=(1,c_0).$ 

\medskip
We distinguish between the cases where $\pa_2\phi^{(1)}_{\ka^{(1)}}(v) \neq 0$ (Case 1),  $\pa_2\phi^{(1)}_{\ka^{(1)}}(v) =0$ and $\pa_1\phi^{(1)}_{\ka^{(1)}}(v) \ne 0$ (Case 2),  and the case where $\nabla \phi^{(1)}_{\ka^{(1)}}(v) =0$ (Case 3).

\medskip
Now, in Case 1, we can argue as in the corresponding case in Section \ref{homogeneous}, since  our arguments in that case did not make use of the condition $l>l_\pr.$ 

\medskip
In Case 2, the  argument given in Section \ref{homogeneous}  may fail, since it made use of the estimate $B<d/2,$ which here no longer may hold true. However, as explained in Remark \ref{tors},  if $\hl\ge 5,$ we may use the alternative argument based on Drury's restriction estimate for curves in this case.

\medskip
If Case 3 does not appear for any choice of $c_0,$ then we stop our algorithm and are done. 

\medskip
Otherwise, assume  Case 3 applies to $c_0,$ so that  $c_0 x_1^{a_{(1)}}$ is a root of $\phi^{(1)}_{\ka^{(1)}},$ say of multiplicity $M_1\ge 2.$  In this case, we  define new coordinates $y$ by putting
\begin{equation}\label{newcoordinates}
y_1:=x_1\qquad \mbox{and  }\ y_2:=x_2-\psi^{(2)}(x_1),
\end{equation}
where 
$$
\psi^{(2)}(x_1):=\psi(x_1)+c_0 x_1^{a_{(1)}}.
$$
We denote by $x=s_{(2)}(y)$ the corresponding change of coordinates, which  in general is  a fractional shear only, since the exponent $a_{(1)}=a$ may be non-integer (but rational). In these coordinates $\y,$ $\phi$ is given by $\phi^{(2)}:= \phi\circ s_{(2)},$ and the domain $D_{(1)}(c_0)$ becomes  the  domain
$$
D_{(1)}'^a:=\{\y:0<y_1<\de, \,  |y_2|\le \ve y_1^{a_{(1)}}\},
$$
which is still $\ka^{(1)}$ homogeneous.

Let  us see to which extent the Newton  polyhedra of $\phi^{(1)}$ and $\phi^{(2)}$ will differ. 

\smallskip
{\bf Claim 1.} The Newton polyhedra of  $\phi^{(1)}$ and $\phi^{(2)}$ agree in the region above the bi-sectrix. In particular, the line $\Delta^{(m)}$ intersects  the boundary of the augmented Newton polyhedron $\N^r(\phi^{(1)})=\N^r(\pad)$ at the same point as the  augmented Newton polyhedron $\N^r(\phi^{(2)})$ of $\phi^{(2)},$ so that we can use the modified ``adapted'' coordinates  \eqref{newcoordinates} in place of our earlier adapted coordinates to compute the $r$-height of $\phi.$
\smallskip

To see this, observe that $\phi^{(2)}\x=\phi^{(1)}(x_1,x_2+c_0 x_1^{a_{(1)}}),$ where the exponent $a_{(1)}$ is just the reciprocal of the slope of the line containing the principal face of $\phi^{(1)}=\pad.$  This implies that the edges of $\N(\phi^{(1)})$ and $\N(\phi^{(2)})$ which lie strictly above the bi-sectrix  and do not intersect it are the same (compare corresponding discussions in \cite{IM-ada}). Moreover, if 
$\ga_{(1)}=[(A_{(0)},B_{(0)}), (A_{(1)},B_{(1)})]= [(A_{l_\pr-1}, B_{l_\pr-1}),(A_{l_\pr}, B_{l_\pr})]$ is the principal face of  $\N(\phi^{(1)}),$  then it is easy to see that the principal face of $\N(\phi^{(2)})$ is given by the edge $\ga'_{(1)}:=[(A_{(0)},B_{(0)}),(A'_{(1)},B'_{(1)})],$
where
\begin{equation}\label{edgechange}
A'_{(1)}:=A_{(1)}+a_{(1)}(B_{(1)}-M_1),\quad B'_{(1)}=M_1,
\end{equation}
(write  $\phi^{(1)}_{\ka^{(1)}}$ in the normal form \eqref{normalhom} and use that $c_0 x_1^{a_{(1)}}$ is a root of of multiplicity $M_1$ of $\phi^{(1)}_{\ka^{(1)}}$). Observe also that $M_1\le h,$ because $\pad$ is in adapted coordinates. We thus see  that the right endpoint  of $\ga'_{(1)}$ still lies on or below the bi-sectrix. This proves the  claim.

\medskip
Observe that our considerations show that it suffices to study the contributions of  narrow domains of the form  
\begin{equation}\label{d1'}
D_{(1)}'=\{ \x: 0<x_1<\de,  |x_2-\psi^{(2)}(x_1)|\le \ve x_1^{a_{(1)}}\}
\end{equation}
 in place of $D_{(1)}$ (these actually depend on the choice or real root of  $\phi^{(1)}_{\ka^{(1)}}$  - this corresponds to a ``fine splitting'' of roots of $\phi,$ in the case where $\phi$ is analytic).

\medskip
{\it Case A.} $\N(\phi^{(2)})\subset \{(t_1,t_2): t_2\ge B'_{(1)}=M_1\}.$
In this case, we again stop our algorithm. 

\smallskip
{\it Case B.} $\N(\phi^{(2)})$ contains a point below the line  where $t_2= B'_{(1)}=M_1.$

\medskip
\noi Then $\N(\phi^{(2)})$ will contain a further compact edge 
$$
\ga_{(2)}=[(A'_{(1)},B'_{(1)}),(A_{(2)},B_{(2)})],
$$
so that $(A'_{(1)},B'_{(1)})$ is a vertex at which the edges $\ga'_{(1)}$ and $\ga_{(2)}$ meet.
Determine the weight $\ka^{(2)}$ by requiring that  $\ga_{(2)}$ lies on the line 
$$
\ka^{(2)}_1t_1+\ka^{(2)}_2t_2=1,
$$
and put $ a_{(2)}:=\ka^{(2)}_2/\ka^{(2)}_1.$ Then  clearly $a_{(1)}<a_{(2)}.$

 \smallskip
Next, we decompose the domain  $D'_{(1)}$ Êinto the domains
$$
E_{(1)}:=\{ \x: 0<x_1<\de, \, N x_1^{a_{(2)}}< |x_2-\psi^{(2)}(x_1)|\le \ve x_1^{a_{(1)}}\}
$$
and 
$$
D_{(2)}:=\{ \x: 0<x_1<\de, \,  |x_2-\psi^{(2)}(x_1)|\le N x_1^{a_{(2)}}\},
$$
where $N>0$ will be a sufficiently large constant. 

\medskip
The contributions by the transition domain  $E_{(1)}$ can be estimated in exactly the same way as we did for the domains $E_l$  in Section \ref{transition}. Indeed, notice that our arguments for the domains $E_l$ did apply to any $l\ge l_0$  as long as $B_l\ge 1,$ so that this statement is immediate when $c_0=0,$ where the coordinates $y$ in \eqref{newcoordinates} do agree with our original adapted coordinates. When $c_0\ne 0,$ there are two minor twists in the arguments needed: firstly, observe that Lemma \ref{equivcond} remains valid for $\Phi=\phi^{(2)}$ and the domain $$
E^a_{(1)}:=\{ \y: 0<y_1<\de, \, 2^M y_1^{a_{(2)}}< |y_2|\le 2^{-M} y_1^{a_{(1)}}\}
$$
corresponding to the domain $E_{(1)}$ in the coordinates \eqref{newcoordinates} when $\ve=2^{-M}$ and $N=2^M.$ The fact that $a_{(2)}$ may be non-integer, but rational, say $a_{(2)}=p/q,$ with $p,q\in\bN,$ requires minor changes of the proof only: just consider the Taylor expansion of the smooth function $\Phi(y_1^q,y_2).$ Secondly, if we define in analogy with $h_l$ in \eqref{hl} the corresponding quantity associated to the edges $\ga'_{(1)}$  and $\ga_{(2)}$ of $\N(\phi^{(2)})$ by 
$$
h_{(1)}:=\frac {1+m\ka^{(1)}_1-\ka^{(1)}_2}{\ka^{(1)}_1+\ka^{(1)}_2}=h_{l_\pr} \quad \mbox{and} \quad
h_{(2)}:=\frac {1+m\ka^{(2)}_1-\ka^{(2)}_2}{\ka^{(2)}_1+\ka^{(2)}_2},
$$
then Claim 1 shows that $\max\{h_{(1)},h_{(2)}\}\le h^r(\phi),$ which replaces the condition \hfill\newline$\max\{h_l,h_{l+1}\}\le h^r(\phi)$ that was needed in the proof of Proposition \ref{inter1}.

\subsection{Further  steps of the algorithm}

\medskip
We are thus left with the domains $D_{(2)},$ which formally look exactly like $D_{(2)},$ only with $\psi^{(1)}$ replaced by 
$\psi^{(2)}$  and $a_{(1)}$ replaced by $a_{(2)}.$ This allows to iterate this first step of the algorithm which led from $D_{(1)}$ to $D_{(2)},$ producing in this way  nested  sequences of domains 
$$
D_\pr= D_{(1)}\supset D_{(2)}\supset\cdots \supset D_{(l)}\supset D_{(l+1)}\supset \cdots,
$$
of the form
$$
D_{(l)}:=\{ \x: 0<x_1<\de, \,  |x_2-\psi^{(l)}(x_1)|\le N x_1^{a_{(l)}}\},
$$
where the functions $\psi^{(l)}$  are of the form
$$
\psi^{(l)}(x_1)=\psi(x_1)+\sum_{j=1}^{l-1}c_{j-1} x_1^{a_{(j)}},
$$
with real coefficients $c_j$, and where the exponents $a_{(j)}$ form a strictly increasing sequence
$$
a=a_{(1)}<a_{(2)}<\cdots a_{(l)}<a_{(l+1)}<\cdots
$$
of rational numbers. 

Moreover, each of the domains $D_{(l)}$ will be covered by a finite number of domains $D'_{(l)}$  of the form
\begin{equation}\label{dl'}
D_{(l)}'=\{ \x: 0<x_1<\de,  |x_2-\psi^{(l+1)}(x_1)|\le \ve x_1^{a_{(l)}}\},
\end{equation}
where $\ve>0$ can be chosen as small as we please. These  in return will decompose as 
\begin{equation}\label{e+d}
D'_{(l)}=E_{(l)}\cup D_{(l+1)},
\end{equation}
where $E_{(l)}$ is a transition domain of the form
$$
E_{(l)}:=\{ \x: 0<x_1<\de, \, N x_1^{a_{(l+1)}}< |x_2-\psi^{(l+1)}(x_1)|\le \ve x_1^{a_{(l)}}\}
$$

Putting 
$$
\phi^{(l)}(x_1,x_2):=\phi(x_1, x_2+\psi^{(l)}(x_1)),
$$
one finds that the Newton polyhedron $\N(\phi^{(l+1)})$ agrees with that one of $\pad=\phi^{(1)}$ in the region above the bi-sectrix, and it will have subsequent ``edges'' 
\begin{eqnarray*}
&&\ga'_{(1)}=[(A_{(0)},B_{(0)}), (A'_{(1)},B'_{(1)})], \ga'_{(2)}=[(A'_{(1)},B'_{(1)}), (A'_{(2)},B'_{(2)})], \dots,\\
&&\hskip1cm \ga'_{(l)}=[(A'_{(l-1)},B'_{(l-1)}), (A'_{(l)},B'_{(l)})], \ga_{(l+1)}=[(A'_{(l)},B'_{(l)}), (A_{(l+1)},B_{(l+1)})],
\end{eqnarray*}
crossing or lying below the bi-sectrix, at least (possibly more). In fact, it is possible that some of these ``edges'' degenerate  and become a single point (we then shall still speak of an edge, with a slight abuse of notation).  The edge with index $l$ will lie on a line 
$$
L^{(l)}:=\{(t_1,t_2)\in\bR^2:\ka^{(l)}_1t_1+\ka^{(l)}_2t_2=1\},
$$
where $ a_{(l)}=\ka^{(l)}_2/\ka^{(l)}_1.$ Moreover, $c_{l-1} x_1^{a_{(l)}}$ is any  real root of the $\ka^{(l)}$- homogeneous polynomial $\phi^{(l)}_{\ka^{(l)}},$  of multiplicity $M_l\ge 2.$ Notice that when $\phi$ is real-analytic, then this just means that  $\psi^{(l)}$ is a leading term of a root of $\phi$ belonging to the cluster of roots defined by $\psi$  (in the sense of \cite{phong-stein}). Our algorithm thus follows any possible ``fine splitting'' of the roots belonging to this cluster, and the domains $D_{(l)}$ etc. depend on the branches of these roots that we chose along the way.

By our construction, we see that $M_l=B'_{(l)},$ which shows that the sequence of multiplicities is decreasing, i.e., 
\begin{equation}\label{mdecrease}
M_1\ge M_2\ge \cdots \ge M_l\ge M_{l+1}\ge \cdots .
\end{equation}

\smallskip
Observe also that the  transition domains $E_{(l)}$ can be handled by the same reasoning that we had applied to 
$E_{(1)}.$ 
\smallskip

When will our algorithm stop?  Clearly, this will happen at step $l$ when $\phi^{(l)}_{\ka^{(l)}}$ has no real root, so that only Case 1 and Case 2 will arize at this step. In that case, we do obtain the desired Fourier restriction estimate for the piece of surface corresponding to $D_{(l)},$ just by the same reasoning that we applied in Section \ref{homogeneous}. Otherwise,
we  shall also stop our algorithm  in  step $l$ when 
\begin{equation}\label{stop}
\N(\phi^{(l+1)})\subset \{(t_1,t_2): t_2\ge B'_{(l)}=M_l\}.
\end{equation}
In this situation, the domain which still needs to be understood is  the domain $D_{(l)}'$ given by \eqref{dl'}.

Notice that in this case, Condition (R) implies that there is a function $\tilde\psi^{(l+1)}\sim \psi^{(l+1)}$ such that $\phi$ factors as 
\begin{equation}\label{factor1}
\phi\x=(x_2-\tilde\psi^{(l+1)}(x_1))^{M_l}\tilde\phi\x,\end{equation}
where $\tilde\phi$ is fractionally smooth.
This means  that Lemma \ref{equivcond2} (respectively its immediate extension to fractionally smooth functions) applies to the function $\Phi\y:= \phi(y_1, y_2+\tilde\psi^{(l+1)}(y_1)),$ and since the domain $D_{(l)}'$ can a be regarded  as a generalized transition domain, like the domains $E_{l_\pr-1}$ that appeared when the principal face of $\pad$ was an unbounded horizontal face, we can argue in the same way as we did for the domains $E_{l_\pr-1}$ in Section \ref{transition} to derive the required restriction estimates for the piece of $S$ corresponding to $D_{(l)}'.$

\medskip
There is finally the possibility that our algorithm does not terminate. In this case, \eqref{mdecrease} shows that the sequence  of integers $M_l$ will eventually become constant. We then choose $L$ minimal so that $M_l=M_L$ for all $l\ge L.$ Note that, by our construction, $M_L\ge 2.$  For every $l\ge L+1,$ the point $(A,B):=(A_{(L)},B_{(L)})=(A_L,M_L)$ will be a vertex of  $\N(\phi^{(l)})$ which is contained in the line $L^{(l)},$ whose slope $1/a_{(l)}$ tends to zero as  $l\to\infty,$ and $\N(\phi^{(l)})$ is contained in the half-plane bounded by $L^{(l)}$ from below. 

Notice also that there is a fixed rational number $1/q, $ with $q$ integer,  such that every $a_{(l)}$ is a multiple of $1/q.$ This can be proven in the same way as the corresponding statement in \cite{IKM-max} on p. 240.

We can thus apply a  classical theorem of E. Borel  in a similar way as \cite{IM-ada} in order to  show that there is a smooth function $h$ of $x_1$ whose Taylor series expansion is given  by the  formal series 
$$
h(x_1)\sim \psi(x_1^q)+\sum_{j=1}^{\infty}c_{j-1} x_1^{qa_{(j)}}.
$$
If we put $\psi^{(\infty)}(x_1):= h(x_1^{1/q})$ and set $\phi^{(\infty)}\y:=\phi(y_1,y_2-\psi^{(\infty)}(y_1)),$ it is then easily seen that a straight-forward adaption of the proof Theorem 5.1 in \cite{IM-ada} to show that  $\N(\phi^{(\infty)})\subset \{(t_1,t_2): t_2\ge B\}.$ Therefore, Condition (R)  in Theorem \ref{nonadarestrict} implies  that, possibly after adding a flat function to $\psi^{(\infty)},$ we may assume that $\phi$ factors as $\phi\x=(x_2-\psi^{(\infty)}(x_1))^B\tilde\phi\x,$  which means that the analogue of \eqref{factor1} holds true. We can thus argue as before to complete also this case, hence also the discussion of the Case (a) where the principal face of $\N(\pad)$ is a compact edge.

\medskip

Finally, in Case (b) where  the principal face of $\N(\pad)$ is a vertex,  we have  that $D_\pr=\{ \x:|x_2-\psi(x_1)|\le \ve x_1^{a}\},$ which corresponds to the domain  $D'_{(1)}$ in the discussion of Case  (a). This means that we can just drop the initial step of the algorithm described before, and from then on may proceed as in Case (a).
\medskip

We have thus established the desired restriction estimates for the piece of the surface $S$ corresponding to the remaining domain $D_\pr,$ which completes the proof of  Proposition \ref{rdomain1}, hence also of Theorem \ref{nonadarestrict} in the case where $\hl(\phi)\ge 5.$

\bigskip

What remains open at this stage is the proof  of  the analogue of Proposition \ref{rdomain1} in the case where $2\le \hl(\phi)<5,$ i.e., of 
\begin{proposition}\label{rdomain}
Assume that $2\le \hl(\phi)<5,$ and that we are in Case (a) or (b) of Section \ref{prep}. When $N$ is sufficiently large in Case (a),  respectively  $\ve$ is sufficiently small in Case (b), then
$$
\Big(\int_{D_{\pr}} |\widehat f|^2\,d\mu^{\rho_{l_\pr}}\Big)^{1/2}\le C_{p} \|f\|_{L^p(\RR^3)},\qquad  f\in\S(\RR^3),
$$
whenever $p'\ge p'_c .$
\end{proposition}

The discussion of this case requires substantially more refined techniques and will be the content of \cite{IM-rest2}.
\bigskip

\setcounter{equation}{0}
\section{Necessary conditions, and proof of Proposition \ref{rheight}}\label{necessary}

We now turn to the proof of Theorem \ref{nec}. We shall prove  the following, more general result (notice that we are making no assumption on adaptedness of $\phi$  here): 
 \begin{proposition}\label{necs.1}
Assume that the coordinates $x=\x$ are linearly adapted to $\phi,$ and  that the restriction estimate \eqref{rest1} holds true in a neighborhood of  $x^0=0,$ where $\rho(x^0)\ne 0.$  Consider any  {\it fractional shear}, say on   $H^+,$ given by 
$$
y_1:=x_1, \ y_2:=x_2-f(x_1),
$$
where $f$ is  real-valued and fractionally smooth, but not flat. Let $\phi^f(y)=\phi(y_1,y_2+f(y_1))$ be the function expressing $\phi$ in the coordinates  $y=\y.$ Then necessarily 
\begin{equation}\label{nec1}
p'\ge 2h^f(\phi)+2.
\end{equation}
\end{proposition}

Theorem \ref{nec} will follow by choosing for $f$ the principal root jet $\psi.$ 

\begin{proof}
The proof will be based on suitable Knapp-type arguments.

Let us use the same notation  for the Newton polyhedron of $\phi^f$ as we did for $\phi^a$ in Section \ref{intro}, i.e., 
 the vertices of the Newton polyhedron $\N(\phi^f)$ will be   denoted by 
 $(A_l,B_l), \  l=0,\dots,n,$ where  we assume that they are ordered so that $A_{l-1}< A_{l},\ l=1,\dots,n,$ with associated compact edges given by the intervals
  $\ga_l:=[(A_{l-1},B_{l-1}), (A_l,B_l)],$  $l=1,\dots,n,$ contained in the $L_l$ and associated with  the weights $\ka^l.$ The unbounded horizontal edge with left endpoint $(A_n,B_n)$ will be denoted by  $\ga_{n+1}.$  
For $l=n+1,$ we have $\ka_1^{n+1}:=0,\ka_2^{n+1}=1/B_{n}.$ Again,  we put $a_l:=\ka^l_2/{\ka^l_1},$ and  $a_{n+1}:=\infty$.

\smallskip
 
Because of \eqref{hf}, we have  to prove the following estimates:
\begin{eqnarray}\label{nec2}
&&p'\ge 2d^f+2;\\
&&p'\ge 2h^f_l+2\quad \mbox{for every $l$ such that $a_l>m_0.$}\label{nec3}
\end{eqnarray}
where, according to \eqref{hl2}, 
$$
h^f_l=\frac {1+m_0\ka^l_1-\ka^l_2}{\ka^l_1+\ka^l_2}.
$$

Let  first $\ga_l$ be any  non-horizontal edge of $\N(\phi^f)$ with $a_l>m_0,$ and consider the region
$$D^a_\ve:=\{y\in\RR^2:|y_1|\le\ve^{\ka^l_1}, |y_2|\le\ve^{\ka^l_2}\},\quad\ve>0,
$$ 
in the coordinates $y.$ In the original coordinates $x$, it corresponds to
$$
D_\ve:=\{x\in\RR^2:|x_1|\le\ve^{\ka^l_1}, |x_2-f(x_1)|\le\ve^{\ka^l_2}\}.
$$
Assume that $\ve$ is sufficiently small. Since 
$$
\phi^f(\ve^{\ka^l_1}y_1,\ve^{\ka^l_2}y_2)=\ve (\phi^f_{\ka^l}(y_1,y_2)+O(\ve^\delta)),
$$
for some $\de>0,$
where $\phi^f_{\ka^l}$ denotes the $\ka^l$-principal part of $\phi^f,$ 
 we have that $|\phi^f(y)|\le C\ve$ for every $y\in D^a_\ve,$ i.e.
\begin{equation}\label{nec4}
|\phi(x)|\le C \ve \quad \mbox {   for every   }x\in D_\ve.
\end{equation}
Moreover, for $x\in D_\ve,$ because $|f(x_1)|\lesssim |x_1|^{m_0}$ and $m_0\le a_l=\ka^l_2/\ka^l_1,$ we have 
$$
|x_2|\le \ve^{\ka^l_2}+|f(x_1)|\lesssim \ve^{\ka^l_2}+\ve^{m_0\ka^l_1}\lesssim \ve^{m_0\ka^l_1}.
$$

We may thus assume that  $D_\ve$ is contained in the box where $|x_1|\le2\ve^{\ka^l_1}, |x_2|\le2\ve^{m_0\ka^l_1}.$ Choose a Schwartz function  $\vp_\ve$ such that
$$
\widehat{\vp_\ve}(x_1,x_2,x_3)=\chi_0\Big(\frac {x_1}{\ve^{\ka^l_1}}\Big)\chi_0\Big(\frac {x_2}{\ve^{m_0\ka^l_1}}\Big)\chi_0\Big(\frac {x_3}{C\ve}\Big),
$$
where $\chi_0$ is again a smooth cut-off function supported in $[-2,2]$ identically $1$ on $[-1,1].$

Then by \eqref{nec4} we see that $\widehat{\vp_\ve}(x_1,x_2,\phi(x_1,x_2))\ge 1$ on $D_\ve,$ hence, if $\rho(0)\ne 0,$ then 
$$
\Big(\int_S |\widehat {\vp_\ve}|^2\, \rho\, d\si\Big)^{1/2}\ge C_1|D_\ve|^{1/2}=C_1\ve^{(\ka^l_1+\ka^l_2)/2},
$$
where $C_1>0$ is a positive constant.
Since $\|\vp_\ve\|_p\simeq \ve^{((1+m_0)\ka^l_1+1)/p'},$ we find that the restriction estimate \eqref{rest1}  can only hold if 
$$
p'\ge 2\frac{(1+m_0)\ka^l_1+1}{\ka^l_1+\ka^l_2}=2h^f_l+2.
$$

The case  $l=n+1,$ where $\ga_l$ is the horizontal edge, so that $h^f_l=B_n-1, $  requires a minor modification of this argument. Observe that, by Taylor expansion,  in this case $\phi^f$ can be written as 
\begin{equation}\label{nec5}
\phi^f\y=y_2^{B_{n}} h\y +\sum_{j=0}^{B_n-1} y_2^j g_j(y_1),
\end{equation}
where the functions $g_j$ are flat and $h$ is fractionally smooth and continuous at the origin. Choose $\de>0,$ and define here 
$$
D^a_\ve:=\{y\in\RR^2:|y_1|\le\ve^{\de}, |y_2|\le\ve^{\ka^l_2}\},\quad\ve>0.
$$ 
Then \eqref{nec5} shows that again $|\phi^f(y)|\le C\ve$ for every $y\in D^a_\ve,$  so that \eqref{nec4} holds true again. 
Moreover, for $x\in D_\ve,$ we now find that 
$$
|x_2|\le \ve^{\ka^l_2}+|f(x_1)|\lesssim \ve^{\ka^l_2}+\ve^{m_0\de}\lesssim \ve^{m_0\de}
$$
for $\de$ sufficiently small.  Choosing
$$
\widehat{\vp_\ve}(x_1,x_2,x_3)=\chi_0\Big(\frac {x_1}{\ve^{\de}}\Big)\chi_0\Big(\frac {x_2}{\ve^{m_0\de}}\Big)\chi_0\Big(\frac {x_3}{C\ve}\Big),
$$
arguing as before we find that here \eqref{rest1} implies that 
$$
p'\ge 2\frac{(1+m_0)\de+1}{\de+\ka^l_2} \quad \mbox{for every   }\  \de>0, 
$$
hence $p'\ge 2B_n=2h^f_{l_{n+1}}+2.$ This finishes the proof of \eqref{nec3}.

\smallskip
Notice finally  that the argument for the non-horizontal edges  still works if we replace the line $L_l$  by the line $L^f$ and the weight $\ka^l$ by the weight $\ka^f$ associated with that line. Since here $m_0\ka^f_1=\ka^f_2,$ this  leads to the condition \eqref{nec2}. 
\end{proof}

Proposition \ref{necs.1} also allows to give a short, but admittedly  indirect proof of Proposition \ref{rheight}, which will make use of Theorem  \ref{nonadarestrictanal}, too.
\medskip 

\noindent{\it Proof of Proposition \ref{rheight}.} Recall that we assume that the original coordinates $\x$ are linearly adapted to $\phi.$  

In oder to prove (a), assume furthermore  that the  coordinates $\x$ are not adapted to $\phi,$ and let $f(x_1)$ be any 
non-flat fractionally smooth, real  function  $f(x_1),$ with corresponding fractional shear, say  in $H^+.$ We have to show that 
\begin{equation}\label{nec6}
h^f(\phi)\le h^r(\phi). 
\end{equation}

\smallskip 

We begin with the special case where $\phi$ is analytic,  then Theorem \ref{nonadarestrictanal} shows that the restriction estimate \eqref{rest1} holds true  for $p=p_c,$ where $p_c'=2h^r(\phi)+2.$ Moreover, choosing $\rho$ so that $\rho(x^0)\ne 0,$ then Proposition \ref{necs.1} 
 implies that $p'\ge 2h^f(\phi)+2.$   Combining these estimates we obtain \eqref{nec6}.

\medskip
The case of a general smooth,  finite type $\phi$ can be reduced to the previous case. To this end,  denote by $\phi_N$ the Taylor polynomial of degree $N$  centered at the origin.  
It is not difficult to show that if $N$ is sufficiently large, then
$$
h^r(\phi)=h^r(\phi_N)\quad \mbox{and }\quad h^{f}(\phi_N)=h^f(\phi).
$$
Since \eqref{nec6} holds true for $\phi_N,$ we thus see that it holds true also for $\phi.$

\medskip

In order to prove (b), we assume that the coordinates $\x$ are adapted to $\phi,$ so that $d(\phi)=h(\phi).$  We have to prove that 
\begin{equation}\label{nec7}
\tilde h^r(\phi)= d(\phi).
\end{equation}
Let us first observe that Theorem \ref{adarestrict}   and Proposition \ref{necs.1}  imply, in a similar way as in the proof of (a), that $2h(\phi)+2\ge 2h^f(\phi)+2,$ hence $d(\phi)\ge h^f(\phi).$ We thus see that 
$$
\tilde h^r(\phi)\le d(\phi).
$$

On the other hand, when the  principal face $\pi(\phi)$ is compact, then we can choose a support line 
$$
 L=\{(t_1,t_2)\in \RR^2:\ka_1t_1+\ka_2 t_2=1\}
$$
to the Newton polyhedron of $\phi$ containing $\pi(\phi)$ and such that  $0<\ka_1\le \ka_2.$ We then put $f(x_1):=x_1^{m_0},$ where $m_0:=\ka_2/\ka_1.$ Then $d(\phi)=1/(\ka_1+\ka_2)=d^f\le h^f(\phi)\le \tilde h^r(\phi),$ and we obtain \eqref{nec7}.

\smallskip Assume finally that $\pi(\phi)$ is an unbounded horizontal half-line, with left endpoint $(A,B),$ where $A<B.$ We then choose $f_n(x_1):=x_1^n,\ n\in\bN.$ Then it is easy to see that for $n$ sufficiently large, the line $L^{f_n}$ will pass through  the  point $(A,B),$ and thus   $\lim_{n\to\infty} h^{f_n}(\phi)=B=d(\phi).$ Therefore, $\tilde h^r(\phi)\ge d(\phi),$ which shows that \eqref{nec7} is valid also in this case.
\qed

\setcounter{equation}{0}
\section{Appendix: Proof of Lemma \ref{doublesum}}\label{doublesumproof}

The basic idea of the proof becomes most transparent under the
additional assumption that also the vectors  $(\be^k_1,\be^k_2),
k=1,\dots,n,$ are pairwise linearly independent, i.e.,
\begin{equation}\label{8.12I}
\be^l_1\be^k_2-\be^k_1\be^l_2\ne 0, \qquad \mbox{ for all } l\ne k.
\end{equation}

We shall therefore begin with this case, and later indicate the
modifications  needed for  the general case.

\medskip

For   $y=(y_1,\dots, y_n)$  in an open neighborhood of $Q,$
Taylor's integral formula allows to  write
 $ H(y)=H(0)+\sum_{k=1}^n y_k H_k(y),
 $
 with $C^1$-functions $H_k$ whose $C^1$-norms on  $Q$ are controlled by the $C^2(Q)$-norm of $H.$  Similarly, putting $h_k(y_k):= H_k(0,\dots,0, y_k,0,\dots,0),$ we may decompose
 $H_k(y)=h_k(y_k)+\sum_{\{l:l\ne k\}} y_lH_{kl}(y),
 $
 with  continuous functions $H_{kl}$ whose $C(Q)$-norms  are controlled by the $C^2(Q)$-norm of $H.$ This allows to write
 $$
 H(y)=H(0)+\sum_k y_kh_k(y_k)+\sum_{l\ne k} y_ky_lH_{kl}(y).
 $$
 Accordingly, we shall decompose $F(t)=F_0(t)+\sum_k F_k(t) +\sum_{l\ne k} F_{kl}(t),$ where
 \begin{eqnarray*}
F_0(t)&:=&H(0)\sum_{m_1=0}^{M_1}\sum_{m_2=0}^{M_2}
2^{i(\al_1m_1+\al_2m_2)t}
\chi_Q\Big(2^{(\be^1_1 m_1+\be^1_2 m_2)}a_1,\dots, 2^{(\be^n_1 m_1+\be^n_2 m_2)}a_n\Big),\\
F_k(t)&:=&\sum_{m_1=0}^{M_1}\sum_{m_2=0}^{M_2}
2^{i(\al_1m_1+\al_2m_2)t} \,
(y_k h_k)(2^{(\be^k_1 m_1+\be^k_2 m_2)}a_k)\, \\
&& \hskip3cm \chi_Q\Big(2^{(\be^1_1 m_1+\be^1_2 m_2)}a_1,\dots, 2^{(\be^n_1 m_1+\be^n_2 m_2)}a_n\Big),\\
F_{kl}(t)&:=&\sum_{m_1=0}^{M_1}\sum_{m_2=0}^{M_2}
2^{i(\al_1m_1+\al_2m_2)t} \, (y_ky_l H_{kl}\,\chi_Q)\Big(2^{(\be^1_1
m_1+\be^1_2 m_2)}a_1,\dots, 2^{(\be^n_1 m_1+\be^n_2 m_2)}a_n\Big),
\end{eqnarray*}
It will therefore suffice to establish estimates of the form
\eqref{8.9I} for each of these functions $F_0, F_k$ and $F_{kl}.$ We
begin with $F_0.$

We may choose $r\in\NN^\times$ so that  every $\be^k_i$ can be
written as $\be^k_i=p^k_i/r,$ with $p^k_i\in\ZZ.$ Let us assume that
there is a least one $\be^k_2\ne 0$ (otherwise, we find some
$\be^k_1\ne 0,$ and may proceed with the roles of the indices $i=1$
and $i=2$ interchanged). We then put $p_2:=|\prod_{k: p^k_2\ne 0}
p^k_2|,$  whenever $p^k_2\ne 0,$ and $q_k:= (p^k_1 p_2)/p^k_2,$ so
that $q_k\in \ZZ.$ Observe next that we may write every $m_1\in\NN$
uniquely in the form $m_1=\al +j_1p_2,$ with $\al\in\{0, \dots,
p_2-1\}$ and $j_1\in \ZZ.$ This allows to decompose
$F_0(t)=\sum_{\al=0}^{p_2-1}F^\al_0(t),$ where $F^\al_0(t)$ is
defined like $F_0(t),$ only that the summation in $m_1$ is
restricted to those $m_1$ which are congruent to $\al$ modulo $p_2.$

Next, an easy computation shows that if $\be^k_2\ne 0,$ then
$$
\be^k_1 (\al +j_1p_2)+\be^k_2 m_2= \be^k_1\al +\be^k_2(m_2+q_k j_1).
$$
Therefore, if we write $R_k/a_k=(\sgn a_k ) 2^{b_k},$ then the
restriction imposed by  $\chi_Q$ on the  $k$' s coordinate leads to
the condition
$$
0\le  \be^k_1\al +\be^k_2(m_2+q_k j_1)\le b_k.
$$
This means that $m_2$ lies  in an ``interval'' of the form $\{e_k-
q_k j_1,\dots, d_k-q_k j_1\},$  for every $k$ such that $\be^k_2\ne
0$ (by an interval  we mean here the set of integer points within a
real  interval). We may therefore decompose the set of $j_1$'s over
which we are summing into a finite number  of (at most $(n!)^2$)
pairwise disjoint intervals $J_s$   such that for each given $s$
there are indices  $k_s, k'_s$ such that for $j_1\in I_s,$ $m_2$
will run through an interval of the form $\{e'_s-u_s j_1,\dots, d'_s
-v_s j_1\},$ where $e'_s:=e_{k_s}, u_s:=q_{k_s}$ and
$d'_s:=d_{k'_s}, v_s:=q_{k'_s}.$ We may thus reduce to considering,
for each fixed $s,$ the corresponding part $F^\al_s$ of $F^\al$
given by summation over the interval $I_s,$ i.e.,
\begin{eqnarray*}
F^\al_s(t) &:=&H(0)\sum_{\{j_1\in I_s: 0\le \al+j_1 p_2\le M_1\}}\sum_{m_2=e'_s-u_s j_1}^{d'_s -v_s j_1} 2^{i(\al\al_1+p_2\al_1j_1+\al_2m_2)t}\\
&&\hskip5cm
\times\prod_{\{k:\be^k_2=0\}}\chi_{[-R_k,R_k]}\Big(2^{\be^k_1(\al+j_1p_2)}a_k\Big)
\end{eqnarray*}
Evaluating the geometric sums in $m_2,$ this shows that $F^\al_s(t)$
is  the difference of  two terms, one arising from the lower  limit
$m_2=e'_s-u_s j_1,$ which is given by
\begin{eqnarray*}
F^\al_{s,-}(t) &=&H(0)\sum_{\{j_1\in I_s: 0\le \al+j_1 p_2\le M_1\}}
2^{i\al\al_1t}\frac {2^{i (\al_2 e'_s +j_1(\al_1p_2-\al_2u_s))t}-2^{ip_2\al_1j_1t}}{2^{i\al_2 t}-1}\\
&&\hskip
4cm\times\prod_{\{k:\be^k_2=0\}}\chi_{[-R_k,R_k]}\Big(2^{\be^k_1(\al+j_1p_2)}a_k\Big),
\end{eqnarray*}
and an analogous term arising from the upper limit $m_2=d'_s -v_s
j_1.$ But, by our assumptions,
$\al_1p_2-\al_2u_s=(\al_1\be^{k_s}_2-\al_2 \be^{k_s}_1)w_{k_s}\ne
0,$ where $w_{k_s}:=p_2/\be^{k_s}_2\in\NN,$ and the characteristic
functions of the intervals $[-R_k,R_k]$ again localize the summation
over the $j_1$'s to the summation over some interval, which shows
that we may estimate
$$
|F^\al_{s,-}(t)|\le \frac C{|2^{i\al_2 t}-1| |2^{i\al_1p_2
t}-1||2^{i(\al_1\be^{k_s}_2-\al_2 \be^{k_s}_1)w_{k_s}t}-1|}\le
\frac{C}{|\rho(t)|}.
$$
This establishes the desired estimate for $F_0(t).$

\medskip
We next turn to $F_k(t).$ Given $k$, let us assume again without
loss of generality that $\be^k_2\ne 0.$ Then we may  write
$m_1,m_2\in \ZZ$  in a unique way as
\begin{equation}\label{8.13I}
m_1=\al+j_1 p^k_2, \ m_2=j_2-j_1 p^k_1, \quad \mbox{ with }\quad
\al\in\{0, \dots, |p^k_2|\},
\end{equation}
with $j_1,j_2\in\ZZ.$  Observe that then
\begin{eqnarray*}
\be^l_1m_1+\be^l_2 m_2&=&\be^l_1 \al + j_2\be^l_2 + j_1(\be^l_1\be^k_2-\be^l_2\be^k_1) r,\\
\al_1m_1+\al_2 m_2&=&\al_1\al+ j_2\al_2
+j_1(\al_1\be^k_2-\al_2\be^k_1)r.
\end{eqnarray*}
In particular, $\be^k_1m_1+\be^k_2 m_2=\be^k_1 \al + j_2\be^k_2$
does not depend on $j_1.$  Moreover,  for given $\al$ and $j_2,$ the
localizations given by the conditions $|2^{(\be^l_1 m_1+\be^l_2
m_2)}a_l|\le R_l, \, l\ne k,$ reduce the summation over $j_1$ to the
summation over an interval $I(\al,j_2),$ and summing a geometric sum
with respect to $j_1,$ we thus see that
$$
|F_k(t)|\le \frac {C'}{|2^{i(\al_1\be^k_2-\al_2\be^k_1)rt}-1|}
\sum_{\al=0}^{|p^k_2|}\sum_{\{j_2:|2^{\be^k_1 \al +
j_2\be^k_2}a_k|\le R_k } |2^{\be^k_1 \al + j_2\be^k_2}a_k|\le
\frac{C R_k}{|\rho(t)|}.
$$

\medskip
Consider finally $F_{kl}(t),$ for $k\ne l.$ We may simply estimate
$$
|F_{kl}(t)|\le C \sum_{(m_1,m_2)\in J_{k,l}} |2^{(\be^k_1
m_1+\be^k_2 m_2)}a_k|\,|2^{(\be^l_1 m_1+\be^l_2 m_2)}a_l|,
$$
where $J_{kl}$ is the set of all $(m_1,m_2)\in\NN^2$ satisfying
$|2^{(\be^k_1 m_1+\be^k_2 m_2)}a_k|\le R_k$ and $|2^{(\be^l_1
m_1+\be^l_2 m_2)}a_l|$  \hfill
$\le R_l.$ By comparing with an integral and
changing variables in the integral (recall  that  by our assumption \eqref{8.12I}  the matrix
$\left(
\begin{array}{ccc}
  \be^k_1   & \be^k_2  \\
 \be^l_1   & \be^l_2  \\
\end{array}
\right)$ is non-degenerate)  this leads to the estimate
\begin{eqnarray*}
|F_{kl}(t)|&\le& C'' \iint_{ I_{k,l}}|2^{(\be^k_1 s_1+\be^k_2 s_2)}a_k|\,|2^{(\be^l_1 s_1+\be^l_2 s_2)}a_l| \, ds_1 ds_2\\
&\le&
C'\int_{-\infty}^{\log_2(R_l/|a_l|)}\int_{-\infty}^{\log_2(R_k/|a_k|)}
|2^{x_1} a_k| |2^{x_2} a_l|\, dx_1 dx_2\le C R_k R_l,
\end{eqnarray*}
where  $I_{kl}$ denotes  the set of all $(s_1,s_2)\in\RR_+^2$
satisfying $|2^{(\be^k_1 s_1+\be^k_2 s_2)}a_k|\le R_k$ and
$|2^{(\be^l_1 s_1+\be^l_2 s_2)}a_l|\le R_l.$
\medskip

This concludes the proof of the lemma under our additional
hypotheses \eqref{8.12I}.

\bigskip
Let us finally indicate how to remove the assumptions \eqref{8.12I}.
To this end, let us write $\be^j:=(\be^j_1,\be^j_2).$  In the
general case, we may decompose the index set $\{1,\dots, n\}$ into
pairwise disjoint subset $I_1,\dots, I_h$ such that the following
hold true: There are non-trivial vectors $\ga^k=(\ga^k_1,\ga^k_2),
k=1,\dots, h,$ in $ \QQ^2$  and rational numbers $r_j\ne 0,
j=1,\dots,n,$ such that
 \begin{itemize}
\item[(a)] If $j\in I_k,$ then $\be^j=r_j\ga^k;$
\item[(b)] For  $k\ne l,$  the vectors $\ga^k$ and $\ga^l$  are linearly independent.
 \end{itemize}
 Let us accordingly define the vectors $ Y_k:= (y_j)_{j\in I_k}\in \RR^{I_k}, k=1,\dots, h.$ We may assume  (possibly after a permutation of coordinates) that $y=(Y_1,\dots,Y_h).$ Following the first step of the previous proof, we then decompose
 $H(y)=H(0) + \sum_{k=1}^h \trans Y_k\cdot H_k(y),$ where now $H_k$ maps into $\RR^{I_k}.$  Next, we put
 $$
 h_k(Y_k):= H_k(0, \dots,0, Y_k,0,\dots,0)\in \RR^{I_k},
 $$
 and apply Taylor's formula in order to write
 $$
 H(y)=H(0) + \sum_{k=1}^h \trans Y_k\cdot h_k(Y_k)+\sum_{k\ne l} \trans Y_k \cdot H_{kl}(y) \cdot Y_l,
 $$
 where here $H_{kl}$ is a matrix-valued function. Correspondingly, we define the function
 $F_0(t), F_k(t)$ and $F_{kl}(t)$ as before, only with $y_k h_k(y_k)$ replaced
  by $\trans Y_k\cdot h_k(Y_k)$ and $y_ky_lH_{kl}(y)$ by $\trans Y_k \cdot H_{kl}(y) \cdot Y_l,$ respectively.

 The discussion of $F_0(t)$  remains unchanged, and the same applies essentially also to the
 discussion of $F_{kl}(t),$ because of property (b). Finally,
 for the estimation of $F_k(t),$ notice that for a given, fixed $k,$ if $j\in I_k,$ then
 by (a) we see that the arguments  at which $\trans Y_k \cdot h_{k}$ is evaluated are all of the
 form $2^{r_j(\ga^k_1m_1+\ga^k_2 m_2)}a_j.$  Therefore, in the coordinates given by
   $\al,j_1,j_2$  from \eqref{8.13I}, they  all  will not depend on $j_1.$
   We may therefore proceed in the estimation of $F_k(t)$ essentially as before,
   which concludes the proof of Lemma \ref{doublesum} also in the general case.

\end{document}